\documentclass[11pt]{article}%
\usepackage{graphicx}
\usepackage{amsmath}
\usepackage{amsfonts}
\usepackage{amssymb}%
\providecommand{\U}[1]{\protect\rule{.1in}{.1in}}
\newtheorem{theorem}{Theorem}[section]
\newtheorem{acknowledgement}[theorem]{Acknowledgement}
\newtheorem{corollary}[theorem]{Corollary}

\newtheorem{definition}[theorem]{Definition}

\newtheorem{lemma}[theorem]{Lemma}

\newtheorem{proposition}[theorem]{Proposition}
\newtheorem{remark}[theorem]{Remark}

\newenvironment{proof}[1][Proof]{\textbf{#1.} }{\hfill\rule{0.5em}{0.5em}}
{\catcode`\@=11\global\let\AddToReset=\@addtoreset
\AddToReset{equation}{section}

\AddToReset{theorem}{section}

\begin{document}

\title{$L^{\infty}$ estimates and uniqueness results for nonlinear parabolic
equations with gradient absorption terms}
\author{Marie Fran\c{c}oise BIDAUT-VERON\thanks{Laboratoire de Math\'{e}matiques et
Physique Th\'{e}orique, CNRS UMR 6083, Facult\'{e} des Sciences, 37200 Tours
France. E-mail address:veronmf@univ-tours.fr}
\and Nguyen Anh DAO\thanks{Laboratoire de Math\'{e}matiques et Physique
Th\'{e}orique, CNRS UMR 6083, Facult\'{e} des Sciences, 37200 Tours France.
E-mail address: Anh.Nguyen@lmpt.univ-tours.fr}}
\date{}
\maketitle

\begin{abstract}
We study the nonnegative solutions of the viscous Hamilton-Jacobi
problem
\[
\left\{
\begin{array}
[c]{c}%
u_{t}-\nu\Delta u+|\nabla u|^{q}=0,\\
u(0)=u_{0},
\end{array}
\right.
\]
in $Q_{\Omega,T}=\Omega\times\left(  0,T\right)  ,$ where $q>1,\nu\geqq
0,T\in\left(  0,\infty\right]  ,$ and $\Omega=\mathbb{R}^{N}$ or $\Omega$ is a
smooth bounded domain, and $u_{0}\in L^{r}(\Omega),r\geqq1,$ or $u_{0}%
\in\mathcal{M}_{b}(\Omega).$ We show $L^{\infty}$ decay estimates, valid for
\textit{any weak solution}, \textit{without any conditions a}s $\left\vert
x\right\vert \rightarrow\infty,$ and \textit{without uniqueness assumptions}.
As a consequence we obtain new uniqueness results, when $u_{0}\in
\mathcal{M}_{b}(\Omega)$ and $q<(N+2)/(N+1),$ or $u_{0}\in L^{r}(\Omega)$ and
$q<(N+2r)/(N+r).$ We also extend some decay properties to quasilinear
equations of the model type
\[
u_{t}-\Delta_{p}u+\left\vert u\right\vert ^{\lambda-1}u|\nabla u|^{q}=0
\]
where $p>1,\lambda\geqq0,$ and $u$ is a signed solution. \bigskip\ \bigskip

\textbf{Keywords }Viscous Hamilton-Jacobi equation; quasilinear parabolic
equations with gradient terms; regularity; decay estimates; regularizing
effects; uniqueness results. \bigskip\ \bigskip

\textbf{A.M.S. Subject Classification }35K15, 35K55, 35B33, 35B65, 35D30

\end{abstract}

\section{ Introduction\label{sec1}}

In this article we study a class of  heat equations involving a nonlinear gradient absorption
term. We are mainly concerned by the nonnegative solutions of the 
viscous parabolic Hamilton-Jacobi equation%
\begin{equation}
u_{t}-\nu\Delta u+|\nabla u|^{q}=0 \label{un}%
\end{equation}
in $Q_{\Omega,T}=\Omega\times\left(  0,T\right)  ,$ $T\leqq\infty,$ where
$q>1,\nu\geqq0,$ and $\Omega=\mathbb{R}^{N},$ or $\Omega$ is a smooth bounded
domain of $\mathbb{R}^{N}$.\medskip

We study the Cauchy problem in $\mathbb{R}^{N}$ and the Cauchy-Dirichlet
problem when $\Omega$ is bounded, with initial data $u(.,0)=u_{0}%
\geqq0,$ where $u_{0}\in L^{r}(\Omega),r\geqq1,$ or $u_{0}$ is a bounded Radon
measure on $\Omega.$ \medskip\ 

We also consider the (signed) solutions of quasilinear equations of the type
\begin{equation}
u_{t}-\nu\Delta_{p}u+\left\vert u\right\vert ^{\lambda-1}u|\nabla u|^{q}=0
\label{fru}%
\end{equation}
where $p>1$ and $\Delta_{p}$ is the $p$-Laplacian, and more generally
\begin{equation}
u_{t}-\operatorname{div}(\mathrm{A}(x,t,u,\nabla u))+g(x,u,\nabla u)=0
\label{eqge}%
\end{equation}
with natural growth conditions on the function $\mathrm{A,}$ and nonnegativity
conditions
\begin{equation}
\mathrm{A}(x,t,u,\eta).\eta\geqq\nu\left\vert \eta\right\vert ^{p},\qquad
g(x,u,\eta)u\geqq\gamma\left\vert u\right\vert ^{\lambda+1}|\nabla
u|^{q}\qquad\gamma\geqq0,\nu\geqq0,\lambda\geqq0, \label{alp}%
\end{equation}
without monotonicity assumption.\medskip

In the sequel we give some \textit{decay estimates}, \textit{under very few
assumptions on the solutions}. Then from Moser's technique, we deduce
\textit{regularizing effects} : $L^{\infty}$ estimates, in terms of $u_{0},$
and universal estimates when $\Omega$ is bounded. We show that \textit{two
types} of regularizing effect can occur: the first one is due to the gradient
term $|\nabla u|^{q}$ (when $\gamma>0$)$,$ the second one is due to the
operator itself (when $\nu>0$)$.$ \medskip

A part of these estimates are well known for equation (\ref{un}) when the
solutions can be approximated by smooth solutions, or satisfy growth
conditions as $\left\vert x\right\vert \rightarrow$ $\infty$ when
$\Omega=\mathbb{R}^{N}$, for example semi-group solutions. Our approach is
different, and our results are valid for \textit{all the solutions} of the
equation \textit{in a weak sense}: in the sense of distributions for equation
(\ref{un}), in the renormalized sense for equation (\ref{eqge}). And we make
\textit{no assumption of uniqueness}. In the case of equation (\ref{un}) in
$\mathbb{R}^{N}$, we require \textit{no condition as }$\left\vert x\right\vert
\rightarrow\infty$\textit{, all our assumptions are local. }\medskip

As a consequence we deduce \textit{new uniqueness results} for equation
(\ref{un}) in $\mathbb{R}^{N}$ or in a bounded domain $\Omega.$

\section{Main results}

\noindent We denote by $\mathcal{M}_{b}(\Omega)$ the set of bounded Radon
measures in $\Omega,$ and $\mathcal{M}_{b}^{+}(\Omega)$ the cone of
nonnegative ones.

\noindent We set $Q_{\Omega,s,\tau}=\Omega\times\left(  s,\tau\right)  ,$ for
any $0\leqq s<\tau\leqq\infty,$ thus $Q_{\Omega,T}=Q_{\Omega,0,T}.$

\noindent As usual, for any $\theta\geqq1$ we note by $\theta^{\prime}%
=\theta/(\theta-1)$ the conjugate of $\theta.$\bigskip

In \textbf{Section \ref{sec2}, w}e give some key tools for obtaining
regularizing properties. The main one is an iteration property based of
Moser's method, inspired by \cite{Ve}:

\begin{lemma}
\label{prod} Let $m>1,$ $\theta>1$ and $\lambda\in\mathbb{R}$ and $C_{0}>0$.
Let $v\in C(\left[  0,T\right)  ;L_{loc}^{1}(\Omega))$ be nonnegative, and
$v_{0}=v(x,0)\in L^{r}(\Omega)$ for some $r\geqq1$ such that
\begin{equation}
r>\theta^{\prime}(1-m-\lambda). \label{ele}%
\end{equation}
If $r>1$ we assume that for any $0\leqq s<t<T$ and any $\alpha\geqq r-1,$
there holds
\begin{equation}
\frac{1}{\alpha+1}\int_{\Omega}v^{\alpha+1}(.,t)dx+\frac{C_{0}}{\beta^{m}}%
\int_{s}^{t}(\int_{\Omega}v^{\beta m\theta}(.,\tau)dx)^{\frac{1}{\theta}}%
d\tau\leqq\frac{1}{\alpha+1}\int_{\Omega}v^{\alpha+1}(.,s)dx, \label{mac}%
\end{equation}
where $\beta=\beta(\alpha)=1+\left(  \alpha+\lambda\right)  /m$, and the
right-hand side \textbf{can be infinite}.$\medskip$

\noindent If $r=1$ we make one of the two following assumptions:\medskip

(H$_{1}$) (\ref{mac}) holds for any $\alpha\geqq0$,\medskip

(H$_{2}$) $\int_{\Omega}v(.,t)dx\leqq\int_{\Omega}v_{0}dx$ for any
$t\in\left(  0,T\right)  $, and $v_{0}\in L^{\rho}(\Omega)$ for some $\rho>1$
such that $\rho\theta^{\prime}(1-m-\lambda)<1$ and (\ref{mac}) holds for any
$\alpha\geqq\rho-1.$\medskip

\noindent\ Then there exists $C>0,$ depending on $N,m,r,\lambda,C_{0},$ and
possibly $\rho,$ such that for any $t\in(0,T),$
\begin{equation}
\Vert v(.,t)\Vert_{L^{\infty}(\Omega)}\leqq Ct^{-\sigma_{r,m,\lambda,\theta}%
}\Vert v_{0}\Vert_{L^{r}(\Omega)}^{\varpi_{r,m,\lambda,\theta}}, \label{trap}%
\end{equation}
where
\begin{equation}
\sigma_{r,m,\lambda,\theta}=\frac{1}{\frac{r}{\theta^{\prime}}+\lambda
+m-1}=\frac{\theta^{\prime}}{r}\varpi_{r,m,\lambda,\theta}. \label{jam}%
\end{equation}

\end{lemma}

This Lemma allows to obtain $L^{\infty}$ estimates for the solutions of
equation (\ref{un}), when $q\leqq N,$ or $2\leqq N$ , and for equation
(\ref{fru}) when $p\leqq N$ . In the other cases the $L^{\infty}$ estimates
follow from the Gagliardo-Nirenberg inequality, see Lemma \ref{gani}. Moreover
we deduce \textit{universal} $L^{\infty}$ estimates when $\Omega$ is bounded,
see Lemma \ref{univ}.\bigskip

In \textbf{Section \ref{Rn}} we study the Cauchy Hamilton-Jacobi problem in
$\mathbb{R}^{N}:$
\begin{equation}
\left\{
\begin{array}
[c]{l}%
u_{t}-\nu\Delta u+|\nabla u|^{q}=0,\quad\text{in}\hspace{0.05in}%
Q_{\mathbb{R}^{N},T},\\
u(x,0)=u_{0}\geqq0\quad\text{in}\hspace{0.05in}\mathbb{R}^{N},
\end{array}
\right.  \label{cau}%
\end{equation}
This equation is the objet of a huge literature, see \cite{AmBa},
\cite{BeLa99}, \cite{BeBALa}, \cite{BASoWe}, \cite{SoZh}, and the references
therein, and also \cite{BeBALa}, \cite{BeLaSc}, \cite{GaLa}. \medskip

The first studies concern smooth initial data $u_{0}\in C_{b}^{2}\left(
\mathbb{R}^{N}\right)  .$ From \cite{AmBa}, (\ref{cau}) has a unique global
solution $u\in C^{2,1}(\mathbb{R}^{N}\times\left[  0,\infty\right)  )$, and
$u$ satisfies decay properties:
\begin{align*}
\left\Vert u(.,t)\right\Vert _{L^{\infty}(\mathbb{R}^{N})}  &  \leqq\left\Vert
u_{0}\right\Vert _{L^{\infty}(\mathbb{R}^{N})},\\
\left\Vert \nabla u(.,t)\right\Vert _{L^{\infty}(\mathbb{R}^{N})}  &
\leqq\left\Vert \nabla u_{0}\right\Vert _{L^{\infty}(\mathbb{R}^{N})}.
\end{align*}
Estimates of the gradient have been obtained \textit{for this solution}, by
using the Bersnstein technique, which consists in computing the equation
satisfied by $|\nabla u|^{2}:$ first from \cite{Li},%
\[
\left\Vert \nabla u(.,t)\right\Vert _{L^{\infty}(\mathbb{R}^{N})}^{q}\leqq
t^{-1}\left\Vert u_{0}\right\Vert _{L^{\infty}(\mathbb{R}^{N})},
\]
then from \cite{BeLa99}, when $\nu>0,$%
\begin{equation}
\Vert\nabla(u^{\frac{1}{q^{\prime}}})(.,t)\Vert_{L^{\infty}(\mathbb{R}^{N}%
)}\leqq C(q,\nu)t^{-\frac{1}{2}}\Vert u_{0}\Vert_{L^{\infty}(\mathbb{R}^{N}%
)}^{\frac{1}{q^{\prime}}},, \label{cor}%
\end{equation}%
\begin{equation}
\left\Vert \nabla(u^{\frac{1}{q^{\prime}}})(.,t)\right\Vert _{L^{\infty
}(\mathbb{R}^{N})}\leqq\frac{(q-1)^{\frac{1}{q^{\prime}}}}{q}t^{-\frac{1}{q}%
},\quad\text{that is}\quad\left\vert \nabla u(.,t)\right\vert ^{q}\leqq
\frac{t^{-1}u(.,t)}{q-1},\quad a.e.\text{in }\mathbb{R}^{N}. \label{bela}%
\end{equation}
If one only assumes $u_{0}\in C_{b}\left(  \mathbb{R}^{N}\right)  $, then
(\ref{cau}) still has a unique solution $u$ such that $u\in C^{2,1}%
(Q_{\mathbb{R}^{N},\infty})$ and $u\in C(\mathbb{R}^{N}\times\left[
0,\infty\right)  \cap L^{\infty}(\mathbb{R}^{N}\times\left(  0,\infty\right)
)$ see \cite{GiGuKe}, and estimates (\ref{cor}) and (\ref{bela}) are still
valid, from \cite{BeBALa}.\medskip\ 

In case of rough initial data $u_{0}\in\mathcal{M}_{b}^{+}(\mathbb{R}^{N})$ or
$u\in L^{r}(\mathbb{R}^{N}),$ $r\geqq1$, assuming $\nu>0,$ the solutions have
been searched in an integral form%
\begin{equation}
u(.,t)=e^{t\Delta}u_{0}(.)-\nu\int_{0}^{t}e^{(t-s)\Delta}|\nabla
u(.,s)|^{q}ds, \label{itf}%
\end{equation}
involving the semi-group of heat equation $e^{t\Delta}.$ Existence results
hold in corresponding classes of solutions, involving integral conditions on
the gradient in space and time, of \textit{global type}:\medskip

$\bullet$ If $u_{0}\in\mathcal{M}_{b}^{+}(\mathbb{R}^{N})$ and
$1<q<(N+2)/(N+1),$ the existence of a solution $u\in C^{2,1}(Q_{\mathbb{R}%
^{N},\infty})$ is proved in \cite{BeLa99} by approximation, and independently
in \cite{BASoWe}, from the Banach fixed point theorem. \medskip\ 

$\bullet$ If $u_{0}\in L^{r}(\mathbb{R}^{N}),$ $r\geqq1,$ existence holds for
any $q\leqq2$ from \cite{BASoWe}. When $q>2,$ it is required that $u_{0}$ is a
limit of a monotone sequence of continuous functions, and existence is not
known in the general case.\bigskip

In those classes, decay properties and a regularizing effect follow directly
from the semigroup $e^{t\Delta},$ since $u(.,t)\leqq e^{t\Delta}u_{0}.$ Our
first main results shows that \textit{decay properties and }$L^{\infty}%
$\textit{ estimates} are valid \textit{for any weak solution}, for any
$\nu\geqq0,$ \textit{without any condition as }$\left\vert x\right\vert
\rightarrow\infty$:

\begin{theorem}
\label{main}Let $u\in L_{loc}^{1}(Q_{\mathbb{R}^{N},T}),$ with $|\nabla u|\in
L_{loc}^{q}(Q_{\mathbb{R}^{N},T}),$ be any nonnegative solution of equation
(\ref{un}) in$\hspace{0.05in}\mathcal{D}^{\prime}(Q_{\mathbb{R}^{N},T}%
).$\medskip

(i) Let $u_{0}\in L^{r}(\mathbb{R}^{N}),r\geqq1.$ Assume that $u\in C(\left[
0,T\right)  ;L_{loc}^{r}(\mathbb{R}^{N}))$ and $u(.,0)=u_{0}$. Then $u\in
C(\left[  0,T\right)  ;L^{r}(\mathbb{R}^{N}));$ and for any $t\in(0,T),$
$u(.,t)\in L^{\infty}(\mathbb{R}^{N})$ and%
\begin{equation}
\Vert u(.,t)\Vert_{L^{r}(\mathbb{R}^{N})}\leqq\Vert u_{0}\Vert_{L^{r}%
(\mathbb{R}^{N})}, \label{pas}%
\end{equation}%
\begin{equation}
\Vert u(.,t)\Vert_{L^{\infty}(\mathbb{R}^{N})}\leqq\left\{
\begin{array}
[c]{ccc}%
Ct^{-\sigma_{r,q,N}}\Vert u_{0}\Vert_{L^{r}(\mathbb{R}^{N})}^{\varpi_{r,q,N}%
}, & C=C(N,q,r), & \text{if }q\neq N,\\
C_{\varepsilon}t^{-(1+\varepsilon)\sigma_{r,N,N}}\Vert u_{0}\Vert
_{L^{r}(\mathbb{R}^{N})}^{(1+\varepsilon)\varpi_{r,q,N}}, & \forall
\varepsilon>0,\quad C_{\varepsilon}=C(N,q,r,\varepsilon), & \text{if }q=N,
\end{array}
\right.  \label{uinf}%
\end{equation}
where
\begin{equation}
\sigma_{r,q,N}=\frac{1}{\frac{rq}{N}+q-1}=\frac{N}{rq}\varpi_{r,q,N}.
\label{sigr}%
\end{equation}
And if $\nu>0,$ then%
\begin{equation}
\Vert u(.,t)\Vert_{L^{\infty}(\mathbb{R}^{N})}\leqq\left\{
\begin{array}
[c]{ccc}%
Ct^{-\frac{N}{2r}}\Vert u_{0}\Vert_{L^{r}(\mathbb{R}^{N})}, & C=C(N,r,\nu), &
\text{if }N\neq2,\\
C_{\varepsilon}t^{-\frac{1+\varepsilon}{r}}\Vert u_{0}\Vert_{L^{r}%
(\mathbb{R}^{N})}, & \forall\varepsilon>0,\quad C_{\varepsilon}=C(N,r,\nu
,\varepsilon), & \text{if }N=2.
\end{array}
\right.  \label{use}%
\end{equation}

(ii) Let $u_{0}\in\mathcal{M}_{b}^{+}(\mathbb{R}^{N})$. Assume that $u(.,t)$
converges weakly $^{\ast}$ to $u_{0}$ as $t\rightarrow0.$ Then $u\in
C((0,T);L^{1}(\mathbb{R}^{N})),$ and for any $t\in(0,T),$ the conclusions
above with $r=1$ are still valid with $\Vert u_{0}\Vert_{L^{1}(\mathbb{R}%
^{N})}$ replaced by $%
{\displaystyle\int_{\mathbb{R}^{N}}}
du_{0}.$
\end{theorem}

Note that estimates (\ref{pas}) are not valid for \textit{any weak subsolution
}of the heat equation. Here we prove that the result of (\ref{pas}) is
\textit{essentially} \textit{due to the gradient term} $\left\vert \nabla
u\right\vert ^{q},$ which has a main regularizing effect on the equation. And
then a second regularizing effect holds, due to the Laplacian, when $\nu
>0.$\bigskip

For any $q\leq2,$ we deduce estimates of the gradient, obtained from
(\ref{cor}). As a consequence we deduce new uniqueness results, where the
assumptions are only of \textit{local type}:

\begin{theorem}
\label{uni1}(i) Let $1<q<(N+2)/(N+1),$ and $u_{0}\in\mathcal{M}_{b}%
^{+}(\mathbb{R}^{N}).$ Then there exists a unique nonnegative function $u\in
L_{loc}^{1}(Q_{\mathbb{R}^{N},T}),$ such that $|\nabla u|\in L_{loc}%
^{q}(Q_{\mathbb{R}^{N},T}),$ solution of equation (\ref{un}) in $\mathcal{D}%
^{\prime}(Q_{\mathbb{R}^{N},T})$ such that
\[
\lim_{t\rightarrow0}\int_{\mathbb{R}^{N}}u(.,t)\psi dx=\int_{\mathbb{R}^{N}%
}\psi du_{0},\qquad\forall\psi\in C_{c}(\mathbb{R}^{N}).
\]

(ii) Let $u_{0}\in L^{r}(\mathbb{R}^{N}),$ $r\geqq1$ and $1<q<(N+2r)/(N+r).$
Then there exists a unique nonnegative solution $u$ as above, such that $u\in
C\left(  \left[  0,T\right)  ;L_{loc}^{r}(\mathbb{R}^{N})\right)  $ and
$u(.,0)=u_{0}.$\bigskip
\end{theorem}

This improves the former uniqueness results of \cite{BeLa99} and \cite[Theorem
4.1]{BASoWe}, given in classes of semigroup solutions, satisfying conditions
up to $t=0$ for the gradient: $|\nabla u|^{q}\in L_{loc}^{1}(\left[
0,T\right)  ;L^{1}\left(  \mathbb{R}^{N}\right)  )$ in case (i), and $|\nabla
u|^{q}\in L_{loc}^{1}(\left[  0,T\right)  ;L^{r}\left(  \mathbb{R}^{N}\right)
)$ in case (ii).\bigskip

We also find again in a shorter way the existence result of \cite[Theorem
4.1]{BASoWe}, see Proposition \ref{pou}. Finally we improve the estimate
(\ref{pas}) when $q<(N+2r)/(N+r),$ see Theorem \ref{esti}. \bigskip

In \textbf{Section \ref{omega}} we study the Cauchy-Dirichlet problem in a
bounded domain $\Omega$:
\begin{equation}
\left\{
\begin{array}
[c]{l}%
u_{t}-\nu\Delta u+|\nabla u|^{q}=0,\quad\text{in}\hspace{0.05in}Q_{\Omega
,T},\\
u=0,\quad\text{on}\hspace{0.05in}\partial\Omega\times(0,T),\\
u(x,0)=u_{0}\geqq0,
\end{array}
\right.  \label{1.1}%
\end{equation}
Here also the problem is the object of many works, such as \cite{CLS},
\cite{BeDa}, \cite{So}, \cite{BeDaLa}, \cite{Por}.\medskip

If $u_{0}\in C_{0}^{1}\left(  \overline{\Omega}\right)  $, from \cite{CLS},
(\ref{1.1}) admits a unique nonnegative solution $u\in C^{2,1}\left(
\Omega\times(0,\infty)\right)  \cap$ $C\left(  \overline{\Omega}\times\left[
0,\infty\right)  \right)  ,$ such that $|\nabla u|\in C\left(  \overline
{\Omega}\times\left[  0,\infty\right)  \right)  .$ Universal a priori
estimates hold: there exist $C=C(N,q,\Omega)>0$ and a function $D\in
C((0,\infty)$ such that
\begin{equation}
u(.,t)\leqq C(1+t^{-\frac{1}{q-1}})d(x,\partial\Omega),\qquad|\nabla
u(.,t)|\leqq D(t), \label{gd}%
\end{equation}
see \cite[Remark 2.8]{BiDao1}. The estimate on $u$ is based on the
construction of supersolutions, and the estimate of the gradient is deduced
from the first one by the Bernstein technique.\medskip

In case of rough initial data, a notion of mild solutions has been introduced
by \cite{BeDa} (see definition \ref{mildom}). Such solutions satisfy $|\nabla
u|^{q}\in L_{loc}^{1}(\left[  0,T\right)  ;L^{1}\left(  \Omega\right)
).$\medskip

$\bullet$ If $u_{0}\in\mathcal{M}_{b}^{+}(\Omega)$ and $1<q<(N+2)/(N+1),$
there is a unique nonnegative mild solution, see \cite{BeDa}, \cite{Al}. If
$u_{0}\in L^{1}(\Omega),$ and $1<q\leqq2,$ there exists at least a solution,
such that $u\in C\left(  \left[  0,T\right)  ;L^{1}(\Omega\right)  )$.\medskip

$\bullet$ If $1<q<(N+2r)/(N+r)$ uniqueness holds in the class of mild
solutions such that $u\in C\left(  \left[  0,T\right)  ;L^{r}(\Omega)\right)
\cap$ $L_{loc}^{q}\left(  \left[  0,T\right)  ;W^{qr}\left(  \Omega\right)
\right)  .$\medskip

Next we give decay properties and regularizing effects valid for \textit{any
weak solution} of the problem, in particular the universal estimate
\[
\Vert u(.,t)\Vert_{L^{\infty}(\Omega)}\leqq Ct^{-\frac{1}{q-1}}\quad\text{in
}\left(  0,T\right)  ,
\]
where $C=C(N,q),$ see Theorem \ref{deco}. As above we deduce uniqueness results:

\begin{theorem}
\label{uni2}Assume that $\Omega$ is bounded.\medskip

\noindent(i) Let $1<q<(N+2)/(N+1),$ and $u_{0}\in\mathcal{M}_{b}^{+}(\Omega).$
Then there exists a unique nonnegative function $u\in C((0,T);L^{1}\left(
\Omega\right)  )\cap L_{loc}^{q}((0,T);W_{0}^{1,q}\left(  \Omega\right)  ),$
solution of equation (\ref{un}) in $\mathcal{D}^{\prime}(Q_{\Omega,T}),$ such
that
\[
\lim_{t\rightarrow0}\int_{\Omega}u(.,t)\psi dx=\int_{\Omega}\psi du_{0}%
,\qquad\forall\psi\in C_{b}(\Omega).
\]

\noindent(ii) Let $u_{0}\in L^{r}(\Omega),$ $r\geqq1,u_{0}\geqq0,$ and
$1<q<(N+2r)/(N+r).$ Then there exists a unique nonnegative solution $u$ as
above, such that $u\in C\left(  \left[  0,T\right)  ;L^{r}(\Omega)\right)  $
and $u(.,0)=u_{0}.$
\end{theorem}

This improves the results of \cite{BeDa}, which required assumptions up to
$t=0$ for the gradient: $|\nabla u|^{q}\in L_{loc}^{1}(\left[  0,T\right)
;L^{1}\left(  \Omega\right)  )$ in case (i), $|\nabla u|\in L_{loc}%
^{q}(\left[  0,T\right)  ;L^{r}\left(  \Omega\right)  )$ in case (ii).\medskip

Finally we show the existence of weak solutions for any $u_{0}\in L^{r}%
(\Omega),$ $r\geqq1,$ such that $u\in C\left(  \left[  0,T\right)
;L^{r}(\Omega\right)  ),$ see Proposition \ref{plu}.\bigskip

In \textbf{Section \ref{quas}} we extend some results of Section \ref{omega}
to the case of the quasilinear equations (\ref{eqge}), with initial data
$u_{0}\in L^{r}\left(  \Omega\right)  $ or $u_{0}\in\mathcal{M}_{b}\left(
\Omega\right)  $, and $u$ may be a signed solution. In the case of equation%
\[
u_{t}-\Delta_{p}u=0,
\]
several local or global $L^{\infty}$ estimates and Harnack properties have
been obtained in the last decades, see for example \cite{Ve}, \cite{DiBeHe1},
\cite{DiBeHe2}, \cite{Ju}, and \cite{DiBeGV}, \cite{BonIaVa} and references
therein. Regularizing properties for equation (\ref{fru}) are given in
\cite{Por} in a Hilbertian context in case $g=0$ or $p=2$.\medskip

Here we combine our iteration method of Section \ref{sec2} with a notion of
\textit{renormalized solution}, developped by many authors \cite{BlMu},
\cite{Po},\cite{Pr}, well adapted to rough initial data. We do not require
that $u(.,t)\in L^{2}(\Omega),$ but we only assume that the truncates
$T_{k}(u)$ of $u$ by $\pm k$ $(k>0)$ lie in $L^{p}((0,T);W^{1,p}(\Omega))$. We
prove decay and $L^{\infty}$ estimates of the following type: if $\gamma>0$,
for any $r\geqq1,$ $p>1$ and (for simplicity) $q\neq N,$ then%
\[
\Vert u(.,t)\Vert_{L^{\infty}(\Omega)}\leqq Ct^{-\sigma}\Vert u_{0}%
\Vert_{L^{r}(\Omega)}^{\varpi},\qquad\sigma=\frac{1}{\frac{rq}{N}+\lambda
+q-1}=\frac{N}{rq}\varpi,
\]
If $\nu>0,$ then for any $r\geqq1,$ and $p\neq N$ such that $p>2N/(N+2),$
\[
\Vert u(.,t)\Vert_{L^{\infty}(\Omega)}\leqq Ct^{-\tilde{\sigma}}\Vert
u_{0}\Vert_{L^{r}(\Omega)}^{\tilde{\varpi}},\qquad\tilde{\sigma}=\frac
{1}{\frac{rp}{N}+p-2}=\frac{N}{rp}\tilde{\varpi}.
\]
And we deduce universal estimates as before:%
\[
\Vert u(.,t)\Vert_{L^{\infty}(\Omega)}\leqq Ct^{-\frac{1}{q-1+\lambda}}%
\quad\text{if }\gamma>0;\qquad\Vert u(.,t)\Vert_{L^{\infty}(\Omega)}\leqq
Ct^{-\frac{1}{p-2}}\quad\text{if }\nu>0\text{ and }p>2.
\]
Such methods can also be extended to porous media equations, and doubly
nonlinear equations involving operators of the form $u\mapsto-\Delta
_{p}(\left\vert u\right\vert ^{m-1}u).$

\section{ Regularization lemmas\label{sec2}}

We begin by a simple bootstrap property, used for example in \cite{Ve}. We
recall the proof for simplicity:

\begin{lemma}
\label{elem}Let $\omega\in\left(  0,1\right)  $ and $\sigma>0,$ and $K,M>0.$
Let $y$ be any positive function on $(0,T)$ such that $y(t)\leqq Mt^{-\sigma}%
$, and for any $0<s<t<T,$
\[
y(t)\leqq K(t-s)^{-\sigma}y^{\omega}(s),
\]
Then $y$ satisfies an estimate independent of $M:$ for any $t\in\left(
0,T\right)  ,$
\begin{equation}
y(t)\leqq2^{\sigma(1-\omega)^{-2}}(Kt^{-\sigma})^{(1-\omega)^{-1}} \label{sil}%
\end{equation}

\end{lemma}

\begin{proof}
We get by induction, for any $n\geqq1$%
\[
y^{\omega^{n-1}}(t/2^{n-1})\leqq K^{\omega^{n-1}}2^{n\sigma\omega^{n-1}%
}t^{-\sigma\omega^{n-1}}y^{\omega^{n}}(t/2^{n}),,\qquad y^{\omega^{n}}%
(t/2^{n})\leqq2^{n\sigma\omega^{n}}t^{-\sigma\omega^{n}}M^{\omega^{n}}.
\]
Then%
\[
y(t)\leqq K^{\sum_{k=0}^{n-1}\varpi^{k}}t^{-\sigma}{}^{\sum_{k=0}^{n}%
\varpi^{k}}2^{\sigma\sum_{k=0}^{n}(k+1)\varpi^{k})}M^{\omega^{n+1}},
\]
implying (\ref{sil}) as $n\rightarrow\infty,$ since $\lim_{n\rightarrow\infty
}M^{\omega^{n+1}}=1.\medskip$
\end{proof}

Next we show the Moser's type property:\medskip

\begin{proof}
[Proof of Lemma \ref{prod}](i) Let $\alpha$ be any real such that $\alpha\geqq
r-1,$ and $v(.,s)\in L^{\alpha+1}(\Omega)$. From (\ref{mac}), $\int_{\Omega
}v^{\alpha+1}(t)dx$ is decreasing for $t>s$. And $\int_{\Omega}v^{\beta
m\theta}(.,\xi)dx$ is finite for almost any $\xi\in\left(  s,t\right)  ,$
hence for a sequence $\left(  \xi_{n}\right)  $ decreasing to $s,$
\[
\int_{\Omega}v^{\alpha+1}(.,t)dx+\frac{C_{0}(\alpha+1)}{\beta^{m}}\int%
_{\xi_{n}}^{t}(\int_{\Omega}v^{\beta m\theta}(.,\xi)dx)^{\frac{1}{\theta}}%
d\xi\leqq\int_{\Omega}v^{\alpha+1}(.,\xi_{n})dx\leqq\int_{\Omega}v^{\alpha
+1}(.,s)dx.
\]
From (\ref{ele}), there holds $\beta m\theta>r.$ Applying again (\ref{mac})
with $\beta m\theta-1$ instead of $\alpha,$ and $\xi_{n}$ instead of $s,$ we
deduce that $\int_{\Omega}v^{\beta m\theta}(t)dx$ is decreasing for $t>s,$
thus
\[
\int_{\Omega}v^{\alpha+1}(.,t)dx+\frac{C_{0}(\alpha+1)}{\beta^{m}}(t-\xi
_{n})(\int_{\Omega}v^{\beta m\theta}(.,\xi_{n})dx)^{^{\frac{1}{\theta}}}%
\leqq\int_{\Omega}v^{\alpha+1}(.,s)dx.
\]
As $n\rightarrow\infty,$ $v(.,\xi_{n})\rightarrow v(.,s)$ in $L_{loc}%
^{1}(\Omega),$ and after extraction, a.e. in $\Omega.$ Then from the Fatou
lemma,
\[
\int_{\Omega}v^{\alpha+1}(.,t)dx+\frac{C_{0}(\alpha+1)}{\beta^{m}}%
(t-s)(\int_{\Omega}v^{\beta m\theta}(.,s)dx)^{^{\frac{1}{\theta}}}\leqq
\int_{\Omega}v^{\alpha+1}(.,s)dx.
\]
Hence
\begin{equation}
\Vert v(t)\Vert_{L^{\beta m\theta}(\Omega)}^{\beta m\theta}\leqq\left(
\frac{\beta^{m}}{C_{0}(\alpha+1)}\frac{1}{t-s}\Vert v(s)\Vert_{L^{\alpha
+1}(\Omega)}^{\alpha+1}\right)  ^{\theta}, \label{1.6}%
\end{equation}

$\bullet$ \textbf{Case }$r>1.$We start from $s=0,$ we have $v_{0}\in
L^{r}(\Omega).$ We take\textbf{ } $\alpha_{0}=r-1,$ thus $\int_{\Omega
}v^{\alpha_{0}+1}(t)dx$ is finite, and set $\beta_{0}=1+\left(  \alpha
_{0}+\lambda\right)  /m.$ We define sequences $\left(  t_{n}\right)  ,\left(
\alpha_{n}\right)  ,\left(  r_{n}\right)  ,\left(  \beta_{n}\right)  ,$ by
$t_{0}=0,$ $r_{0}=r$ and for any $n\geqq1,$
\[
t_{n}=t(1-\frac{1}{2^{n}}),\qquad r_{n}=\alpha_{n}+1,\qquad\beta_{n}%
=1+\frac{\alpha_{n}+\lambda}{m},\qquad r_{n+1}=\beta_{n}m\theta=(r_{n}%
+\lambda+m-1)\theta,
\]
hence $\left(  r_{n}\right)  ,\left(  \beta_{n}\right)  $ are increasing,
since $r_{1}>r$ from (\ref{ele}). In (\ref{1.6}), we replace $s,t,\alpha,\beta
m\theta,$ by $t_{n},t_{n+1}r_{n},r_{n+1},$ and get
\begin{equation}
\Vert v(t_{n+1})\Vert_{L^{r_{n+1}}(\Omega)}\leqq\left(  \frac{1}{C_{0}%
(m\theta)^{m}}\frac{r_{n+1}^{m}}{r_{n}}\frac{1}{t_{n+1}-t_{n}}\right)
^{\frac{\theta}{r_{n+1}}}\Vert v(t_{n})\Vert_{L^{r_{n}}(\Omega)}^{\frac
{\theta.r_{n}}{r_{n+1}}}. \label{1.7}%
\end{equation}
From (\ref{mac}), it follows that
\begin{equation}
\Vert v(t)\Vert_{L^{r_{n+1}}(\Omega)}\leqq\Vert v(t_{n+1})\Vert_{L^{r_{n+1}%
}(\Omega)}\leqq I_{n}J_{n}L_{n}\Vert v_{0}\Vert_{L^{r}(\Omega)}^{\frac
{\theta^{n+1}.r}{r_{n+1}}}, \label{1.8}%
\end{equation}
where
\[
I_{n}=%
{\displaystyle\prod_{k=1}^{n+1}}
(\frac{r_{k}^{m}}{r_{k-1}})^{\frac{\theta^{n+2-k}}{r_{n+1}}},\quad J_{n}=%
{\displaystyle\prod_{k=1}^{n+1}}
\left(  \frac{1}{t_{k}-t_{k-1}}\right)  ^{\frac{\theta^{n+2-k}}{r_{n+1}}%
},\quad L_{n}=\left(  C_{0}(m\theta)^{m}\right)  ^{-\sum_{k=1}^{n+1}%
\frac{\theta^{n+2-k}}{r_{n+1}}}.
\]
Since $r_{n}=$ $\theta^{n}(r+(\lambda+m-1)\theta^{\prime}(1-\theta^{-n})),$ we
find
\begin{equation}
\lim_{n\rightarrow\infty}\frac{\theta^{n+1}r}{r_{n+1}}=\varpi_{r,m,\lambda
,\theta},\qquad\lim_{n\rightarrow\infty}\frac{1}{r_{n+1}}\sum_{k=1}%
^{n+1}\theta^{n+2-k}=\sigma_{r,m,\lambda,\theta},\qquad\lim_{n\rightarrow
\infty}\sum_{k=1}^{n+1}k\theta^{1-k}=\theta^{\prime2} \label{1.9}%
\end{equation}
As a consequence
\begin{equation}
\lim_{n\rightarrow\infty}J_{n}=2^{-\frac{\varpi_{r,m,\lambda,\theta}}{r}%
\theta^{\prime2}}t^{-\sigma_{r,m,\lambda,\theta}},\qquad\lim_{n\rightarrow
\infty}L_{n}=\left(  C_{0}(m\theta)^{m}\right)  ^{-\sigma_{r,m,\lambda,\theta
}}. \label{1.10}%
\end{equation}
Otherwise
\[
\ln I_{n}=\frac{m}{r_{n+1}}\sum_{k=1}^{n+1}\theta^{n+2-k}\ln r_{k}-\frac
{1}{r_{n+1}}\sum_{k=0}^{n}\theta^{n+1-k}\ln r_{k}=\frac{\theta^{n+1}}{r_{n+1}%
}(m\theta\sum_{k=1}^{n+1}\theta^{-k}\ln r_{k}-\sum_{k=0}^{n}\theta^{-k}\ln
r_{k})
\]
and the sum $S=\sum_{k=0}^{\infty}\theta^{-k}\ln r_{k}$ is finite, since
$r_{k}\leqq\theta^{k}(r+\left\vert \lambda+m-1\right\vert \theta^{\prime})$.
Then $I_{n}$ has a finite limit $\ell=\ell(N,m,r,\lambda,\theta)=\exp
(r^{-1}\varpi_{r,m,\lambda,\theta}((m\theta-1)S-m\theta\ln r)).$ Thus we can
go to the limit in (\ref{1.8}), and the conclusion follows.$\medskip$

$\bullet$ \textbf{Case }$r=1.$ If (H$_{1}$) holds we can take $\alpha
_{0}=r-1=0$ and the proof is done. Next assume (H$_{2}$). Then we obtain, for
any $0\leqq s<t<T,$ and a constant $C$ as before,
\begin{align*}
\Vert v(.,t)\Vert_{L^{\infty}(\Omega)}  &  \leqq C(t-s)^{-\sigma
_{\rho,m,\lambda,\theta}}\Vert v(.,s)\Vert_{L^{\rho}(\Omega)}^{\varpi
_{\rho,m,\lambda,\theta}}\\
&  \leqq C(t-s)^{-\sigma_{\rho,m,\lambda,\theta}}\Vert v(.,s)\Vert_{L^{\infty
}(\Omega)}^{\varpi_{\rho,m,\lambda,\theta}(\rho-1)/\rho}\Vert v(.,s)\Vert
_{L^{1}(\Omega)}^{\varpi_{\rho,m,\lambda,\theta}/\rho}\\
&  \leqq C\Vert v_{0}\Vert_{L^{1}(\Omega)}^{\varpi_{\rho,m,\lambda,\theta
}/\rho}(t-s)^{-\sigma_{\rho,m,\lambda,\theta}}\Vert v(.,s)\Vert_{L^{\infty
}(\Omega)}^{\varpi_{\rho,m,\lambda,\theta}(\rho-1)/\rho}%
\end{align*}
Let $y(t)=\Vert v(.,t)\Vert_{L^{\infty}(\Omega)}.$ We can apply Lemma
\ref{elem} to $y$, with
\[
\sigma=\sigma_{\rho,m,\lambda,\theta},\quad\quad\omega=\frac{\varpi
_{\rho,m,\lambda,\theta}}{\rho^{\prime}},\quad\quad K=C\Vert v_{0}\Vert
_{L^{1}(\Omega)}^{\varpi_{\rho,m,\lambda,\theta}/\rho},\quad\quad M=C\Vert
v_{0}\Vert_{L^{\rho}(\Omega)}^{\varpi_{\rho,m,\lambda,\theta/\rho}}.
\]
Indeed $\omega<1$ since $\rho\theta^{\prime}(1-m-\lambda)<1$. Then there
holds
\[
\Vert v(.,t)\Vert_{L^{\infty}(\Omega)}\leqq2^{\sigma(1-\omega)^{-2}%
}(Kt^{-\sigma})^{(1-\omega)^{-1}}=2^{\sigma(1-\omega)^{-2}}C^{(1-\omega)^{-1}%
}t^{-\sigma(1-\omega)^{-1}}\Vert v_{0}\Vert_{L^{1}(\Omega)}^{\varpi
_{\rho,q,\lambda,\theta}/\rho((1-\omega))}.
\]
Noticing that $\sigma(1-\omega)^{-1}=\sigma_{1,m,\lambda,\theta}$ and
$\varpi_{\rho,m,\lambda,\theta}/\rho((1-\omega))=\varpi_{1,m,\lambda,\theta}$,
we deduce
\begin{equation}
\Vert v(.,t)\Vert_{L^{\infty}(\Omega)}\leqq Ct^{-\sigma_{1,m,\lambda,\theta}%
}\Vert v_{0}\Vert_{L^{1}(\Omega)}^{\varpi_{1,m,\lambda,\theta}}, \label{blo}%
\end{equation}
with a new constant $C$, now depending on $\rho.$
\end{proof}

\begin{remark}
This lemma can be compared with the result of \cite[Theorem 2.1]{Por} obtained
by the Stampacchia's method. In order to obtain decay estimates for the
solutions $u$ of a parabolic equation such as (\ref{un}) or (\ref{eqge}), the
Moser's method consists to take as test functions powers $\left\vert
u\right\vert ^{\alpha-1}u$ of $u;$ the Stampacchia's method uses test
functions of the form $(u-k)^{+}\mathrm{sign}u.$ If one applies to
sufficiently smooth solutions, both techniques leed to decay estimates of the
same type. In the case of weaker solutions, the Stampacchia method supposes
that the functions $(u-k)^{+}$ are admissible in the equation, which leads to
assume that $u(.,t)\in W^{1,2}(\Omega),$ see \cite{Por}. In the sequel we
combine Moser's method with regularization or truncature of $u$, in order to
admit powers as test functions. So we do not need to make this assumption,
thus the Moser's method appears to be more performant.
\end{remark}

Such type of $L^{\infty}$ estimates as (\ref{trap}) may imply a universal one,
that means independent of the initial data, in case $\Omega$ is bounded. This
was observed for example in \cite{Ve}:

\begin{lemma}
\label{univ} Let $\Omega$ be bounded. (i) Let $v\in C(\left[  0,T\right)
;L_{loc}^{1}(\Omega))$ be nonnegative, and $v_{0}=v(x,0)\in L^{1}(\Omega),$
such that for some $C>0,$ for any $0\leqq s<t<T,$
\[
\Vert v(.,t)\Vert_{L^{\infty}(\Omega)}\leqq C(t-s)^{-\sigma}\Vert
v(.,s)\Vert_{L^{1}(\Omega)}^{\varpi},
\]
where $\sigma>0,\varpi\in(0,1).$ Then there exists $M=M(C,\sigma
,\varpi,\left\vert \Omega\right\vert )$ such that for any $t\in\left(
0,T\right)  ,$
\begin{equation}
\Vert v(.,t)\Vert_{L^{\infty}(\Omega)}\leqq Mt^{-\frac{\sigma}{1-\varpi}}.
\label{foul}%
\end{equation}
(ii) As a consequence, if $v$ satisfies (\ref{trap}), with $m-1+\lambda>0,$
then
\begin{equation}
\Vert v(.,t)\Vert_{L^{\infty}(\Omega)}\leqq Mt^{-\frac{1}{m-1+\lambda}}.
\label{fol}%
\end{equation}

\end{lemma}

\begin{proof}
(i) For any $0<s<t<T,$%
\[
\Vert v(.,t)\Vert_{L^{\infty}(\Omega)}\leqq C(t-s)^{-\sigma}\Vert
v(.,s)\Vert_{L^{r}(\Omega)}^{\varpi}\leqq C(t-s)^{-\sigma}\left\vert
\Omega\right\vert ^{\varpi}\Vert v(.,s)\Vert_{L^{\infty}(\Omega)}^{\varpi}%
\]
Since $\varpi<1,$ (\ref{foul}) follows from Lemma \ref{elem}: for any
$t\in(0,T),$
\[
\Vert v(.,t)\Vert_{L^{\infty}(\Omega)}\leqq2^{\sigma(1-\varpi)^{-2}%
}(C\left\vert \Omega\right\vert ^{\varpi}t^{-\sigma})^{(1-\varpi)^{-1}}.
\]
(ii) If $v$ satisfies (\ref{trap}), with $m-1+\lambda>0,$ we take
$\sigma=\sigma_{r,m,\lambda,\theta},$ and $\varpi=\varpi_{r,m,\lambda,\theta}$
defined at (\ref{jam}), then $\varpi=(1+(m-1+\lambda)\theta^{\prime}%
/r)^{-1}<1$ and $\sigma((1-\varpi)^{-1}=(m-1+\lambda)^{-1},$ which proves
(\ref{fol}).\medskip
\end{proof}

In the sequel Lemma \ref{prod} is applied in situations where (\ref{mac})
comes from an estimate of $v$ in a Sobolev Space $W^{1,m}(Q_{\Omega,s,t}),$
when $1<m<N,$ with $\theta=N/(N-m),$ or $m=N$ and $\theta>1$ is arbitrary.

In the case $m>N,$ where Lemma \ref{prod} does not bring information, we use
in the sequel a limit form of Gagliardo-Nirenberg inequality, see the proof of
Theorems \ref{bound}, \ref{deco} and \ref{Bound2}:

\begin{lemma}
\label{gani}Let $m>N,$ and $r\geqq1.$ Let $\Omega$ be any domain in
$\mathbb{R}^{N}.$ Then there exists $C=C(N,m,r)>0$ such that for any $w\in
L^{r}(\Omega)\cap W_{0}^{1,m}(\Omega),$
\[
\Vert w\Vert_{L^{\infty}(\Omega)}\leqq C\Vert\nabla w\Vert_{L^{m}(\Omega)}%
^{k}\Vert w\Vert_{L^{r}(\Omega)}^{1-k},\qquad\frac{1}{k}=1+r(\frac{1}{N}%
-\frac{1}{m}).
\]

\end{lemma}

\begin{proof}
By extension by $0$ outside of $\Omega,$ we can assume $\Omega=\mathbb{R}%
^{N}.$ Since $m>N,$ for any $\varphi\in\mathcal{D}\left(  \mathbb{R}%
^{N}\right)  ,$ and any $x\in\mathbb{R}^{N},$%
\[
\left\vert \varphi(x)\right\vert \leqq C(N,m)(\left\vert \int_{B(x,1)}\varphi
dx\right\vert +\Vert\nabla\varphi\Vert_{L^{m}(B(x,1))})\leqq C(N,m,r)\left(
\Vert\varphi\Vert_{L^{r}(\mathbb{R}^{N})}+\Vert\nabla\varphi\Vert
_{L^{m}(\mathbb{R}^{N})}\right)  ;
\]
by density, there holds
\[
\Vert w\Vert_{L^{\infty}(\mathbb{R}^{N})}\leqq C\left(  \Vert w\Vert
_{L^{r}(\mathbb{R}^{N})}+\Vert\nabla w\Vert_{L^{m}(\mathbb{R}^{N})}\right)
\]
for any $w\in L^{r}(\mathbb{R}^{N})\cap W^{1,m}(\mathbb{R}^{N}).$ Setting
$w_{t}(x)=w(tx)$ for any $t>0,$ we find
\[
\Vert w\Vert_{L^{\infty}(\mathbb{R}^{N})}=\Vert w_{t}\Vert_{L^{\infty
}(\mathbb{R}^{N})}\leqq C\left(  t^{-\frac{N}{r}}\Vert w\Vert_{L^{r}%
(\mathbb{R}^{N})}+t^{\frac{m-N}{m}}\Vert\nabla w\Vert_{L^{m}(\mathbb{R}^{N}%
)}\right)  ;
\]
the result follows by taking $t=(\Vert w\Vert_{L^{r}(\mathbb{R}^{N})}%
/\Vert\nabla w\Vert_{L^{m}(\mathbb{R}^{N})})^{1/(1-N/m+N/r)}$.
\end{proof}

\section{The Hamilton-Jacobi equation in $\mathbb{R}^{N}$\label{Rn}}

\subsection{Different notions of solution\label{diff}}

\noindent In this section we study the Cauchy problem (\ref{cau}).

Here we consider the solutions \textit{in a weak sense}, which does not use
any formulation in terms of semigroups:

\begin{definition}
\label{weaksol}We say that a \textit{nonnegative} function $u$ is a
\textbf{weak solution} (resp. subsolution) of equation of (\ref{un}) in
$Q_{\mathbb{R}^{N},T},$ if $u\in L_{loc}^{1}(Q_{\mathbb{R}^{N},T}),$ and
$|\nabla u|\in L_{loc}^{q}(Q_{\mathbb{R}^{N},T}),$ and
\begin{equation}
\int_{0}^{T}\int_{\Omega}(-u\varphi_{t}-u\Delta\varphi+|\nabla u|^{q}%
\varphi)dxdt=0,\quad(\text{resp.}\leqq),\quad\quad\forall\varphi\in
\mathcal{D}(Q_{\mathbb{R}^{N},T}),\varphi\geqq0. \label{for}%
\end{equation}

\end{definition}

\begin{remark}
From \cite{BiDao1}, any nonnegative weak solution satisfies
\begin{equation}
u\in L_{loc}^{\infty}(Q_{\mathbb{R}^{N},T}),\qquad\nabla u\in L_{loc}%
^{2}(Q_{\mathbb{R}^{N},T}),\qquad u\in C((0,T);L_{loc}^{\rho}(\mathbb{R}%
^{N}))\quad\forall\rho\geqq1. \label{2.1}%
\end{equation}
Hence (\ref{for}) is equivalent to:
\begin{equation}
\int_{0}^{T}\int_{\Omega}(-u\varphi_{t}+\nabla u.\nabla\varphi+|\nabla
u|^{q}\varphi)dxdt=0,\qquad\forall\varphi\in\mathcal{D}(Q_{\mathbb{R}^{N},T}),
\label{fur}%
\end{equation}
and there holds, for any $s,\tau\in(0,T),$
\begin{equation}
\int_{\mathbb{R}^{N}}u(.,\tau)\varphi(.,\theta)dx-\int_{\mathbb{R}^{N}%
}u(.,s)\varphi(.,s)dx+\int_{s}^{\tau}\int_{\mathbb{R}^{N}}(-u\varphi
_{t}+\nabla u.\nabla\varphi+|\nabla u|^{q}\varphi)dxdt=0;
\end{equation}
and for any $\psi\in$ $C_{c}^{2}\left(  \mathbb{R}^{N}\right)  ,$
\begin{equation}
\int_{\mathbb{R}^{N}}u(.,\tau)\psi dx-\int_{\mathbb{R}^{N}}u(.,s)\psi
dx+\int_{s}^{\tau}\int_{\mathbb{R}^{N}}(\nabla u.\nabla\psi+|\nabla u|^{q}\psi
dxdt=0.
\end{equation}

\end{remark}

\begin{definition}
\label{weakrloc}Let $u_{0}\in L_{loc}^{r}\left(  \mathbb{R}^{N}\right)
,r\geqq1.\medskip$

We say that $u$ is a \textbf{weak }$L_{loc}^{r}$\textbf{ solution }if $u$ is a
weak solution of (\ref{un}) and the extension of $u$ by $u_{0}$ at time $0$
satisfies $u\in C\left(  \left[  0,T\right)  ;L_{loc}^{r}(\mathbb{R}%
^{N}\right)  ).\medskip$

We say that $u$ is a \textbf{weak }$r$ \textbf{solution} of problem
(\ref{cau}) if it is a weak solution of equation (\ref{un}) such that
\begin{equation}
\lim_{t\rightarrow0}\int_{\mathbb{R}^{N}}u^{r}(.,t)\psi dx=\int_{\mathbb{R}%
^{N}}u_{0}^{r}\psi dx,\qquad\forall\psi\in C_{c}(\mathbb{R}^{N}). \label{int}%
\end{equation}

\end{definition}

\begin{definition}
\label{weakmloc}Let $u_{0}$ be any nonnegative Radon measure in $\mathbb{R}%
^{N},$ we say that $u$ is a \textbf{weak }$\mathcal{M}_{loc}$\textbf{ solution
}of problem (\ref{cau}) if it is a weak solution of (\ref{un}) such that
\textbf{ }
\begin{equation}
\lim_{t\rightarrow0}\int_{\mathbb{R}^{N}}u(.,t)\psi dx=\int_{\mathbb{R}^{N}%
}\psi du_{0},\qquad\forall\psi\in C_{c}(\mathbb{R}^{N}). \label{mea}%
\end{equation}

\end{definition}

\begin{remark}
\label{comp}Obviously, any weak $L_{loc}^{r}$ solution is a weak $r$ solution.
When $r=1,$ the notions of weak $1$-solution and weak $\mathcal{M}_{loc}$
solution coincide\textbf{. }When $r>1,$ it can be easily checked that $u$ is a
weak $L_{loc}^{r}$ solution if and only if it is a weak $r$ solution and
\begin{equation}
\lim_{t\rightarrow0}\int_{\mathbb{R}^{N}}u(.,t)\psi dx=\int_{\mathbb{R}^{N}%
}u_{0}\psi dx,\qquad\forall\psi\in C_{c}(\mathbb{R}^{N}). \label{fai}%
\end{equation}

\end{remark}

Other types of solutions using the semigroup of the heat equation have been
introduced in (\cite{BASoWe}):

\begin{definition}
\label{mildr}Let $u_{0}\in L^{r}\left(  \mathbb{R}^{N}\right)  .$ A function
$u$ is called \textbf{mild }$L^{r}$\textbf{ solution} of problem (\ref{cau})
if $u\in C(\left[  0,T\right)  ;L^{r}\left(  \mathbb{R}^{N}\right)  )$, and
$|\nabla u|^{q}\in L_{loc}^{1}(\left[  0,T\right)  ;L^{r}\left(
\mathbb{R}^{N}\right)  )$ and
\[
u(.,t)=e^{t\Delta}u_{0}-\int_{0}^{t}e^{(t-s)\Delta}|\nabla u(.,s)|^{q}%
ds\qquad\text{in }L^{r}(\mathbb{R}^{N});
\]
here $e^{t\Delta}$ is the semi-group of the heat equation acting on
$L^{r}\left(  \mathbb{R}^{N}\right)  .$
\end{definition}

\begin{definition}
\label{mildm}Let $u_{0}\in\mathcal{M}_{b}^{+}(\mathbb{R}^{N}).$ A function $u$
is called \textbf{mild }$\mathcal{M}$ \textbf{solution} of (\ref{cau}) if
$u\in C_{b}((0,T);L^{1}\left(  \mathbb{R}^{N}\right)  )$ and $|\nabla
u|^{q}\in L_{loc}^{1}(\left[  0,T\right)  ;L^{1}\left(  \mathbb{R}^{N}\right)
),$ and for any $0<t<T,$
\begin{equation}
u(.,t)=e^{t\Delta}u_{0}(.)-\int_{0}^{t}e^{(t-s)\Delta}|\nabla u(.,s)|^{q}%
ds\qquad\text{in }L^{1}(\mathbb{R}^{N}), \label{hc}%
\end{equation}
where $e^{t\Delta}$ is defined on $\mathcal{M}_{b}^{+}(\mathbb{R}^{N})$ as the
adjoint of the operator $e^{t\Delta}$ on $C_{0}(\mathbb{R}^{N}),$ the space of
continuous functions on $\mathbb{R}^{N}$ which tend to $0$ as $\left\vert
x\right\vert \rightarrow\infty.$
\end{definition}

\begin{remark}
Every mild $L^{r}$ solution is a weak $L_{loc}^{r}$ solution$.$ Any mild
$\mathcal{M}$ solution is a weak $\mathcal{M}_{loc}$ solution. Indeed for any
$0<\epsilon<t<T,$ we find
\[
u(.,t)=e^{(t-\epsilon)\Delta}u(.,\epsilon)-\int_{\epsilon}^{t}e^{(t-s)\Delta
}|\nabla u(.,s)|^{q}ds\qquad\text{in }L^{1}(\mathbb{R}^{N});
\]
and $u(.,\epsilon)\in L^{1}(\mathbb{R}^{N}),$ thus $u$ is a weak solution on
$\left(  \epsilon,T\right)  ,$ then on $(0,T).$ As $t\rightarrow0,$
$u(.,t)-e^{t\Delta}u_{0}(.)$ converges to $0$ in $L^{1}(\mathbb{R}^{N}),$ then
weakly *, and $e^{t\Delta}u_{0}(.)\rightarrow u_{0}$ weakly *, hence
(\ref{mea}) holds.
\end{remark}

Another definition of solution with initial data measure was given in
(\cite{BeLa99}):

\begin{definition}
\label{weaksemi}Let $u_{0}\in\mathcal{M}_{b}^{+}(\mathbb{R}^{N}).$ A function
$u$ is called \textbf{weak semi-group solution} if $u\in C((0,T);L^{1}\left(
\mathbb{R}^{N}\right)  )$ and $|\nabla u|^{q}\in L_{loc}^{1}(\left[
0,T\right)  ;L^{1}\left(  \mathbb{R}^{N}\right)  )$ and for any $0<\epsilon
<t<T,$
\begin{equation}
u(.,t)=e^{(t-\epsilon)\Delta}u(.,\epsilon)-\int_{\epsilon}^{t}e^{(t-s)\Delta
}|\nabla u(.,s)|^{q}ds\qquad\text{in }L^{1}(\mathbb{R}^{N}), \label{ha}%
\end{equation}
\begin{subequations}
\label{0}%
\begin{equation}
\lim_{t\rightarrow0}\int_{\mathbb{R}^{N}}u(.,t)\varphi dx=\int_{\mathbb{R}%
^{N}}\varphi du_{0},\qquad\forall\varphi\in C_{b}(\mathbb{R}^{N}), \label{hb}%
\end{equation}

\end{subequations}
\end{definition}

In fact the two definitions\textbf{ }coincide, see the proof in the Appendix:

\begin{lemma}
\label{lem} Let $u_{0}\in\mathcal{M}_{b}^{+}(\mathbb{R}^{N}).$ Then
\[
u\text{ is a mild }\mathcal{M}\text{ solution of (\ref{cau})}%
\Longleftrightarrow u\text{ is a weak semi-group solution of (\ref{cau}).}%
\]

\end{lemma}

\begin{remark}
\label{all}All these definitions of semi-group solutions assume a global in
space condition: $|\nabla u|^{q}\in L_{loc}^{1}(\left[  0,T\right)
;L^{1}\left(  \mathbb{R}^{N}\right)  )$ or $|\nabla u|^{q}\in L_{loc}%
^{1}(\left[  0,T\right)  ;L^{r}\left(  \mathbb{R}^{N}\right)  ).$ Observe also
that (\ref{hb}) is assumed for any $\varphi\in C_{b}(\mathbb{R}^{N})$. On the
contrary, our definitions of weak\textbf{ }solutions are \textbf{local in
space and time,} they do not require such global properties.\bigskip
\end{remark}

Finally we mention another weaker form of semi-group solutions, given in
(\cite{BASoWe}), which will be used in the sequel:

\begin{definition}
\label{pointsol}Let $u_{0}\in\mathcal{M}_{b}^{+}(\mathbb{R}^{N}).$ Then $u$ is
\textit{a pointwise mild solution} of (\ref{cau}) if $u\in L_{loc}%
^{1}(Q_{\mathbb{R}^{N},T}),$ and $|\nabla u|^{q}\in L_{loc}^{1}(Q_{\mathbb{R}%
^{N},T}),$ and
\[
u(x,t)=(e^{t\Delta}u_{0})(x)-\int_{0}^{t}\int_{\mathbb{R}^{N}}%
g(x-y,t-s)|\nabla u(y,s)|^{q}dyds\qquad\text{for a.e. }(x,t)\in Q_{\mathbb{R}%
^{N},T},
\]
where $g$ is the heat kernel.
\end{definition}

\begin{remark}
For $r\geqq1,$ it is clear that every mild $L^{r}$ solution is a pointwise
mild solution. If $u_{0}\in L^{1}\left(  \mathbb{R}^{N}\right)  $ every
pointwise mild solution is a mild $L^{1}$ solution; if $u_{0}\in
\mathcal{M}_{b}^{+}(\mathbb{R}^{N})$, every pointwise mild solution, is a mild
$\mathcal{M}$ solution. see \cite[Proposition 1.1 and Remark 1.2]%
{BASoWe}.\medskip
\end{remark}

\subsection{ Decay of the norms}

\textbf{ }Next we show a \textit{decay result} for the solutions of Hamilton
Jacobi equations, which is valid for any $q>1,$ and for \textit{all the weak
solutions}, \textit{with no condition of boundedness at infinity}.\medskip

When $q\leqq2,$ any weak solution $u$ of equation (\ref{un}) is smooth: $u\in
C^{2,1}\left(  Q_{\mathbb{R}^{N},T}\right)  ,$ from \cite[Theorem 2.9]%
{BiDao1}. Since it may be false for $q>2,$ we regularize $u$ by convolution,
setting
\[
u_{\varepsilon}=u\ast\varrho_{\varepsilon},
\]
where $(\varrho_{\varepsilon})_{\varepsilon>0}$ is a sequence of mollifiers.
We recall that for given $0<s<\tau<T$, and $\varepsilon$ small enough,
$u_{\varepsilon}$ is a \textit{subsolution} of equation (\ref{un})$,$ see
\cite{BiDao1}:
\begin{equation}
(u_{\varepsilon})_{t}-\nu\Delta u_{\varepsilon}+|\nabla u_{\varepsilon}%
|^{q}\leqq0,\qquad\text{in }Q_{\mathbb{R}^{N},s,\tau}. \label{epsi}%
\end{equation}

\begin{theorem}
\label{decay}Assume $q>1.$ Let $r\geqq1.$ Let $u_{0}\in L^{r}(\mathbb{R}^{N})$
be nonnegative. Let $u$ be any non-negative weak $r$ solution of problem
(\ref{cau}).\medskip

\noindent(i) Then $u(.,t)\in L^{r}\left(  \mathbb{R}^{N}\right)  $ for any
$t\in\left(  0,T\right)  ,$ and%
\begin{equation}
\int_{\mathbb{R}^{N}}u^{r}(.,t)dx\leqq\int_{\mathbb{R}^{N}}u_{0}^{r}dx.
\label{ghi}%
\end{equation}
(ii) Moreover \textbf{ }$u^{r-1}|\nabla u|^{q}\in L_{loc}^{1}(\left[
0,T\right)  ;L^{1}\left(  \mathbb{R}^{N}\right)  );$ and $u^{r-2}|\nabla
u|^{2}\in L_{loc}^{1}(\left[  0,T\right)  ;L^{1}\left(  \mathbb{R}^{N}\right)
)$ if $r>1$ and $\nu>0.$ For any $t\in\left(  0,T\right)  ,$ we have the
\textbf{equalities}
\begin{equation}
\int_{\mathbb{R}^{N}}u^{r}(.,t)dx+r\int_{0}^{t}\int_{\mathbb{R}^{N}}%
u^{r-1}|\nabla u|^{q}dxdt+r(r-1)\nu\int_{0}^{t}\int_{\mathbb{R}^{N}}%
u^{r-2}|\nabla u|^{2}dxdt=\int_{\mathbb{R}^{N}}u_{0}^{r}dx,\qquad\text{if
}r>1, \label{egge}%
\end{equation}%
\begin{equation}
\int_{\mathbb{R}^{N}}u(.,t)dx+\int_{0}^{t}\int_{\mathbb{R}^{N}}|\nabla
u|^{q}dxdt=\int_{\mathbb{R}^{N}}u_{0}dx,\qquad\text{if }r=1, \label{agga}%
\end{equation}%
\begin{equation}
\lim_{t\rightarrow0}\int_{\mathbb{R}^{N}}u^{r}(.,t)dx=\int_{\mathbb{R}^{N}%
}u_{0}^{r}dx. \label{limu}%
\end{equation}
(iii) $u^{q-1+r}\in L_{loc}^{1}((\left[  0,T\right)  ;W^{1,1}\left(
\mathbb{R}^{N}\right)  );$ and if $\nu>0,$ then $u^{r/2}\in L_{loc}%
^{2}(\left[  0,T\right)  ;W^{1,2}\left(  \mathbb{R}^{N}\right)  ).$

\noindent(iv) If $u$ is a weak $L_{loc}^{r}$ solution, then $u\in C(\left[
0,T\right)  ;L^{r}\left(  \mathbb{R}^{N}\right)  ).$
\end{theorem}

\begin{proof}
(i) \textbf{ First step: case } $q^{\prime}>N/r.$ That means $r\geqq N$ or $q$
is small enough: $1<q<N/(N-r).$ Let $0<s<\tau<T$. Take $\varepsilon>0$ small
enough such that (\ref{epsi}) holds. Let $\delta>0,$ and $u_{\varepsilon
,\delta}=u_{\varepsilon}+\delta,$ so that $u_{\varepsilon,\delta}^{r-2}$ is
well defined for $r<2.$ For any $R>0,$ we consider $\xi(x)=\xi_{R}%
(x)=\psi(x/R),$ where $\psi(x)\in\left[  0,1\right]  ,\psi(x)=1$ for
$\left\vert x\right\vert \leqq1,\psi(x)=0$ for $\left\vert x\right\vert
\geqq2.$ Multiplying (\ref{epsi}) by $u_{\varepsilon,\delta}^{r-1}\xi
^{\lambda}$ where $\lambda>0$, we get for any $t\in\left[  s,\tau\right]  ,$
\begin{align*}
&  \frac{d}{dt}\left(  \frac{1}{r}\int_{\mathbb{R}^{N}}u_{\varepsilon,\delta
}^{r}(.,t)\xi^{\lambda}dx\right)  +(r-1)\nu\int_{\mathbb{R}^{N}}%
u_{\varepsilon,\delta}^{r-2}\left\vert \nabla u_{\varepsilon,\delta
}\right\vert ^{2}(.,t)\xi^{\lambda}dx+\int_{\mathbb{R}^{N}}|\nabla
u_{\varepsilon,\delta}|^{q}u_{\varepsilon,\delta}^{r-1}\xi^{\lambda-1}%
\xi^{\lambda}dx\\
&  \leqq-\lambda\int_{\mathbb{R}^{N}}u_{\varepsilon,\delta}^{r-1}\xi
^{\lambda-1}\xi^{\lambda-1}\nabla u_{\varepsilon,\delta}.\nabla\xi dx,
\end{align*}
and from the H\"{o}lder inequality, with $C=C(q,\lambda)$
\[
\lambda\int_{\mathbb{R}^{N}}u_{\varepsilon,\delta}^{r-1}\left\vert \nabla
u_{\varepsilon,\delta}\right\vert (.,t)\xi^{\lambda-1}\left\vert \nabla
\xi\right\vert dx\leqq\frac{1}{2}\int_{\mathbb{R}^{N}}|\nabla u_{\varepsilon
,\delta}(.,t)|^{q}u_{\varepsilon,\delta}^{r-1}\xi^{\lambda}dx+C\int%
_{\mathbb{R}^{N}}u_{\varepsilon,\delta}^{r-1}(.,t)\xi^{\lambda-q^{\prime}%
}\left\vert \nabla\xi\right\vert ^{q^{\prime}}dx,
\]%
\[
\int_{\mathbb{R}^{N}}u_{\varepsilon,\delta}^{r-1}(.,t)\xi^{\lambda-q^{\prime}%
}\left\vert \nabla\xi\right\vert ^{q^{\prime}}dx\leqq\left(  \int%
_{\mathbb{R}^{N}}u_{\varepsilon,\delta}^{r}(.,t)\xi^{\lambda}dx\right)
^{\frac{1}{r^{\prime}}}\left(  \int_{\mathbb{R}^{N}}\xi^{\lambda-rq^{\prime}%
}\left\vert \nabla\xi\right\vert ^{rq^{\prime}}dx\right)  ^{\frac{1}{r}}.
\]
Choosing $\lambda=rq^{\prime}$ we deduce
\[
\frac{d}{dt}\left(  \left(  \int_{\mathbb{R}^{N}}u_{\varepsilon,\delta}%
^{r}(.,t)\xi^{\lambda}dx\right)  ^{\frac{1}{r}}\right)  \leqq CR^{\frac{N}%
{r}-q^{\prime}},
\]
where $C=C(N,q,r,\psi).$ By integration,
\[
\left(  \int_{\mathbb{R}^{N}}u_{\varepsilon,\delta}^{r}(.,t)\xi^{\lambda
}dx\right)  ^{\frac{1}{r}}\leqq\left(  \int_{\mathbb{R}^{N}}u_{\varepsilon
,\delta}^{r}(.,s)\xi^{\lambda}dx\right)  ^{\frac{1}{r}}+C\tau R^{\frac{N}%
{r}-q^{\prime}}.
\]
with a new constant $C$ as above. Let $R_{0}>0$ be fixed and take $R>R_{0},$
thus
\[
\left(  \int_{B_{R_{0}}}u_{\varepsilon,\delta}^{r}(.,t)dx\right)  ^{\frac
{1}{r}}\leqq\left(  \int_{B_{2R}}u_{\varepsilon,\delta}^{r}(.,s)\xi^{\lambda
}dx\right)  ^{1/r}+C\tau R^{\frac{N}{r}-q^{\prime}}.
\]
As $\delta\rightarrow0,$ and then as $\varepsilon\rightarrow0,$ we deduce
that
\begin{equation}
\left(  \int_{B_{R_{0}}}u(.,t)^{r}dx\right)  ^{\frac{1}{r}}\leqq\left(
\int_{\mathbb{R}^{N}}u(.,s)^{r}\xi^{\lambda}dx\right)  ^{\frac{1}{r}}+C\tau
R^{\frac{N}{r}-q^{\prime}} \label{fis}%
\end{equation}
for any $0<s<t<T;$ from (\ref{int}) we obtain, as $s\rightarrow0$,
\[
\left(  \int_{B_{R_{0}}}u(.,t)^{r}dx\right)  ^{\frac{1}{r}}\leqq\left(
\int_{\mathbb{R}^{N}}u_{0}^{r}\xi^{\lambda}dx\right)  ^{\frac{1}{r}}+C\tau
R^{\frac{N}{r}-q^{\prime}}\leqq\left(  \int_{\mathbb{R}^{N}}u_{0}%
^{r}dx\right)  ^{\frac{1}{r}}+C\tau R^{\frac{N}{r}-q^{\prime}}.
\]
Finally (\ref{ghi}) follows as $R\rightarrow\infty$ and then as $R_{0}%
\rightarrow\infty$.$\medskip$

\textbf{Second step: case } $q^{\prime}\leqq N/r.$ Then $r<N$ and $q\geqq
N/(N-r)>1.$ Let us fix some $k\in\left(  1,N/(N-r)\right)  .$ For any $\eta
\in\left(  0,1\right)  $, we have $\eta|\nabla u|^{k}\leqq\eta+|\nabla
u|^{q},$ hence the function
\[
w_{\eta}=\eta^{1/(k-1)}(u-\eta t)
\]
satisfies
\[
(w_{\eta})_{t}-\nu\Delta w_{\eta}+|\nabla w_{\eta}|^{k}\leqq0
\]
in the weak sense. Thanks to Kato's inequality, see \cite{BrFr}, \cite{BaPi},
we deduce that
\begin{equation}
(w_{\eta}^{+})_{t}-\nu\Delta w_{\eta}^{+}+|\nabla w_{\eta}^{+}|^{k}\leqq0,
\label{II1.36}%
\end{equation}
in $\hspace{0.05in}\mathcal{D}^{\prime}(Q_{\mathbb{R}^{N},T}).$ And $w_{\eta
}^{+}$ has the same regularity as $u.$ Moreover it satisfies an analogous
property to (\ref{int}):
\begin{equation}
\lim_{t\rightarrow0}\int_{\mathbb{R}^{N}}(w_{\eta}^{+})^{r}(.,t)\psi
dx=\int_{\mathbb{R}^{N}}(\eta^{1/(k-1)}u_{0})^{r}\psi dx,\qquad\forall\psi\in
C_{c}(\mathbb{R}^{N}). \label{mol}%
\end{equation}
Indeed
\begin{align*}
\left\vert \int_{\mathbb{R}^{N}}((u-\eta t)^{+})^{r}-u^{r}(.,t))\psi
dx\right\vert  &  \leqq\int_{\left\{  u\geqq\eta t\right\}  }\left\vert
(u(.,t)-\eta t)^{r}-u^{r}(.,t)\right\vert \psi dx+\int_{\left\{  u\leqq\eta
t\right\}  }u^{r}(.,t))\psi dx\\
&  \leqq r\eta t\int_{\mathbb{R}^{N}}u^{r-1}(.,t)\psi dx+C(\psi)t^{r}\\
&  \leqq r\eta t(\int_{\mathbb{R}^{N}}u^{r}(.,t)\psi dx)^{1/r^{\prime}}%
(\int_{\mathbb{R}^{N}}\psi dx)^{1/r}+C(\psi)t^{r}%
\end{align*}
then
\[
\lim_{t\rightarrow0}\int_{\mathbb{R}^{N}}((u-\eta t)^{+})^{r}-u^{r}(.,t))\psi
dx=0,
\]
and (\ref{mol}) follows from (\ref{int}) applied to $\eta^{1/(k-1)}u.$ Since
$k^{\prime}>N/r,$ we can apply the first step to $w_{\eta}^{+};$ we deduce
that $w_{\eta}^{+}(t)\in L^{r}\left(  \mathbb{R}^{N}\right)  $ and
\[
\int_{\mathbb{R}^{N}}(w_{\eta}^{+})^{r}(.,t)dx\leqq\eta^{\frac{r}{k-1}}%
\int_{\mathbb{R}^{N}}u_{0}^{r}dx.
\]
Then $\left\Vert (u-\eta t)^{+}\right\Vert _{L^{r}\left(  \mathbb{R}%
^{N}\right)  }\leqq\left\Vert u_{0}\right\Vert _{L^{r}\left(  \mathbb{R}%
^{N}\right)  }.$ Since $u\leqq\eta t+(u-\eta t)^{+},$ we find, for any $R>0,$
\[
\left\Vert u(.,t)\right\Vert _{L^{r}\left(  B_{R}\right)  }\leqq\left\Vert
u_{0}\right\Vert _{L^{r}\left(  \mathbb{R}^{N}\right)  }+\eta t\left\vert
B_{R}\right\vert ^{\frac{1}{r}}%
\]
As $\eta\rightarrow0$ we get $\left\Vert u(.,t)\right\Vert _{L^{r}\left(
B_{R}\right)  }\leqq\left\Vert u_{0}\right\Vert _{L^{r}\left(  \mathbb{R}%
^{N}\right)  },$ then as $R\rightarrow\infty$ we deduce that $u(.,t)\in
L^{r}\left(  \mathbb{R}^{N}\right)  ,$ and (\ref{ghi}) holds$.\medskip$

(ii) Consider again $0<s<\tau<T$ and $u_{\varepsilon,\delta}$ as above.
Setting $F_{\varepsilon}=\left\vert \nabla u\right\vert ^{q}\ast
\varrho_{\varepsilon},$ there holds
\[
(u_{\varepsilon,\delta})_{t}-\nu\Delta u_{\varepsilon,\delta}+F_{\varepsilon
}=0.
\]
Then for any $t\in\left[  s,\tau\right]  ,$
\begin{align*}
&  \frac{d}{dt}\left(  \int_{\mathbb{R}^{N}}u_{\varepsilon,\delta}^{r}%
(.,t)\xi^{\lambda}dx\right)  +r(r-1)\nu\int_{\mathbb{R}^{N}}u_{\varepsilon
,\delta}^{r-2}\left\vert \nabla u_{\varepsilon,\delta}\right\vert ^{2}%
(.,t)\xi^{\lambda}dx+r\int_{\mathbb{R}^{N}}F_{\varepsilon}u_{\varepsilon
,\delta}^{r-1}(.,t)\xi^{\lambda}dx\\
&  =-r\nu\int_{\mathbb{R}^{N}}u_{\varepsilon,\delta}^{r-1}\nabla
u_{\varepsilon,\delta}(.,t).\nabla(\xi^{\lambda})dx=\nu\int_{\mathbb{R}^{N}%
}u_{\varepsilon,\delta}^{r}(.,t)\Delta(\xi^{\lambda})dx
\end{align*}
thus
\begin{subequations}
\begin{align}
&  \int_{\mathbb{R}^{N}}u_{\varepsilon,\delta}^{r}(.,t)\xi^{\lambda}%
dx+r\int_{s}^{t}\int_{\mathbb{R}^{N}}u_{\varepsilon,\delta}^{r-1}%
F_{\varepsilon}\xi^{\lambda}dxdt\nonumber\\
+r(r-1)\nu\int_{s}^{t}\int_{\mathbb{R}^{N}}u_{\varepsilon,\delta}%
^{r-2}\left\vert \nabla u_{\varepsilon,\delta}\right\vert ^{2}\xi^{\lambda}dx
&  =\int_{\mathbb{R}^{N}}u_{\varepsilon,\delta}^{r}(.,s)\xi^{\lambda}%
dx+\nu\int_{s}^{t}\int_{\mathbb{R}^{N}}u_{\varepsilon,\delta}^{r}\Delta
(\xi^{\lambda})dx\nonumber
\end{align}
First we go to the limit as $\varepsilon\rightarrow0,$ because $u\in
L_{loc}^{\infty}(Q_{\mathbb{R}^{N},T}),$ and $\left\vert \nabla u\right\vert
^{2}$ $\in L_{loc}^{1}(Q_{\mathbb{R}^{N},T}),$ and $F_{\varepsilon}$ converges
to $\left\vert \nabla u\right\vert ^{q}$ in $L_{loc}^{1}(Q_{\mathbb{R}^{N}%
,T}).$ Setting $v_{\delta}=u+\delta,$ we obtain for almost any $s,t,$ and by
continuity for any $0<s<t\leqq\tau$,
\end{subequations}
\begin{subequations}
\begin{align}
&  \int_{\mathbb{R}^{N}}v_{\delta}^{r}(.,t)\xi^{\lambda}dx+r\int_{s}^{t}%
\int_{\mathbb{R}^{N}}v_{\delta}^{r-1}\left\vert \nabla u\right\vert ^{q}\psi
dxdt\nonumber\\
+r(r-1)\nu\int_{s}^{t}\int_{\mathbb{R}^{N}}v_{\delta}^{r-2}\left\vert \nabla
u\right\vert ^{2}\xi^{\lambda}dx  &  =\int_{\mathbb{R}^{N}}v_{\delta}%
^{r}(.,s)\xi^{\lambda}dx+\nu\int_{s}^{t}\int_{\mathbb{R}^{N}}v_{\delta}%
^{r}\Delta(\xi^{\lambda})dx\nonumber
\end{align}
Next we go to the limit as $\delta\rightarrow0:$ from the Fatou Lemma,
$\int_{s}^{t}\int_{\mathbb{R}^{N}}u^{r-1}\left\vert \nabla u\right\vert
^{q}\psi dxdt$ and $(r-1)\nu\int_{s}^{t}\int_{\mathbb{R}^{N}}u^{r-2}\left\vert
\nabla u\right\vert ^{2}\xi^{\lambda}dx$ are finite, and then from the
dominated convergence theorem,%
\end{subequations}
\begin{align*}
&  \int_{\mathbb{R}^{N}}u^{r}(.,t)\xi^{\lambda}dx+r\int_{s}^{t}\int%
_{\mathbb{R}^{N}}u^{r-1}\left\vert \nabla u\right\vert ^{q}\xi^{\lambda}dxdt\\
+r(r-1)\nu\int_{s}^{t}\int_{\mathbb{R}^{N}}u^{r-2}\left\vert \nabla
u\right\vert ^{2}\xi^{\lambda}dx  &  =\int_{\mathbb{R}^{N}}u^{r}(.,\sigma
)\xi^{\lambda}dx+\nu\int_{s}^{t}\int_{\mathbb{R}^{N}}u^{r}\Delta(\xi^{\lambda
})dx.
\end{align*}
As $s\rightarrow0,$ from (\ref{int}), we deduce that
\begin{align*}
&  \int_{\mathbb{R}^{N}}u^{r}(.,t)\xi^{\lambda}dx+r\int_{0}^{t}\int%
_{\mathbb{R}^{N}}u^{r-1}\left\vert \nabla u\right\vert ^{q}\xi^{\lambda}dxdt\\
+r(r-1)\int_{0}^{t}\int_{\mathbb{R}^{N}}u^{r-2}\left\vert \nabla u\right\vert
^{2}\xi^{\lambda}dx  &  =\int_{\mathbb{R}^{N}}u_{0}^{r}(.,\sigma)\xi^{\lambda
}dx+\nu\int_{\sigma}^{t}\int_{\mathbb{R}^{N}}u^{r}\Delta(\xi^{\lambda})dx
\end{align*}
Now $u(.,t)\in L^{r}\left(  \mathbb{R}^{N}\right)  $ for any $t\in\left[
s,\tau\right]  $, and
\[
\int_{\sigma}^{t}\int_{\mathbb{R}^{N}}u^{r}\Delta(\xi^{\lambda})dx\leqq
\frac{C}{R^{2}}\tau\int_{\mathbb{R}^{N}}u^{r}(\sigma)dx,
\]
thus we can make $R\rightarrow\infty$. Then $\int_{0}^{\tau}\int%
_{\mathbb{R}^{N}}u^{r-1}|\nabla u|^{q}dxdt$ and $(r-1)\nu\int_{0}^{\tau}%
\int_{\mathbb{R}^{N}}u^{r-2}|\nabla u|^{2}dxdt$ are finite and, from the
dominated convergence theorem,
\begin{equation}
\int_{\mathbb{R}^{N}}u^{r}(.,t)dx+r\int_{0}^{t}\int_{\mathbb{R}^{N}}%
u^{r-1}|\nabla u|^{q}dxdt+r(r-1)\nu\int_{0}^{t}\int_{\mathbb{R}^{N}}%
u^{r-2}\left\vert \nabla u\right\vert ^{2}dx=\int_{\mathbb{R}^{N}}u_{0}^{r}dx
\label{lon}%
\end{equation}
Hence (\ref{egge}) and (\ref{agga}) follow, implying (\ref{limu}).\medskip

(iii) Setting $v=u^{b}$ with $b=(q-1+r)/q\leqq$ $r,$ there holds $|\nabla
v|^{q}\in L_{loc}^{1}(\left[  0,T\right)  ;L^{1}(\mathbb{R}^{N})),$ and $v\in
L^{\infty}((0,T);L^{r/b}\left(  \mathbb{R}^{N}\right)  ).$ From the
Gagliardo-Nirenberg inequality,
\begin{equation}
\left\Vert v(.,t)\right\Vert _{L^{q}\left(  \mathbb{R}^{N}\right)  }\leqq
C(N,q,r)\left\Vert v(.,t)\right\Vert _{L^{\frac{r}{b}}\left(  \mathbb{R}%
^{N}\right)  }^{1-k}\left\Vert \nabla v(.,t)\right\Vert _{L^{q}\left(
\mathbb{R}^{N}\right)  }^{k},\qquad\frac{1}{k}=1+\frac{rq^{\prime}}{N}.
\label{gal}%
\end{equation}
By integration, for any $0<\tau<T,$ we get, from H\"{o}lder inequality, with
$C=C((\tau,N,q,r),$
\[
\int_{0}^{\tau}\int_{\mathbb{R}^{N}}v^{q}(.,t)dxdt=\int_{0}^{\tau}%
\int_{\mathbb{R}^{N}}u^{q-1+r}(.,t)dxdt\leqq C\left\Vert v\right\Vert
_{L^{\infty}((0,\tau);L^{\frac{r}{m}}\left(  \mathbb{R}^{N}\right)  }%
^{(1-k)q}(\int_{0}^{\tau}\int_{\mathbb{R}^{N}}\left\vert \nabla v\right\vert
^{q}dxdt)^{k}.
\]
Then $u\in L^{q-1+r}(Q_{\mathbb{R}^{N},\tau}),$ and $v^{q}=u^{q-1+r}\in
L^{1}((0,\tau);W^{1,1}\left(  \mathbb{R}^{N}\right)  ),$ $v\in L^{q}%
((0,\tau);W^{1,q}\left(  \mathbb{R}^{N}\right)  )$. If $\nu>0,$ we also have
$u^{r-2}\left\vert \nabla u\right\vert ^{2}=\left\vert \nabla(u^{r/2}%
)\right\vert ^{2}\in L^{1}(Q_{\mathbb{R}^{N},\tau}),$ and $u^{r/2}\in
L^{2}(Q_{\mathbb{R}^{N},\tau}),$ then $u^{r/2}\in L^{2}((0,\tau);W^{1,2}%
\left(  \mathbb{R}^{N}\right)  ).$\medskip

(iv) Here $u\in C(\left[  0,T\right)  ;L_{loc}^{r}\left(  \mathbb{R}%
^{N}\right)  )$. We only need to prove that $\lim_{t\rightarrow0}\left\Vert
u(.,t)-u_{0}\right\Vert _{L^{r}\left(  \mathbb{R}^{N}\right)  }=0.$ From a
diagonal procedure, there exists $t_{n}\rightarrow0$ such that $\left(
u(.,t_{n})\right)  $ converges to $u_{0}$ a.e. in $\mathbb{R}^{N}.$ First
assume $r>1;$ since the convergence holds weakly in $L^{r}\left(
\mathbb{R}^{N}\right)  ,$ and $\lim_{n\rightarrow\infty}\left\Vert
u(.,t_{n})\right\Vert _{L^{r}\left(  \mathbb{R}^{N}\right)  }=\left\Vert
u_{0}\right\Vert _{L^{r}\left(  \mathbb{R}^{N}\right)  }$ from (\ref{limu}).
Then it holds from any sequence, and $u\in C(\left[  0,T\right)  ;L^{r}\left(
\mathbb{R}^{N}\right)  )$. Next assume $r=1.$ We have for any $p>0,$
\begin{align*}
\int_{\mathbb{R}^{N}}\left\vert u(t_{n})-u_{0}\right\vert dx  &  \leqq
\int_{B_{p}}\left\vert u(t_{n})-u_{0}\right\vert dx+\int_{\mathbb{R}%
^{N}\backslash B_{p}}\left\vert u(t_{n})-u_{0}\right\vert dx\\
&  \leqq\int_{B_{p}}\left\vert u(t_{n})-u_{0}\right\vert dx+\int%
_{\mathbb{R}^{N}\backslash B_{p}}u(t_{n})dx+\int_{\mathbb{R}^{N}\backslash
B_{p}}u_{0}dx\\
&  =\int_{B_{p}}\left\vert u(t_{n})-u_{0}\right\vert dx+\int_{\mathbb{R}^{N}%
}u(t_{n})dx-\int_{B_{p}}u_{0}dx\\
&  -\int_{B_{p}}(u(t_{n})-u_{0})dx+\int_{\mathbb{R}^{N}\backslash B_{p}}%
u_{0}dx\\
&  \leqq2\int_{B_{p}}\left\vert u(t_{n})-u_{0}\right\vert dx+\int%
_{\mathbb{R}^{N}}u(t_{n})dx-\int_{\mathbb{R}^{N}}u_{0}dx+2\int_{\mathbb{R}%
^{N}\backslash B_{p}}u_{0}dx.
\end{align*}
The result follows from (\ref{limu}), because $u_{0}\in L^{1}\left(
\mathbb{R}^{N}\right)  .\medskip$
\end{proof}

The decay result is also available for initial data measures, where we do not
assume that $q<(N+2)/(N+1):$

\begin{theorem}
\label{inmea}Assume $q>1.$ Let $u_{0}\in\mathcal{M}_{b}^{+}(\mathbb{R}^{N})$
and $u$ be any non-negative weak $\mathcal{M}_{loc}$ solution of equation
(\ref{cau}) in$\hspace{0.05in}Q_{\mathbb{R}^{N},T}.$ Then $u(.,t)\in
L^{1}\left(  \mathbb{R}^{N}\right)  $ for any $t>0,$ and%
\begin{equation}
\int_{\mathbb{R}^{N}}u(.,t)dx\leqq%
{\displaystyle\int_{\mathbb{R}^{N}}}
du_{0}. \label{mis}%
\end{equation}
Moreover $u\in C((0,T);L^{1}\left(  \mathbb{R}^{N}\right)  ),$ $|\nabla
u|^{q}\in L_{loc}^{1}(\left[  0,T\right)  ;L^{1}\left(  \mathbb{R}^{N}\right)
)$ and
\begin{equation}
\int_{\mathbb{R}^{N}}u(.,t)dx+\int_{0}^{t}\int_{\mathbb{R}^{N}}|\nabla
u|^{q}dxdt=\,%
{\displaystyle\int_{\mathbb{R}^{N}}}
du_{0}, \label{his}%
\end{equation}
and
\begin{equation}
\lim_{t\rightarrow0}\int_{\mathbb{R}^{N}}u(.,t)\varphi dx=%
{\displaystyle\int_{\mathbb{R}^{N}}}
\varphi du_{0},\qquad\forall\varphi\in C_{b}(\mathbb{R}^{N}). \label{ris}%
\end{equation}

\end{theorem}

\begin{proof}
If $q^{\prime}<N,$ we obtain in the same way (\ref{fis}) with $r=1$, and we go
to the limit as $s\rightarrow0$ from (\ref{mea}), then
\[
\int_{B_{R_{0}}}u(.,t)dx\leqq%
{\displaystyle\int_{\mathbb{R}^{N}}}
\xi^{\lambda}du_{0}+C\tau R^{N-q^{\prime}}\leqq%
{\displaystyle\int_{\mathbb{R}^{N}}}
du_{0}+C\tau R^{N-q^{\prime}}.
\]
Going to the limit as $R\rightarrow\infty,$ and then as $R_{0}\rightarrow
\infty,$ we deduce (\ref{mis}). If $q^{\prime}\geqq N,$ we proceed as in the
second step of Theorem \ref{decay}, and get again (\ref{mis}). Then
(\ref{his}) follows. And $u\in C((0,T);L^{1}\left(  \mathbb{R}^{N}\right)  ),$
from the dominated convergence theorem, because $u\in C((0,T);L_{loc}%
^{1}\left(  \mathbb{R}^{N}\right)  ),$ and $u\in L^{\infty}((0,T);L^{1}\left(
\mathbb{R}^{N}\right)  ).\medskip$

Let us show (\ref{ris}): let $\varphi\in C_{b}(\mathbb{R}^{N})$ be
nonnegative, we can assume that $\varphi$ takes its values in $\left[
0,1\right]  .$ Let $t_{n}\rightarrow0.$ We know that $\lim_{n\rightarrow
\infty}\int_{\mathbb{R}^{N}}u(.,t_{n})dx=%
{\displaystyle\int_{\mathbb{R}^{N}}}
du_{0}.$ Let $\psi_{p}\in\mathcal{D}(\mathbb{R}^{N})$ with values in $\left[
0,1\right]  ,$ $\psi_{p}(x)=1$ if $\left\vert x\right\vert \leqq p,$ $0$ if
$\left\vert x\right\vert \geqq2p.$ Then $\lim_{p\rightarrow\infty}%
{\displaystyle\int_{\mathbb{R}^{N}}}
(1-\psi_{p})du_{0}=0$, from the dominated convergence Theorem. Thus for any
$\eta>0,$ one can choose $p_{\eta}$ such that $%
{\displaystyle\int_{\mathbb{R}^{N}}}
(1-\psi_{p_{\eta}})du_{0}\leqq\eta$; and
\begin{align*}
&  \overline{\lim}\left\vert \int_{\mathbb{R}^{N}}u(.,t_{n})\varphi dx-%
{\displaystyle\int_{\mathbb{R}^{N}}}
\varphi du_{0}\right\vert \\
&  \leqq\lim_{n\rightarrow\infty}\left\vert \int_{\mathbb{R}^{N}}%
u(.,t_{n})\varphi\psi_{p_{\eta}}dx-%
{\displaystyle\int_{\mathbb{R}^{N}}}
\varphi\psi_{p_{\eta}}du_{0}\right\vert +%
{\displaystyle\int_{\mathbb{R}^{N}}}
\varphi(1-\psi_{p_{\eta}})du_{0}+\lim_{n\rightarrow\infty}\int_{\mathbb{R}%
^{N}}u(.,t_{n})\varphi(1-\psi_{p_{\eta}})dx\leqq\eta,
\end{align*}
hence the conclusion follows.\medskip
\end{proof}

\subsection{Regularizing effects}

Here we deduce of the decay estimates a regularizing effect \textit{without
any condition at}\textbf{ }$\infty,$ ending the proof of Theorem \ref{main}.

\begin{theorem}
\label{bound}Let $q>1.$ Let $r\geqq1$ and $u_{0}\in L^{r}(\mathbb{R}^{N})$.
Let $u$ be any non-negative weak $L_{loc}^{r}$ solution of problem (\ref{cau})
in$\hspace{0.05in}Q_{\mathbb{R}^{N},T}$ (\ref{int}). $\medskip$

Then $u(.,t)\in L^{\infty}(\mathbb{R}^{N})$ for any $t\in$ $(0,T)$ and $u$
satisfies the estimates (\ref{uinf}), where $\sigma_{r,q,N},\varpi_{r,q,N}$
are given by (\ref{sigr}).$\medskip$

Moreover if $\nu>0,$ then $u$ satifies the estimates (\ref{use}). If $u_{0}%
\in\mathcal{M}_{b}^{+}(\mathbb{R}^{N})$, the same results hold, where $\Vert
u_{0}\Vert_{L^{1}(\mathbb{R}^{N})}$ is replaced by $%
{\displaystyle\int_{\mathbb{R}^{N}}}
du_{0}.$
\end{theorem}

\begin{proof}
Since $u$ is a weak $L_{loc}^{r}$ solution, then $u\in C(\left[  0,T\right)
;L^{r}\left(  \mathbb{R}^{N}\right)  ),$ from Theorem \ref{decay}. Thus for
any $0\leqq s<T,$ $u$ is a weak $r$ solution in $Q_{\mathbb{R}^{N},s,T}$; and
$\int_{\mathbb{R}^{N}}u^{r}(s)dx<\infty$ with $r\geqq1$. For any $0<s\leqq
t<T,$ and any $\alpha\geqq r-1$ such that $\int_{\mathbb{R}^{N}}u^{\alpha
+1}(s)dx<\infty,$ we can apply Theorem \ref{decay} to $u$ starting at point
$s,$ because of (\ref{2.1}). Denoting $\beta=1+\alpha/q,$ we have
\begin{equation}
\frac{1}{\alpha+1}\int_{\mathbb{R}^{N}}u^{\alpha+1}(.,t)dx+\frac{1}{\beta^{q}%
}\int_{s}^{t}\int_{\mathbb{R}^{N}}|\nabla(u^{\beta})|^{q}dxdt\leqq\frac
{1}{\alpha+1}\int_{\mathbb{R}^{N}}u^{\alpha+1}(.,s)dx, \label{miss}%
\end{equation}
and $u^{\beta}(.,t)\in L^{q}(\mathbb{R}^{N})$ for almost any $t\in\left(
0,T\right)  $. \medskip

(i) Proof of (\ref{uinf}).\medskip\ 

First assume $q<N.$ From the Sobolev injection of $W^{1,q}\left(
\mathbb{R}^{N}\right)  $ into $L^{Nq/(N-q)}\left(  \mathbb{R}^{N}\right)  $,
there holds
\[
\frac{1}{\alpha+1}\int_{\mathbb{R}^{N}}u^{\alpha+1}(.,t)dx+\frac{C(N,q)}%
{\beta^{q}}\int_{s}^{t}(\int_{\mathbb{R}^{N}}u^{\beta\frac{Nq}{N-q}%
}(.,t)dx)^{\frac{N-q}{N}})dt\leqq\frac{1}{\alpha+1}\int_{\mathbb{R}^{N}%
}u^{\alpha+1}(.,s)dx;
\]
thus Lemma \ref{prod} applies with $m=q$ and $\theta=N/(N-q).$ We obtain
\[
\Vert u(.,t)\Vert_{L^{\infty}(\mathbb{R}^{N})}\leqq C(t-s)^{-\sigma_{r,q,N}%
}\Vert u(.,s)\Vert_{L^{r}(\mathbb{R}^{N})}^{\varpi_{r,q,N}},\qquad C=C(N,q,r),
\]
and deduce (\ref{uinf}) as $s$ goes to $0$.

If $q=N,$ we deduce (\ref{uinf}) from Lemma \ref{prod} with $\theta>1$
arbitrary, since $W^{1,N}\left(  \mathbb{R}^{N}\right)  \subset$ $L^{N\theta
}\left(  \mathbb{R}^{N}\right)  $.

Next assume $q>N$. We straight away obtain, for any $t\in\left(  0,T\right)
,$%
\begin{equation}
\int_{0}^{t}\int_{\mathbb{R}^{N}}|\nabla(u^{\beta})|^{q}dxdt\leqq\frac{1}%
{r}\int_{\mathbb{R}^{N}}u_{0}^{r}dx,\label{mic}%
\end{equation}
with $\beta=1+(r-1)/q.$ From the Sobolev injection $W^{1,q}\left(
\mathbb{R}^{N}\right)  \subset$ $L^{\infty}\left(  \mathbb{R}^{N}\right)  ,$
$u(.,s)\in L^{\infty}\left(  \mathbb{R}^{N}\right)  $ for almost any
$s\in(0,T),$ hence $u(.,s)\in L^{\rho}\left(  \mathbb{R}^{N}\right)  )$ for
any $\rho\geqq r,$ and $u\in C(\left[  s,t\right)  ,L^{\rho}\left(
\mathbb{R}^{N}\right)  )$ from (\ref{2.1}). In turn $u(.,t)\in L^{\infty
}\left(  \mathbb{R}^{N}\right)  $ for any $t\in(0,T)$ and $t\mapsto\left\Vert
u(.,t)\right\Vert _{L^{\infty}\left(  \mathbb{R}^{N}\right)  }$ is
nonincreasing, thus
\[
rt\left\Vert u(.,t)\right\Vert _{L^{\infty}\left(  \mathbb{R}^{N}\right)
}^{q+r-1}\leqq C(N,q)\int_{\mathbb{R}^{N}}u_{0}^{r}dx.
\]
This \textit{does not give the optimal estimate} (\ref{sigr}). However from
Lemma \ref{gani}, $v=u^{\beta}$ satisfies the Gagliardo-Nirenberg inequality,
for almost any $t\in(0,T),$
\[
\left\Vert v(.,t)\right\Vert _{L^{\infty}\left(  \mathbb{R}^{N}\right)  }\leqq
C\left\Vert v(.,t)\right\Vert _{L^{\frac{r}{\beta}}\left(  \mathbb{R}%
^{N}\right)  }^{1-k}\left\Vert \nabla v(.,t)\right\Vert _{L^{q}\left(
\mathbb{R}^{N}\right)  }^{k},
\]
where $1/k=1+(1/N-1/q)r/\beta$ and $C=C(N,q,r).$ Then
\[
\left\Vert u(.,t)\right\Vert _{L^{\infty}\left(  \mathbb{R}^{N}\right)
}^{\frac{\beta q}{k}}\leqq C\left\Vert u(.,t)\right\Vert _{L^{r}\left(
\mathbb{R}^{N}\right)  }^{\frac{\beta q(1-k)}{k}}\int_{\mathbb{R}^{N}}%
|\nabla(u^{\beta})|^{q}dxdt\leqq C\left\Vert u_{0}\right\Vert _{L^{r}\left(
\mathbb{R}^{N}\right)  }^{\frac{\beta q(1-k)}{k}}\int_{\mathbb{R}^{N}}%
|\nabla(u^{\beta})|^{q}dxdt.
\]
By integration, using (\ref{mic}), we find
\[
rt\left\Vert u(.,t)\right\Vert _{L^{\infty}\left(  \mathbb{R}^{N}\right)
}^{\frac{\beta q}{k}}\leqq C\left\Vert u_{0}\right\Vert _{L^{r}\left(
\mathbb{R}^{N}\right)  }^{\frac{\beta q(1-k)}{k}+r},
\]
which gives precisely (\ref{uinf}), since $k/\beta q=\sigma_{r,q,N}$ and
$(1-k)+kr/\beta q=\varpi_{r,q,N}.\medskip$

(ii) Proof of (\ref{use}).

First assume $N>2.$ For any $\alpha\geqq r-1$ such that $\int_{\mathbb{R}^{N}%
}u^{\alpha+1}(s)dx<\infty,$
\[
\frac{1}{\alpha+1}\int_{\mathbb{R}^{N}}u^{\alpha+1}(t)dx+\frac{\alpha}%
{\tilde{\beta}^{2}}\nu\int_{s}^{t}\int_{\mathbb{R}^{N}}\left\vert
\nabla(u^{\tilde{\beta}})\right\vert ^{2}dxdt\leqq\frac{1}{\alpha+1}%
\int_{\mathbb{R}^{N}}u^{\alpha+1}(s)dx
\]
where $\tilde{\beta}=(\alpha+1)/2;$ and $u^{\tilde{\beta}}\in L_{loc}%
^{2}((0,\tau);W^{1,2}\left(  \mathbb{R}^{N}\right)  ).$ From the Sobolev
injection of $W^{1,2}\left(  \mathbb{R}^{N}\right)  $ into $L^{2N/(N-2)}%
\left(  \mathbb{R}^{N}\right)  $, we get
\[
\frac{1}{\alpha+1}\int_{\mathbb{R}^{N}}u^{\alpha+1}(t)dx+\frac{\alpha
C(N)}{\tilde{\beta}^{2}}\nu\int_{s}^{t}(\int_{\mathbb{R}^{N}}u^{\tilde{\beta
}\frac{2N}{N-2}})^{\frac{N-2}{N}}dx\leqq\frac{1}{\alpha+1}\int_{\mathbb{R}%
^{N}}u^{\alpha+1}(s)dx.
\]
In case $r>1,$ Lemma \ref{prod} applies with $C_{0}=(r-1)C(N)\nu,$
$q=2,\theta=N/(N-2)$ and $\lambda=-1,$ $\tilde{\beta}=1+(\alpha-1)/2$, since
$r>N(1-2+1)/2;$ and (\ref{use}) follows.

\noindent In case $r=1,$ then $u\in C(\left[  0,T\right)  ;L^{1}%
(\mathbb{R}^{N}))\cap L_{loc}^{\infty}((0,T);L^{\infty}(\mathbb{R}^{N}))$
\textit{because of estimate} (\ref{uinf}). Hence $C(\left[  0,T\right)
;L^{\rho}(\mathbb{R}^{N}))$ for any $\rho>1,$ for example with $\rho=2,$ and
$\Vert u(.,t)\Vert_{L^{1}(\mathbb{R}^{N})}$ is nonincreasing, from Theorem
\ref{decay}. Therefore Lemma \ref{prod} applies on $\left(  \epsilon,t\right)
$ for $0<\epsilon<t<T:$
\[
\Vert u(.,t)\Vert_{L^{\infty}(\mathbb{R}^{N})}\leqq C(t-\epsilon)^{-\frac
{N}{2}}\Vert u(.,\epsilon)\Vert_{L^{1}(\mathbb{R}^{N})}\leqq C(t-\epsilon
)^{-\frac{N}{2}}\Vert u_{0}\Vert_{L^{1}(\mathbb{R}^{N})},
\]
with $C=C(N,q,r,\nu),$ hence (\ref{use}) follows as $\epsilon\rightarrow
0.\medskip$

If $N=2,$ we proceed as above to conclude. Next assume $N=1$. In case $r>1,$
there holds, for any $t\in(0,T),$
\[
4(r-1)\nu\int_{0}^{t}\int_{\mathbb{R}^{N}}|\nabla(u^{\frac{r}{2}}%
)|^{2}dxdt\leqq\int_{\mathbb{R}^{N}}u_{0}^{r}dx;
\]
and, from Lemma \ref{gani}, applied to $v=u^{r/2},$ with $m=2=1/k,$
\[
\left\Vert u(.,t)\right\Vert _{L^{\infty}\left(  \mathbb{R}\right)  }%
^{2r}\leqq C\left\Vert u(.,t)\right\Vert _{L^{r}\left(  \mathbb{R}\right)
}^{r}\int_{\mathbb{R}}|\nabla(u^{\frac{r}{2}})|^{2}dxdt\leqq C\left\Vert
u_{0}\right\Vert _{L^{r}\left(  \mathbb{R}^{N}\right)  }^{r}\int_{\mathbb{R}%
}|\nabla(u^{\frac{r}{2}})|^{2}dxdt;
\]
by integration, we get, with a new constant $C=C(r,\nu),$
\[
t\left\Vert u(.,t)\right\Vert _{L^{\infty}\left(  \mathbb{R}\right)  }%
^{2r}\leqq C\left\Vert u_{0}\right\Vert _{L^{r}\left(  \mathbb{R}\right)
}^{2r},
\]
which proves (\ref{use}). In case $r=1,$ taking $\rho=2$ as above, we obtain,
for any $0<\epsilon<t<T,$
\[
\Vert u(.,t)\Vert_{L^{\infty}(\mathbb{R})}\leqq C(\nu)(t-\epsilon)^{-\frac
{1}{4}}\Vert u(.,\epsilon)\Vert_{L^{2}(\mathbb{R})}\leqq C(\nu)(t-\epsilon
)^{-\frac{1}{4}}\Vert u(.,\epsilon)\Vert_{L^{\infty}(\mathbb{R})}^{\frac{1}%
{2}}\Vert u(.,\epsilon)\Vert_{L^{1}(\mathbb{R})}^{\frac{1}{2}}.
\]
From Lemma \ref{elem}, we deduce
\[
\Vert u(.,t)\Vert_{L^{\infty}(\mathbb{R})}\leqq C(\nu)(t-\epsilon)^{-\frac
{1}{2}}\Vert u_{0}\Vert_{L^{1}(\mathbb{R})},
\]
and we conclude as $\epsilon\rightarrow0.\medskip$

If $u_{0}\in\mathcal{M}_{b}^{+}(\mathbb{R}^{N}),$ we apply the estimates on
$\left(  \epsilon,T\right)  $ and go to the limit as $\epsilon\rightarrow0.$
\end{proof}

\begin{remark}
As a consequence, for any $k\geqq1,$ and for example $q\neq N,N\neq2,$%
\begin{equation}
\Vert u(.,t)\Vert_{L^{kr}(\mathbb{R}^{N})}\leqq Ct^{-\frac{\sigma_{r,q,N}%
}{k^{\prime}}}\Vert u_{0}\Vert_{L^{r}(\mathbb{R}^{N})}^{\frac{\varpi_{r,q,N}%
}{k^{\prime}}+\frac{1}{k}},\label{uk}%
\end{equation}%
\begin{equation}
\Vert u(.,t)\Vert_{L^{kr}(\mathbb{R}^{N})}\leqq Ct^{-\frac{N}{2rk^{\prime}}%
}\Vert u_{0}\Vert_{L^{r}(\mathbb{R}^{N})},\qquad\text{if }\nu>0.\label{uke}%
\end{equation}
Indeed it follows from (\ref{ghi}) and (\ref{uinf}), (\ref{use}) by
interpolation:
\[
\Vert u(.,t)\Vert_{L^{kr}(\mathbb{R}^{N})}\leqq\Vert u(.,t)\Vert_{_{L^{\infty
}(\mathbb{R}^{N})}}^{1/k^{\prime}}\Vert u(.,t)\Vert_{L^{r}(\mathbb{R}^{N}%
)}^{1/k}.
\]

\end{remark}

\begin{remark}
\label{q2}If $q\leqq2,$ then $u\in C^{2,1}\left(  Q_{\mathbb{R}^{N},T}\right)
$, thus \textbf{we do not need to introduce the regularization by
}$u_{\varepsilon};$ we only need to introduce $u+\delta,$ when $r>1$ and make
$\delta\rightarrow0.$
\end{remark}

\begin{remark}
\label{decreg}Up to now, the decay estimate (\ref{ghi}) and the $L^{\infty}$
estimate  of $u$ were proved for $u_{0}\in C_{b}\left(  \mathbb{R}^{N}\right)
\cap L^{r}\left(  \mathbb{R}^{N}\right)  ,$ and for the \textbf{unique bounded
solution} $u$ of problem (\ref{cau}), and based on the estimate (\ref{bela})
given in \cite[Theorem 5.6]{BASoWe}; indeed  from the classical inequality
\[
\Vert u(.,t)\Vert_{L^{\infty}(\mathbb{R}^{N})}\leqq C(N,r)\Vert\nabla
u(.,t)\Vert_{L^{\infty}(\mathbb{R}^{N})}^{\frac{N}{N+r}}\Vert u(.,t)\Vert
_{L^{r}(\mathbb{R}^{N})}^{\frac{r}{N+r}},
\]
and  (\ref{bela}), there holds, with $C=C(N,q,r),$%
\[
\Vert u(.,t)\Vert_{L^{\infty}(\mathbb{R}^{N})}\leqq Ct^{-\frac{N}{q(N+r)}%
}\Vert u(.,t)\Vert_{L^{\infty}(\mathbb{R}^{N})}^{\frac{N}{q(N+r)}}\Vert
u(.,t)\Vert_{L^{r}(\mathbb{R}^{N})}^{\frac{r}{N+r}},
\]%
\begin{equation}
\Vert u(.,t)\Vert_{L^{\infty}(\mathbb{R}^{N})}\leqq Ct^{-\sigma_{r,q,N}}\Vert
u_{0}\Vert_{L^{r}(\mathbb{R}^{N})}^{\varpi_{r,q,N}}.\label{vic}%
\end{equation}

\end{remark}

\subsection{Further estimates and convergence results for $q\leqq2.$}

Here we consider the case $1<q\leqq2.$ From the $L^{\infty}$ estimates above,
and the interior regularity of $u,$ we deduce new local estimates and
convergence results:

\begin{corollary}
\label{pass}Assume $1<q\leqq2.\medskip$

(i) Any nonnegative weak $L_{loc}^{r}$ solution (resp. $\mathcal{M}_{loc}$
solution) $u$ of problem (\ref{cau}) with initial data $u_{0}\in
L^{r}(\mathbb{R}^{N}),$ $r\geqq1$ (resp. $u_{0}\in\mathcal{M}_{b}%
^{+}(\mathbb{R}^{N})$) satisfies $u\in C^{2,1}(Q_{\mathbb{R}^{N},T})\cap
L_{loc}^{\infty}(\left(  0,T\right)  ;C_{b}(\mathbb{R}^{N})).\medskip$

(ii) Let $(u_{0,n})$ be any bounded sequence in $L^{r}(\mathbb{R}^{N}),$
$r\geqq1$ (resp. in $\mathcal{M}_{b}^{+}(\mathbb{R}^{N})).$ For any
$n\in\mathbb{N},$ let $u_{n}$ be any nonnegative weak $L_{loc}^{r}$ solution
(resp. $\mathcal{M}_{loc}$ solution) of problem (\ref{cau}) with initial data
$u_{0,n}.$ Then one can extract a subsequence converging in $C_{loc}%
^{2,1}(Q_{\mathbb{R}^{N},T})$ to a weak solution $u$ of (\ref{un}) in
$Q_{\mathbb{R}^{N},T}$.
\end{corollary}

\begin{proof}
From \cite[Theorem 2.9]{BiDao1} there there exists $\gamma\in\left(
0,1\right)  $ such that for any nonnegative weak solution of equation
(\ref{un}) $u$ in $Q_{\mathbb{R}^{N},T}$ and any ball $B_{R}\subset
\mathbb{R}^{N},$ and $0<s<\tau<T,$%
\[
\left\Vert u\right\Vert _{C^{2+\gamma,1+\frac{\gamma}{2}}(Q_{B_{R},s,\tau}%
)}\leqq C\Phi(\left\Vert u\right\Vert _{L^{\infty}(Q_{B_{2R},\frac{s}{2},\tau
})}).
\]
where $C=C(N,q,R,s,\tau)$ and $\Phi$ is a continuous increasing function. From
estimates (\ref{uinf}), we deduce that $u\in L_{loc}^{\infty
}(\left(  0,T\right)  ;C_{b}(\mathbb{R}^{N}))$ and
\begin{equation}
\left\Vert u\right\Vert _{C^{2+\gamma,1+\frac{\gamma}{2}}(Q_{B_{R},s,\tau}%
)}\leqq C\Phi(\left\Vert u_{0}\right\Vert _{L^{r}(\mathbb{R}^{N})}%
),\qquad\text{(resp. }\left\Vert u\right\Vert _{C^{2+\gamma,1+\frac{\gamma}%
{2}}(Q_{B_{R},s,\tau})}\leqq C\Phi(\int_{\mathbb{R}^{N}}du_{0}) \label{gff}%
\end{equation}
and the conclusions follow.$\medskip$
\end{proof}

We also deduce global gradient estimates in $\mathbb{R}^{N}:$

\begin{corollary}
\label{grad}Assume $\nu>0,$ $1<q\leqq2.$ (i) Let $u_{0}\in L^{r}%
(\mathbb{R}^{N}),$ $r\geqq1.$ Then any weak $L_{loc}^{r}$ solution $u$ of
problem (\ref{cau}) satisfies for $q\neq N$
\begin{equation}
\Vert\nabla u(.,t)\Vert_{L^{\infty}(\mathbb{R}^{N})}\leqq Ct^{-\vartheta
_{r,q,N}}\Vert u_{0}\Vert_{L^{r}(\mathbb{R}^{N})}^{\varkappa_{_{r,q,N}}%
},\label{sic}%
\end{equation}%
\[
\vartheta_{r,q,N}=\frac{N+r}{rq+N(q-1)},\qquad\varkappa_{_{r,q,N}}=\frac
{r}{rq+N(q-1)};
\]
and $\left\vert \nabla u\right\vert ^{q}\in L_{loc}^{\infty}((0,T);L^{r}%
(\mathbb{R}^{N})),$ and
\begin{equation}
\int_{\mathbb{R}^{N}}\left\vert \nabla u(.,t)\right\vert ^{qr}dx\leqq
Ct^{-r(\frac{q}{2}+\sigma_{r,q,N}(q-1))}\Vert u_{0}\Vert_{L^{r}(\mathbb{R}%
^{N})}^{(1+\varpi_{r,q,N}(q-1))r}\label{vou}%
\end{equation}
where $\sigma_{r,q,N},$ $\varpi_{r,q,N}$ are defined at (\ref{sigr}), and
$C=C(N,q,r,\nu).$ For $N\neq2,$ then
\begin{equation}
\Vert\nabla u(.,t)\Vert_{L^{\infty}(\mathbb{R}^{N})}\leqq Ct^{-\frac{1}%
{q}(\frac{N}{2r}+1)}\Vert u_{0}\Vert_{L^{r}(\mathbb{R}^{N})}^{\frac{1}{q}%
};\label{sac}%
\end{equation}%
\begin{equation}
\int_{\mathbb{R}^{N}}\left\vert \nabla u(.,t)\right\vert ^{qr}dx\leqq
Ct^{-r(\frac{q}{2}+\frac{N}{2r}(q-1))}\Vert u_{0}\Vert_{L^{r}(\mathbb{R}^{N}%
)}^{qr}.\label{voi}%
\end{equation}
If $N=2$,  estimates hold up to an $\varepsilon>0.$ Moreover if $q<2,$ $u$ is
a pointwise mild solution.\medskip

(ii) Let $u_{0}\in\mathcal{M}_{b}^{+}(\mathbb{R}^{N}).$ Then any weak
$\mathcal{M}_{loc}$ solution of (\ref{cau}) satisfies the same estimates as in
case $r=1,$ with $\Vert u_{0}\Vert_{L^{r}(\mathbb{R}^{N})}$ replaced by $%
{\displaystyle\int_{\mathbb{R}^{N}}}
du_{0}$.
\end{corollary}

\begin{proof}
(i) Let $u_{0}\in L^{r}(\mathbb{R}^{N}),$ $r\geqq1.$ Then for any
$\epsilon>0,$ $u(.,\epsilon)\in C_{b}(\mathbb{R}^{N}),$ from Corollary
\ref{pass}. From \cite{GiGuKe}, $u$ is the unique solution $v$ such that $v\in
C^{2,1}\left(  \mathbb{R}^{N}\times\left(  \epsilon,T\right)  \right)  \cap
C_{b}\left(  \mathbb{R}^{N}\times\left[  \epsilon,T\right)  \right)  $, and
$v(.,\epsilon)=u(.,\epsilon)$; since $v\in C_{b}^{2}\left(  \mathbb{R}%
^{N}\times(\epsilon,T\right)  ),$ we deduce that $u\in C_{b}^{2}\left(
\mathbb{R}^{N}\times(0,T\right)  );$ and for any $\epsilon\leqq t<T,$
\[
\Vert u(.,t)\Vert_{L^{\infty}(\mathbb{R}^{N})}\leqq\Vert u(.,\epsilon
)\Vert_{L^{\infty}(\mathbb{R}^{N})},\qquad\Vert\nabla u(.,t)\Vert_{L^{\infty
}(\mathbb{R}^{N})}\leqq\Vert\nabla u(.,\epsilon)\Vert_{L^{\infty}%
(\mathbb{R}^{N})},
\]
and from (\ref{bela}),
\begin{equation}
\left\vert \nabla u(.,t)\right\vert ^{q}\leqq C(q)(t-\epsilon)^{-1}%
u(.,t),\qquad\text{a.e. in }\mathbb{R}^{N}.\label{lis}%
\end{equation}
From the decay estimates, we also have $\Vert u(.,\epsilon)\Vert
_{L^{r}(\mathbb{R}^{N})}\leqq\Vert u_{0}\Vert_{L^{r}(\mathbb{R}^{N})}.$ And
$u(.,\epsilon)\in L^{\tilde{r}}(\mathbb{R}^{N})$ for any $\tilde{r}\in\left[
r,\infty\right]  ,$ and $u\in C(\left[  \epsilon,T\right)  ;L^{\tilde{r}%
}(\mathbb{R}^{N})).$ Going to the limit in (\ref{lis}) as $\epsilon
\rightarrow0,$ we deduce (\ref{sic}) from (\ref{uinf}), and (\ref{sac}) from
(\ref{use}), if $q\neq N$ or $N\neq2$. Moreover $\left\vert \nabla
u\right\vert ^{q}\in L_{loc}^{\infty}((0,T);L^{r}(\mathbb{R}^{N})),$ since
\[
\Vert\nabla u(.,t)\Vert_{L^{qr}(\mathbb{R}^{N})}\leqq C(q)t^{-\frac{1}{q}%
}\Vert u_{0}\Vert_{L^{r}(\mathbb{R}^{N})}^{\frac{1}{q}}.
\]
More precisely we get from estimate (\ref{cor}),
\[
\Vert\nabla(u^{\frac{1}{q^{\prime}}}(.,t)\Vert_{L^{\infty}(\mathbb{R}^{N}%
)}\leqq C(t-\epsilon)^{-\frac{1}{2}}\Vert u(.,\epsilon)\Vert_{L^{\infty
}(\mathbb{R}^{N})}^{\frac{1}{q^{\prime}}}%
\]
with $C=C(q,\nu);$ then from estimate (\ref{use}), for any $t\in\left(
0,T\right)  ,$ with other constants $C=C(q,\nu),$%
\[
\Vert\nabla(u^{\frac{1}{q^{\prime}}}(.,t)\Vert_{L^{\infty}(\mathbb{R}^{N}%
)}\leqq Ct^{-\frac{1}{2}}\Vert u(.,\frac{t}{2})\Vert_{L^{\infty}%
(\mathbb{R}^{N})}^{\frac{1}{q^{\prime}}}%
\]%
\[
\left\vert \nabla u(.,t)\right\vert ^{q}\leqq Ct^{-\frac{q}{2}}\Vert
u(.,\frac{t}{2})\Vert_{L^{\infty}(\mathbb{R}^{N})}^{q-1}u(.,t),
\]
then from estimate (\ref{uinf}) we get
\[
\int_{\mathbb{R}^{N}}\left\vert \nabla u(.,t)\right\vert ^{qr}dx\leqq C\Vert
u_{0}\Vert_{L^{r}(\mathbb{R}^{N})}^{\varpi_{r,q,N}(q-1)r}t^{-r(\frac{q}%
{2}+\sigma_{r,q,N}(q-1))}\int_{\mathbb{R}^{N}}u(.,t)^{r}dx;
\]
then (\ref{vou}) follows. And (\ref{voi}) follows from (\ref{use}). If  $N=2$,
in particular if $q=N,$ the same estimates hold  up to an $\varepsilon>0,$
from (\ref{uinf}) and  (\ref{use}).\medskip\ 

Next we prove that $u$ is a pointwise mild solution as $q<2.$ From
\cite[Theorem 6]{GiGuKe}, $u(.,t)\in C_{b}^{2}(\mathbb{R}^{N})$ for any
$t\in\left(  \epsilon,T\right)  ,$ in particular $u(.,2\epsilon)\in C_{b}%
^{2}(\mathbb{R}^{N}),$ then for any $t\geqq\epsilon,$ and any $x\in
\mathbb{R}^{N},$
\begin{equation}
u(x,t)=e^{(t-2\epsilon)\Delta}u(x,2\epsilon)-\int_{2\epsilon}^{t}%
\int_{\mathbb{R}^{N}}g(x-y,t-s)|\nabla u(y,s)|^{q}dyds, \label{brou}%
\end{equation}
see for example \cite[Proposition 4.2 ]{BeBALa}. But $u(x,2\epsilon)$
converges to $u_{0}$ in $L^{r}(\mathbb{R}^{N}),$ and then $e^{(t-2\epsilon
)\Delta}u(.,\epsilon)$ converges to $e^{t\Delta}u_{0}$ in $L^{r}%
(\mathbb{R}^{N}).$ Then we can go to the limit as $\epsilon\rightarrow0$ in
(\ref{brou}), for a.e. $x\in\mathbb{R}^{N}:$ the integral is convergent, then
the conclusion follows.\medskip

(ii) For Theorem \ref{inmea}, we have $u(.,t)\in L^{1}(\mathbb{R}^{N})$ for
$t\geqq\epsilon>0,$ which gives from (i)
\[
\Vert u(.,t)\Vert_{L^{\infty}(\mathbb{R}^{N})}\leqq C(t-\varepsilon
)^{-\sigma_{1,q,N}}\Vert u(.,\epsilon)\Vert_{L^{1}(\mathbb{R}^{N})}%
^{\varpi_{1,q,N}}\leqq C(t-\varepsilon)^{-\sigma_{1,q,N}}(%
{\displaystyle\int_{\mathbb{R}^{N}}}
du_{0})^{\varpi_{1,q}}.
\]
As $\epsilon\rightarrow0,$ we obtain (\ref{sic}), (\ref{sac}), (\ref{vou}) and
(\ref{voi}) hold with $r=1$ and $\Vert u_{0}\Vert_{L^{1}(\mathbb{R}^{N})}$
replaced by $%
{\displaystyle\int_{\mathbb{R}^{N}}}
du_{0}.$ And
\[
\Vert\nabla u(.,t)\Vert_{L^{q}(\mathbb{R}^{N})}\leqq Ct^{-\frac{1}{q}}(%
{\displaystyle\int_{\mathbb{R}^{N}}}
du_{0})^{\frac{1}{q}},
\]
thus $\left\vert \nabla u\right\vert ^{q}\in L_{loc}^{\infty}((0,T);L^{1}%
(\mathbb{R}^{N})).$
\end{proof}

\begin{remark}
As a consequence, under the assumptions of Corollary \ref{grad}, there holds
$u(.,t)\in C_{b}^{1}(\mathbb{R}^{N}),$ for any $t\in\left(  0,T\right)  ,$
then $u$ can be extended to a global solution of problem (\ref{cau}) on
$Q_{\mathbb{R}^{N},\infty},$ see for example \cite{SoZh}. 
\end{remark}

\subsection{Existence and uniqueness results for $q\leq2$}

Let $u_{0}\in L^{r}(\mathbb{R}^{N}),$ $r\geqq1.$ We first consider the
"subcritical" case
\begin{equation}
1<q<\frac{N+2r}{N+r},\qquad\text{equivalently \qquad}q<2\text{ and }%
r>\frac{N(q-1)}{2-q}. \label{subi}%
\end{equation}

\begin{theorem}
\label{val}Let $u_{0}\in L^{r}(\mathbb{R}^{N}),$ $r\geqq1.$ Suppose
(\ref{subi}), and $\nu>0$. Then any weak $L_{loc}^{r}$ solution $u$ of problem
(\ref{cau}) satisfies
\begin{equation}
\left\vert \nabla u\right\vert ^{q}\in L_{loc}^{1}(\left[  0,T\right)
;L^{r}(\mathbb{R}^{N})). \label{dim}%
\end{equation}
And
\[
u\text{ is a weak }L_{loc}^{r}\text{ solution}\Longleftrightarrow u\text{ is a
mild }L^{r}\text{ solution.}%
\]

\end{theorem}

\begin{proof}
Let $u$ be any weak $L_{loc}^{r}$ solution. Then from (\ref{vou}),
\[
\int_{0}^{\tau}\left\Vert \nabla u(.,t)\right\Vert _{L^{qr}(\mathbb{R}^{N}%
)}^{q}dt=\int_{0}^{\tau}(\int_{\mathbb{R}^{N}}\left\vert \nabla
u(.,t)\right\vert ^{qr}dx)^{\frac{1}{r}}dt\leqq C\int_{0}^{\tau}t^{-(\frac
{q}{2}+\sigma_{r,q,N}(q-1))}dt
\]
with $C=C_{q}\Vert u_{0}\Vert_{L^{r}(\mathbb{R}^{N})}^{(1+\varpi
_{r,q,N}(q-1))r},$ and (\ref{subi}) is equivalent to $q/2+\sigma
_{r,q,N}(q-1)<1$. Since $\nu>0$, the estimate (\ref{voi}) leads to the same
conclusion, because (\ref{subi}) is also equivalent to $q/2+(q-1)N/2r<1$. Then
(\ref{dim}) holds. Moreover from Corollary \ref{grad}, $u$ is a mild pointwise
solution:
\begin{equation}
u(.,t)=e^{t\Delta}u_{0}(.)-\nu\int_{0}^{t}\int_{\mathbb{R}^{N}}%
g(x-y,t-s)|\nabla u(y,s)|^{q}dyds. \label{frou}%
\end{equation}
Otherwise $u\in C(\left[  0,T\right)  ;L^{r}\left(  \mathbb{R}^{N}\right)  )$
from Theorem \ref{decay}, and $f=\left\vert \nabla u\right\vert ^{q}\in
L_{loc}^{1}(\left[  0,T\right)  ;L^{r}(\mathbb{R}^{N})),$ thus the relation
(\ref{frou}) holds in $L^{r}(\mathbb{R}^{N}),$
\begin{equation}
u(.,t)=(e^{t\Delta}u_{0})-\nu\int_{0}^{t}e^{(t-s)\Delta}\left\vert \nabla
u(.,s)\right\vert ^{q}(s)ds\qquad\text{in }L^{r}(\mathbb{R}^{N}), \label{truc}%
\end{equation}
that means $u$ is a mild $L^{r}$solution. The converse is clear.$\medskip$
\end{proof}

Next we deduce the uniqueness results of Theorem \ref{uni1}.

\begin{theorem}
\label{unir} Let $u_{0}\in L^{r}(\mathbb{R}^{N}).$ Assume (\ref{subi}) or
$q=2,$ and $\nu>0.$ Then there exists a unique weak $L_{loc}^{r}$ solution $u$
of problem (\ref{cau}). In the first case, $u\in C((0,T);W^{1,qr}%
(\mathbb{R}^{N})).$
\end{theorem}

\begin{proof}
(i) Case $1<q<(N+2r)/(N+r).$ From \cite[Theorem 2.1]{BASoWe}, there exists a
mild $L^{r}$ solution $u$, and it is unique in the class of mild $L^{r}$
solutions such that $u\in L_{loc}^{\infty}((0,T);W^{1,qr}\left(
\mathbb{R}^{N}\right)  ),$ see \cite[Lemma 2.2 and Remark 2.5]{BASoWe}. Then u
is a $L_{loc}^{r}$ solution. Let $v$ be any weak $L_{loc}^{r}$ solution, thus
$u$ is a mild $L^{r}$ solution, from Theorem \ref{val}. From Theorem
\ref{decay}, Corollary \ref{pass}, and Theorem \ref{val}, $v\in L^{\infty
}((0,T);L^{r}(\mathbb{R}^{N}))\cap$ $L_{loc}^{\infty}(\left(  0,T\right)
;C_{b}(\mathbb{R}^{N})),$ and $\left\vert \nabla v\right\vert \in
L_{loc}^{\infty}((0,T);L^{qr}\left(  \mathbb{R}^{N}\right)  )$ from Theorem
\ref{val} . Then $v\in L_{loc}^{\infty}((0,T);W^{1,qr}\left(  \mathbb{R}%
^{N}\right)  ),$ then $v=u$, and we reach the conclusion. Moreover $u\in
C((0,T);W^{1,qr}(\mathbb{R}^{N})),$ from \cite[Theorem 2.1]{BASoWe}.\medskip

(ii) Case $q=2.$ From \cite[Theorem 4.2]{BASoWe} there exists a unique
solution $u$ such that $u\in C(\left[  0,T\right)  ;L^{r}\left(
\mathbb{R}^{N}\right)  )\cap u\in C^{2,1}(\left(  Q_{\mathbb{R}^{N},\infty
}\right)  $ solution of (\ref{un}) at each point. Then it is a weak
$L_{loc}^{r}$ solution. Reciprocally any weak $L_{loc}^{r}$ solution $u$
satisfies the conditions above, from Theorem \ref{decay} and \cite{BiDao1}.
\end{proof}

\begin{theorem}
\label{souc} Assume $1<q<(N+2)/(N+1),$ $\nu>0.$ Let $u_{0}\in\mathcal{M}%
_{b}^{+}(\mathbb{R}^{N}).$ Then there exists a unique weak $\mathcal{M}_{loc}$
solution of problem (\ref{cau}).
\end{theorem}

\begin{proof}
The existence of a weak semi-group solution was obtained in \cite{BeLa99} by
approximation. The existence of a mild $\mathcal{M}$ solution was proved in
\cite[Theorem 2.2]{BASoWe}, and the two notions are equivalent from Lemma
\ref{lem}. In any case the solution is a weak $\mathcal{M}_{loc}$ solution.
Next consider any solution $\mathcal{M}_{loc}$ solution $u$. Then $u(.,t)\in
L^{\infty}(\mathbb{R}^{N})$ for any $t\in\left(  \epsilon,T\right)  $ by
applying Theorem \ref{bound} on $\left(  \epsilon/2,T\right)  $. Then again we
deduce $u(.,\epsilon)\in C_{b}(\mathbb{R}^{N}),$ and then (\ref{lis}) holds.
From Theorem \ref{decay} we still obtain that $u\in L_{loc}^{\infty
}((0,T);W^{1,q}\left(  \mathbb{R}^{N}\right)  ).$ And from the uniquenes on
$(\epsilon,T),$ we have $u\in C((\epsilon,T);W^{1,q}\left(  \mathbb{R}%
^{N}\right)  )$ from Theorem \ref{unir}. Then $u\in C((0,T);W^{1,q}\left(
\mathbb{R}^{N}\right)  ).$ And $u$ satisfies (\ref{ris}), from Theorem
\ref{inmea}. Then $u$ is a weak semi-group solution, thus a mild $\mathcal{M}$
solution from Lemma \ref{lem}. Therefore $u$ belongs to the class of
uniqueness of \cite[Theorem 2.2]{BASoWe}. We can also prove the uniqueness
directly: if $u_{1},u_{2}$ are two solutions, they are mild $\mathcal{M}$
solutions, thus
\[
(u_{1}-u_{2})(.,t)=\nu\int_{0}^{t}e^{(t-s)\Delta}(\left\vert \nabla
u_{1}(.,s)\right\vert ^{q}-\left\vert \nabla u_{2}(.,s)\right\vert ^{q})ds
\]
and we know that $\left\vert \nabla u_{j}\right\vert ^{q}\in C((0,T);L^{1}%
(\mathbb{R}^{N})),$ hence%
\begin{align*}
\left\Vert \nabla(u_{1}-u_{2})(.,t)\right\Vert _{L^{q}\left(  \mathbb{R}%
^{N}\right)  }  &  \leqq\nu\int_{0}^{t}\left\Vert \nabla(e^{(t-s)\Delta
})\right\Vert _{L^{1}\left(  \mathbb{R}^{N}\right)  }\left\Vert \left\vert
\nabla u_{1}(.,s)\right\vert ^{q}-\left\vert \nabla u_{2}(.,s)\right\vert
^{q}\right\Vert _{L^{qr}\left(  \mathbb{R}^{N}\right)  }ds\\
&  \leqq C\int_{0}^{t}(t-s)^{-\frac{1}{2}}\max_{j=1,2}\left\Vert \nabla
u_{j}(.,s)\right\Vert _{L^{\infty}\left(  \mathbb{R}^{N}\right)  }%
^{q-1}\left\Vert \nabla(u_{1}-u_{2})(.,s)\right\Vert _{L^{q}\left(
\mathbb{R}^{N}\right)  }ds\\
&  \leqq C\int_{0}^{t}(t-s)^{-\frac{1}{2}}s^{-(q-1)\vartheta_{1,q,N}%
}\left\Vert \nabla(u_{1}-u_{2})(.,s)\right\Vert _{L^{qr}\left(  \mathbb{R}%
^{N}\right)  }ds.
\end{align*}
thus we can apply the singular Gronwall Lemma when $(q-1)\vartheta
_{1,q,N}<1/2,$ which means precisely $q<(N+2)/(N+1).$ Then $\nabla(u_{1}%
-u_{2})(.,t)=0$ in $L^{q}\left(  \mathbb{R}^{N}\right)  ,$ hence $u_{1}%
=u_{2}.\medskip$
\end{proof}

Finally we give a short proof of the existence result of \cite[Theorem
4.1]{BASoWe}.

\begin{proposition}
\label{pou}Let $\nu>0,$ $1<q<2$. For any nonnegative $u_{0}\in L^{r}%
(\mathbb{R}^{N}),r\geq1,$ there exists a mild pointwise solution $u$ of
problem (\ref{cau}), and $u\in C(\left[  0,T\right)  ;L^{r}\left(
\mathbb{R}^{N}\right)  )$.
\end{proposition}

\begin{proof}
Let $u_{0,n}=\min(u_{0},n).$ Then $u_{0,n}\in L^{\rho}(\mathbb{R}^{N})$ for
any $\rho\geq r.\ $We choose $\rho>N(q-1)/(2-q),$ that means $q<(N+2\rho
)/(N+\rho).$ From \cite[Theorem 2.1]{BASoWe}, there exists a mild $L^{\rho}$
solution $u_{n}$ with initial data $u_{0,n},$ and $u_{n}\in C((0,T);C_{b}%
^{2}(\mathbb{R}^{N}))\cap C^{2,1}(Q_{\mathbb{R}^{N},T}).$ The sequence
$\left(  u_{n}\right)  $ is nondecreasing from the comparison principle, and
$u_{n}(.,t)\leq e^{t\Delta}u_{0}\leq Ct^{-N/2r}\left\Vert u_{0}\right\Vert
_{L^{r}(\mathbb{R}^{N})}.$ From Corollary \ref{pass}, $\left(  u_{n}\right)  $
converges in $C_{loc}^{2,1}(Q_{\mathbb{R}^{N},T})$ to a weak solution $u$ of
(\ref{un}) in $Q_{\mathbb{R}^{N},T}$, and $u(.,t)\leq e^{t\Delta}u_{0}.$
Moreover $(\left\vert \nabla u_{n}\right\vert ^{q})$ is bounded in
$L_{loc}^{1}\left(  \left[  0,T\right)  ;L_{loc}^{1}(\mathbb{R}^{N})\right)
:$ indeed for any $\xi\in\mathcal{D}(\mathbb{R}^{N}),$ with values in $\left[
0,1\right]  ,$ and any $0<s<t<T,$%
\begin{align*}
&  \int_{\mathbb{R}^{N}}u_{n}(t,.)\xi^{q^{\prime}}dx+\nu\int_{s}^{t}%
\int_{\mathbb{R}^{N}}|\nabla u_{n}|^{q}\xi^{q^{\prime}}dx\leqq-q^{\prime}%
\nu\int_{s}^{t}\int_{\mathbb{R}^{N}}\xi^{\frac{1}{q-1}}\nabla u_{n}.\nabla\xi
dx+\int_{\mathbb{R}^{N}}u_{n}(s,.)\xi^{q^{\prime}}dx\\
&  \leqq\frac{\nu}{2}\int_{s}^{t}\int_{\mathbb{R}^{N}}|\nabla u_{n}|^{q}%
\xi^{q^{\prime}}dx+Ct\int_{\mathbb{R}^{N}}|\nabla\xi|^{q^{\prime}}%
dx+\int_{\mathbb{R}^{N}}u_{n}(s,.)\xi^{q^{\prime}}dx,
\end{align*}
and $u_{n}\in C(\left[  0,T\right)  ;L^{\rho}\left(  \mathbb{R}^{N}\right)
);$ thus we can go to the limit as $s\rightarrow0:$
\[
\int_{\mathbb{R}^{N}}u_{n}(t,.)\xi^{q^{\prime}}dx+\frac{1}{2}\int_{s}^{t}%
\int_{\mathbb{R}^{N}}|\nabla u_{n}|^{q}\xi^{q^{\prime}}dx\leqq Ct\int%
_{\mathbb{R}^{N}}|\nabla\xi|^{q^{\prime}}dx+\int_{\mathbb{R}^{N}}u_{0}%
\xi^{q^{\prime}}dx.
\]
Thus \textbf{ }$\left\vert \nabla u\right\vert ^{q}\in L_{loc}^{1}\left(
\left[  0,T\right)  ;L_{loc}^{1}(\mathbb{R}^{N})\right)  ,$ hence, from
\cite[Proposition 2.15]{BiDao1}, $u$ admits a trace as $t\rightarrow0:$ there
exists a Radon measure $\mu_{0}$ in $\mathbb{R}^{N},$ such that $u(.,t)$
converges weakly* to $\mu_{0}.$ Otherwise $e^{t\Delta}u_{0}$ converges to
$u_{0}$ in $L^{r}(\mathbb{R}^{N}),$ thus $\mu_{0}\in L_{loc}^{1}%
(\mathbb{R}^{N})$ and $0\leq\mu_{0}\leq u_{0}$; and $u_{n}\leq u,$ thus
$u_{0,n}\leq\mu_{0},$ hence $\mu_{0}=u_{0}.$ Moreover there exists a function
$g\in L^{r}(\mathbb{R}^{N})$ such that $u(.,t)\leqq g$ for small $t.$ Then the
nonnegative function $e^{t\Delta}u_{0}-u(.,t)$ converges weakly* to $0,$ and
then in $L_{loc}^{1}(\mathbb{R}^{N}).$ Hence $u(.,t)$ converges to $u_{0}$ in
$L_{loc}^{1}(\mathbb{R}^{N}),$ then in $L^{r}(\Omega)$ from the dominated
convergence theorem. Thus $u\in C(\left[  0,T\right)  ;L^{r}\left(
\mathbb{R}^{N}\right)  )$. In particular $u$ is a weak $L_{loc}^{r}$ solution,
then a pointwise mild solution, from Corollary \ref{grad}.
\end{proof}

\begin{remark}
The uniqueness of the solution is still an open problem when $u_{0}\in
L^{r}(\mathbb{R}^{N})$ and $q\geqq(N+2r)/(N+r).$
\end{remark}

\subsection{More decay estimates for $q<(N+2r)/(N+r)$}

Here, we exploit theorem \ref{decay} to obtain a better decay estimate of the
$L^{r}$ norm when $u_{0}\in L^{r}(\mathbb{R}^{N})$ in the subcritical case
(\ref{subi}), which appears to be new for $r>1.$ In case $r=1$ we find again
the result of \cite{AnTeU}, proved \textit{under the assumption that the
energy relation (\ref{his}) holds.}

\begin{theorem}
\label{esti}Let $r\geqq1$ and assume (\ref{subi}), $\nu>0$. Let $u$ be any
non-negative weak $r$ solution of problem (\ref{cau}) in $Q_{\mathbb{R}%
^{N},\infty},$ with initial data $u_{0}\in L^{r}(\mathbb{R}^{N})$. Then there
exists $C=C(N,q,r)$ such that, for any $t>0,$
\begin{equation}
\int_{\mathbb{R}^{N}}u^{r}(.,t)dx\leq C(\int_{\{|x|>\sqrt{t}\}}u_{0}%
^{r}(x)dx+t^{-\frac{ar-N}{2}}),\qquad a=\frac{2-q}{q-1}. \label{alb}%
\end{equation}
As a consequence, $\lim_{t\rightarrow\infty}\Vert u(t)\Vert_{L^{r}%
(\mathbb{R}^{N})}=0$ and
\[
r\int_{0}^{\infty}\int_{\mathbb{R}^{N}}u^{r-1}|\nabla u|^{q}dxdt+r(r-1)\nu
\int_{0}^{\infty}\int_{\mathbb{R}^{N}}u^{r-2}|\nabla u|^{2}dxdt=\int%
_{\mathbb{R}^{N}}u_{0}^{r}dx.
\]

\end{theorem}

\begin{proof}
We still consider $v=u^{b}$ with $b=(q-1+r)/q<$ $r,$ and set $E(s)=\int%
_{\mathbb{R}^{N}}u^{r}(.,s)dx.$ Then $E\in W^{1,1}((0,T)),$ from the energy
relation (\ref{egge}), and for almost any $s\in(0,T),$%
\[
E^{\prime}(s)=-r(r-1)\nu\int_{\mathbb{R}^{N}}|\nabla u|^{2}u^{r-2}%
(.,s)dx-\int_{\mathbb{R}^{N}}|\nabla u|^{q}u^{r-1}(.,s)dx\leq0.
\]
Next, we set $E=E_{1}+E_{2}$ with
\[
E_{1}(s)=\int_{\{|x|<2R\}}u^{r}(x,s)dx,\qquad E_{2}(s)=\int_{\{|x|\geq
2R\}}u^{r}(x,s)dx.
\]
From the Gagliardo-Nirenberg inequality (\ref{gal}), we obtain successively,
with $C=C(N,q,r),$
\begin{align*}
E_{1}(s)  &  =\int_{\{|x|<2R\}}v^{\frac{r}{b}}(x,s)dx\leq\left(
\int_{\{|x|<2R\}}v^{q}(x,s)dx\right)  ^{\frac{r}{bq}}(2R)^{1-\frac{r}{bq}}\\
&  \leqq C\Vert\nabla v(s)\Vert_{L^{q}(\mathbb{R}^{N})}^{\frac{kr}{b}}\Vert
v(s)\Vert_{L^{r/b}(\mathbb{R}^{N})}^{\frac{(1-k)r}{b}}R^{N(1-\frac{r}{bq})}\\
&  \leqq\frac{1}{2}\Vert v(s)\Vert_{L^{r/b}(\mathbb{R}^{N})}^{\frac{r}{b}%
}+C\Vert\nabla v(s)\Vert_{L^{q}(\mathbb{R}^{N})}^{\frac{kr}{b}}R^{\frac{N}%
{k}(1-\frac{r}{bq})},
\end{align*}
thus
\begin{equation}
E(s)\leqq C(\Vert\nabla v(s)\Vert_{L^{q}(\mathbb{R}^{N})}^{\frac{r}{b}%
}R^{\frac{N}{k}(1-\frac{r}{bq})}+2E_{2}(s)). \label{3.35a}%
\end{equation}
Let $\eta\in\mathcal{D}(\mathbb{R}^{N})$ with values in $\left[  0,1\right]
,$ such that $\varphi=1$ in $B_{1},$ with support in $\overline{B_{2}}$, and
set $\eta=1-\varphi$, and $\varphi_{l}(x)=\varphi(\frac{x}{l}),\hspace
{0.05in}\eta_{R}(x)=\eta(\frac{x}{R})$. Observe that our assumption on $q$
implies $q^{\prime}>N/r$. As in the first step of theorem \ref{decay}, we
obtain for any $0<\sigma<s<t<T$, and $l>2R$,
\begin{equation}
\left(  \int_{\mathbb{R}^{N}}u^{r}(.,s)\varphi_{l}^{\lambda}\eta_{R}^{\lambda
}dx\right)  ^{\frac{1}{r}}\leq\left(  \int_{\mathbb{R}^{N}}u^{r}%
(.,\sigma)\varphi_{l}^{\lambda}\eta_{R}^{\lambda}dx\right)  ^{\frac{1}{r}%
}+C(s-\sigma)(R^{\frac{N}{r}-q^{\prime}}+l^{\frac{N}{r}-q^{\prime}}),
\label{3.36a}%
\end{equation}
with $\lambda=rq^{\prime},$ and $C=C(N,q,r,\eta).$ As $\sigma\rightarrow0$ and
$l\rightarrow\infty$. we deduce
\[
\left(  \int_{\mathbb{R}^{N}}u^{r}(x,s)\eta_{R}dx\right)  ^{\frac{1}{r}}%
\leq\left(  \int_{\mathbb{R}^{N}}u_{0}^{r}(x)\eta_{R}dx\right)  ^{\frac{1}{r}%
}+CsR^{\frac{N}{r}-q^{\prime}}.
\]
Taking $R=\sqrt{t}$, and setting
\[
\rho=r+\frac{N-rq^{\prime}}{2}=\frac{(N+2r)-q(N+r)}{2(q-1)}=\frac{ar-N}{2},
\]
we find, with a constant $C$ as above,
\[
E_{2}(s)\leq A(t)=C\left(  \int_{\{|x|>\sqrt{t}\}}u_{0}^{r}(x)dx+t^{-\rho
}\right)  ,
\]
Next, we consider $F(s)=E(s)-2A(t).$ $\hspace{0.05in}$If there exists
$t_{0}\in(0,t)$ such that $F(t_{0})\leq0$, then $F(s)\leq0,\hspace
{0.05in}\forall s\in(t_{0},t);$ thus $E(t)\leq2A(t),$ by continuity, hence
(\ref{alb}) holds. Next assume that $F(s)>0$, \hspace{0.05in}for any
$s\in(0,t).$ Since
\begin{equation}
-F^{\prime}(s)\geq\nu\int_{\mathbb{R}^{N}}|\nabla u|^{q}u^{r-1}(x,s)dx=\nu
\int_{\mathbb{R}^{N}}|\nabla v(x,s)|^{q}dx, \label{3.39a}%
\end{equation}
we find $F(s)\leqq C(-F^{\prime}(s))^{r/bq}t^{(1-r/bq)N/2k}$ from
(\ref{3.35a}). By integration we get
\[
C(t-s)t^{-\frac{N}{2k}(1-\frac{r}{bq})}\leqq F(t)^{-\frac{q-1}{r}}%
-F(s)^{\frac{q-1}{r}}.
\]
As $s\longrightarrow0$ we deduce that $F(t)\leqq Ct^{-\rho},$ since
$\rho=r/(q-1)-N/2k$, and (\ref{alb}) still holds.
\end{proof}

\begin{remark}
The case $r=1$ has been the object of many works, assuming that $u_{0}\in
L^{1}(\mathbb{R}^{N})\cap W^{1,\infty}(\mathbb{R}^{N}).$ There holds
\[
\lim_{t\rightarrow\infty}\Vert u(t)\Vert_{L^{1}(\mathbb{R}^{N})}%
=0\Longleftrightarrow q\leq(N+2)/(N+1),
\]
see \cite{AmBa}, \cite{BeLa99}, \cite{BAKo}, \cite{GaLa}. When $q<(N+2)/(N+1)$%
, the absorption plays a role in the asymptotics. From \cite{BeKaLa}, if
lim$_{\left\vert x\right\vert \rightarrow\infty}\left\vert x\right\vert ^{a}$
$u_{0}(x)=0,$ where $a=(2-q)/(q-1),$ then $u(.,t)$ converges as $t\rightarrow
\infty$ to the very singular solution constructed in \cite{QW}, \cite{BeLa01};
then $\int_{\mathbb{R}^{N}}u(.,t)dx$ behaves like $t^{-(a-N)/2}$ for large
$t,$ and estimate (\ref{alb}) is sharp. When $q>(N+2)/(N+1),$ and $u_{0}\in
L^{1}(\mathbb{R}^{N}),$ then $u(.,t)$ behaves as the fundamental solution of
heat equation, see \cite{BeKaLa}.

Our result is new when $u_{0}\in L^{r}(\mathbb{R}^{N}),$ $r>1.$ When
$q>(N+2)/(N+1),$ and $u_{0}$ is bounded and behaves like $\left\vert
x\right\vert ^{-b}$ as $\left\vert x\right\vert \rightarrow\infty$ with
$b\in(a,N),$ it has been shown that $u(.,t)$ behaves as the selfsimilar
solution of the heat equation with initial data $\left\vert x\right\vert
^{-b},$ see \cite{BiGuKa}. In that case $u_{0}\in L^{r}(\mathbb{R}^{N})$ for
any $r>N/b$ and $\int_{\mathbb{R}^{N}}u^{r}(.,t)dx$ behaves like
$t^{-(br-N)/2}$. Thus (\ref{alb}) is sharp as $b\rightarrow a.$
\end{remark}

\section{The Dirichlet problem in $Q_{\Omega,T}$\label{omega}}

Here we study equation (\ref{un}) in case of a regular bounded domain
$\Omega,$ with Dirichlet conditions on $\partial\Omega\times(0,T),$ with
$\nu>0;$ by homothety we can assume $\nu=1$:
\begin{equation}
(D_{\Omega,T})\left\{
\begin{array}
[c]{l}%
u_{t}-\Delta u+|\nabla u|^{q}=0,\quad\text{in}\hspace{0.05in}Q_{\Omega,T},\\
u=0\quad\text{on}\hspace{0.05in}\partial\Omega\times(0,T),
\end{array}
\right.  \label{ome}%
\end{equation}

As in section \ref{Rn}, we study the problem with rough initial data, and
introduce different notions of solutions.\medskip

\subsection{Solutions of the heat equation with $L^{1}$ data}

The regularization method used at Section \ref{Rn} does not provide estimates
up to the boundary. In this section we use another argument: the notion of
\textit{ entropy solution, }introduced in \cite{Pr}, for the problem\textit{
}
\begin{equation}
\left\{
\begin{array}
[c]{l}%
u_{t}-\Delta u=f,\quad\text{in}\hspace{0.05in}Q_{\Omega,s,\tau},\\
u=0\quad\text{on}\hspace{0.05in}\partial\Omega\times(s,\tau),\\
u(.,s)=u_{s}\geqq0
\end{array}
\right.  \label{ggs}%
\end{equation}
when $f$ and $u_{s}$ are integrable, that we recall now. For any $k>0$ and
$\theta\in\mathbb{R},$ we define as usual the truncation function $T_{k}$ and
a primitive $\Theta_{k}$ by
\begin{equation}
T_{k}(\theta)=\max(-k,\min(k,\theta)),\qquad\Theta_{k}(s)=\int_{0}^{r}%
T_{k}(\theta)d\theta. \label{tro}%
\end{equation}

\begin{definition}
Let $s,\tau\in\mathbb{R}$ with $s<\tau,$ and $f\in L^{1}(Q_{\Omega,s,\tau})$
and $u_{s}\in L^{1}(\Omega).$ A function $u\in C(\left[  s,\tau\right]
;L^{1}(\Omega))$ is an entropy solution of the problem (\ref{ggs}) if
$T_{k}(u)\in L^{2}(\left(  s,\tau\right)  ;W_{0}^{1,2}(\Omega))$ for any
$k>0,$ and
\begin{equation}%
\begin{array}
[c]{c}%
\int_{\Omega}\Theta_{k}(u-\varphi)(.,\tau)dx-\int_{\Omega}\Theta_{k}%
(u_{s}-\varphi(.,s)dx+\int_{s}^{\tau}\langle\varphi_{t},T_{k}(u-\varphi
)\rangle dt\\
+\int_{s}^{\tau}\int_{\Omega}(\nabla u.\nabla T_{k}(u-\varphi)-fT_{k}%
(u-\varphi)dxdt\leq0
\end{array}
\label{nnn}%
\end{equation}
for any $\varphi\in L^{2}((s,\tau);W^{1,2}(\Omega))\cap L^{\infty}\left(
Q_{\Omega,\tau}\right)  $ such that $\varphi_{t}\in L^{2}((s,\tau
);W^{-1,2}(\Omega))$.
\end{definition}

Other notions of solutions have been used for this problem, see \cite{BeDa},
recalled below. In fact they are equivalent: here $e^{t\Delta}$ denotes the
semi-group of the heat equation with Dirichlet conditions acting on
$L^{1}\left(  \Omega\right)  ,$

\begin{lemma}
\label{pare} Let -$\infty<s<\tau<\infty,$ $f\in L^{1}(Q_{\Omega,s,\tau})$,
$u_{s}\in L^{1}(\Omega)$ and $u\in C(\left[  s,\tau\right]  ;L^{1}(\Omega)),$
$u(.,s)=u_{s}.$ Then the three properties are equivalent:\medskip

(i) $u\in$ $L^{1}((s,\tau);W_{0}^{1,1}\left(  \Omega\right)  )$, such that
\begin{equation}
u_{t}-\Delta u=f,\quad\text{in}\hspace{0.05in}\mathcal{D}^{\prime}%
(Q_{\Omega,s,\tau}); \label{fil}%
\end{equation}

(ii) $u$ is a mild solution of (\ref{ggs}), that means, for any $t\in\left[
s,\tau\right]  ,$
\begin{equation}
u(.,t)=e^{(t-s)\Delta}u_{s}+\int_{s}^{t}e^{(t-\sigma)\Delta}f(\sigma
)d\sigma\qquad\text{in }L^{1}\left(  \Omega\right)  ; \label{seg}%
\end{equation}

(iii) $u$ is an entropy solution of (\ref{ggs}).\medskip

\noindent Such a solution exists, is unique, and will be called weak solution
of (\ref{ggs}).
\end{lemma}

\begin{proof}
It follows from the existence and uniqueness of the solutions of (i) from
\cite[Lemma 3.4]{BaPi}, as noticed in \cite{BeDa}, and of the entropy
solutions, see \cite{BlMu}.$\medskip$
\end{proof}

As a consequence, when $u$ is bounded, we can admit test functions of the form
$u^{\alpha}:$

\begin{lemma}
\label{dil}Let $s,\tau\in\mathbb{R}$ with $s<\tau,$ and $f\in L^{1}%
(Q_{\Omega,s,\tau})$ and $u$ be any nonnegative \textbf{bounded }weak solution
in $Q_{\Omega,s,\tau}$ of (\ref{ggs}). \medskip

Then, for any $\alpha>0$, there holds $u^{\alpha-1}\left\vert \nabla
u\right\vert ^{2}\in L^{1}(Q_{\Omega,s,\tau})$ and
\begin{equation}
\frac{1}{\alpha+1}\int_{\Omega}u^{\alpha+1}(.,\tau))dx+\alpha%
{\displaystyle\int}
\int_{Q_{\Omega,s,\tau}}u^{\alpha-1}\left\vert \nabla u\right\vert
^{2}dxdt=\frac{1}{\alpha+1}\int_{\Omega}u^{\alpha+1}(.,s))dx+\int_{s}^{\tau
}\int_{\Omega}fu^{\alpha}dxdt. \label{hol}%
\end{equation}

\end{lemma}

\begin{proof}
We have $u\in L^{2}((s,\tau);W_{0}^{1,2}(\Omega))\cap L^{\infty}\left(
Q_{\Omega,s,\tau}\right)  ,$ and $u_{t}\in L^{2}((s,\tau);W^{-1,2}%
(\Omega))+L^{1}\left(  Q_{\Omega,s,\tau}\right)  .$ Then any function
$\varphi\in L^{2}((s,\tau);W_{0}^{1,2}(\Omega))\cap L^{\infty}\left(
Q_{\Omega,s,\tau}\right)  $ is admissible in equation (\ref{fil}). In
particular for any $\alpha>0$, we can take $\varphi=M_{\alpha,\delta
}(u)=(u+\delta)^{\alpha}-\delta^{\alpha},$ with $\delta>0.$ Integrating on
$\left[  s,\tau\right]  $ we deduce that
\[
\int_{s}^{\tau}<u_{t},\varphi>+\alpha%
{\displaystyle\int}
\int_{Q_{\Omega,s,\tau}}(u+\delta)^{\alpha-1}\left\vert \nabla u\right\vert
^{2}dxdt=\int_{s}^{\tau}\int_{\Omega}fM_{\alpha,\delta}(u)dxdt.
\]
Let $k>0$ such that $\sup_{Q_{\Omega,s,\tau}}u\leqq k,$ thus $u=T_{k}(u).$ The
function $\theta\mapsto M(\theta)=(T_{k}(\theta)+\delta)^{\alpha}%
-\delta^{\alpha}$ is continuous on $\mathbb{R}^{+}$and piecewise $C^{1}$ such
that $M(0)=0$ and $M^{\prime}$ has a compact support. Denoting $\mathcal{M}%
_{\alpha,\delta}(r)=(u+\delta)^{\alpha+1}/(\alpha+1)-\delta^{\alpha}u,$ we can
integrate by parts from \cite[Lemma 7.1]{DrPr}, and deduce that
\[
\int_{\Omega}\mathcal{M}_{\alpha,\delta}(u)(.,\tau))dx-\int_{\Omega
}\mathcal{M}_{\alpha,\delta}(u)(.,s))dx+\alpha%
{\displaystyle\int}
\int_{Q_{\Omega,s,\tau}}(u+\delta)^{\alpha-1}\left\vert \nabla u\right\vert
^{2}dxdt=\int_{s}^{\tau}\int_{\Omega}fM_{\alpha,\delta}(u)dxdt;
\]
We can go to the limit as $\delta\rightarrow0$ from the Fatou Lemma, and then
from the dominated convergence theorem. Thus (\ref{hol}) holds for $\alpha>0.$
\end{proof}

\begin{remark}
From \cite{DrPr}, the notion of entropy solution of (\ref{ggs}) is also
equivalent to the notion of renormalized solution, that we develop in Section
\ref{quas}. Lemma \ref{dil} is a special case of a much more general property
of the truncates when $u$ is not necessarily bounded, see Lemma \ref{admi}.
\end{remark}

\subsection{Different notions of solutions of problem $(D_{\Omega,T})$}

\begin{definition}
We say that $u$ is a \textbf{weak } \textbf{solution} of the problem
$(D_{\Omega,T})$ if $u\in C((0,T);L^{1}\left(  \Omega\right)  )\cap
L_{loc}^{1}((0,T);W_{0}^{1,1}\left(  \Omega\right)  ),$ such that $|\nabla
u|^{q}\in L_{loc}^{1}((0,T);L^{1}\left(  \Omega\right)  )$ and $u$ satisfies
\begin{equation}
u_{t}-\Delta u+|\nabla u|^{q}=0,\quad\text{in}\hspace{0.05in}\mathcal{D}%
^{\prime}(Q_{\Omega,T}). \label{dis}%
\end{equation}

\end{definition}

Next we study the Cauchy-Dirichlet problem%
\begin{equation}
\left\{
\begin{array}
[c]{l}%
u_{t}-\Delta u+|\nabla u|^{q}=0,\quad\text{in}\hspace{0.05in}Q_{\Omega,T},\\
u=0\quad\text{on}\hspace{0.05in}\partial\Omega\times(0,T),\\
u(x,0)=u_{0}\geqq0
\end{array}
\right.  \label{def}%
\end{equation}
with $u_{0}\in L^{r}\left(  \Omega\right)  ,$ $r\geqq1,$ or only $u_{0}%
\in\mathcal{M}_{b}^{+}(\Omega).$ Here in any case $u_{0}\in\mathcal{M}_{b}%
^{+}(\Omega).$

\begin{definition}
If $u_{0}\in L^{r}(\Omega),r\geqq1,$ we say that $u$ is a \textbf{weak }%
$L^{r}$\textbf{ solution } of problem (\ref{def}) if it is a weak solution of
$(D_{\Omega,T})$, such that the extension of $u$ by $u_{0}$ at $t=0$ satisfies
$u\in C\left(  \left[  0,T\right)  ;L^{r}(\Omega\right)  ).$
\end{definition}

\begin{definition}
For any $u_{0}\in\mathcal{M}_{b}^{+}(\Omega),$ we say that $u$ is a weak
$\mathcal{M}$ solution of problem (\ref{def}) if it is a weak solution of
$(D_{\Omega,T})$, such that
\begin{equation}
\lim_{t\rightarrow0}\int_{\Omega}u(.,t)\psi dx=\int_{\Omega}\psi du_{0}%
,\qquad\forall\psi\in C_{b}(\Omega). \label{jkl}%
\end{equation}

\end{definition}

Semi-group type solutions have been introduced in \cite{BeDa}, see also
\cite{Al}. For $u_{0}\in\mathcal{M}_{b}^{+}(\Omega),$ we set $e^{t\Delta}%
u_{0}=\int_{\Omega}g_{\Omega}(.,y,t)du_{0}(y),$ where $g_{\Omega}$ is the heat
kernel with Dirichlet conditions on $\partial\Omega.$

\begin{definition}
\label{mildom}For any $u_{0}\in\mathcal{M}_{b}^{+}(\Omega),$ a function $u$ is
a \textbf{mild } \textbf{solution }of problem (\ref{def}) if $u\in
C((0,T);L^{1}\left(  \Omega\right)  )$, and $|\nabla u|^{q}\in L_{loc}%
^{1}(\left[  0,T\right)  ;L^{1}\left(  \Omega\right)  )$ and
\begin{equation}
u(.,t)=e^{t\Delta}u_{0}(.)-\int_{0}^{t}e^{(t-s)\Delta}|\nabla u(.,s)|^{q}%
ds\qquad\text{in }L^{1}\left(  \Omega\right)  , \label{abc}%
\end{equation}

\end{definition}

\begin{remark}
As it was shown in \cite[p.1420]{BeDa}, from Lemma \ref{pare},
\[
u\text{ is a mild\textbf{ }solution}\Longleftrightarrow u\text{ is a weak
}\mathcal{M}\text{ solution such that }|\nabla u|^{q}\in L_{loc}^{1}(\left[
0,T\right)  ;L^{1}\left(  \Omega\right)  );
\]
and then $u\in L_{loc}^{1}(\left[  0,T\right)  ;W_{0}^{1,1}\left(
\Omega\right)  ).$
\end{remark}

\begin{remark}
\label{few}As in Remark \ref{all}, the definition of mild solution requires an
integrability property of the gradient up to $t=0$, namely $|\nabla u|^{q}\in
L_{loc}^{1}(\left[  0,T\right)  ;L^{1}\left(  \Omega\right)  ).$ The
definition of weak solution only assumes that $|\nabla u|^{q}\in L_{loc}%
^{1}((0,T);L^{1}\left(  \Omega\right)  ).$
\end{remark}

\subsection{Decay and regularizing effect}

Here $\Omega$ is bounded, then the situation is simpler than in $\mathbb{R}%
^{N}$: indeed we take benefit of the regularizing effect of the semi-group
$e^{t\Delta}$ associated with the first eigenvalue $\lambda_{1}$ of the
Laplacian, and also of the inclusion $L^{r}(\Omega)\subset$ $L^{1}(\Omega).$

\begin{lemma}
\label{sim}Let $q>1,$ and $u_{0}\in L^{r}(\Omega),$ $r\geqq1.$ 1) Let $u$ be
any non-negative weak $L^{r}$\textbf{ }solution of problem (\ref{def}%
).$\medskip$

\noindent(i) Then $u(.,t)\in L^{\infty}(\Omega)$ for any $t\in\left(
0,T\right)  ,$ and
\begin{equation}
\Vert u(.,t)\Vert_{L^{r}(\Omega)}\leqq Ce^{-\lambda_{1}t}\Vert u_{0}%
\Vert_{L^{r}(\Omega)},\qquad\Vert u(.,t)\Vert_{L^{\infty}(\Omega)}\leqq
Ct^{-\frac{N}{2r}}e^{-\lambda_{1}t}\Vert u_{0}\Vert_{L^{r}(\Omega)}.
\label{clo}%
\end{equation}
(ii) Moreover $|\nabla u|^{q}\in L_{loc}^{1}(\left[  0,T\right)  ;L^{1}\left(
\Omega\right)  ),$ and
\begin{equation}
\int_{\Omega}u(.,t)dx+\int_{0}^{t}\int_{\Omega}|\nabla u|^{q}dxdt\leqq
\int_{\Omega}u_{0}dx. \label{ple}%
\end{equation}
If $r>1,$ then \textbf{ }$u^{r-1}|\nabla u|^{q}\in L_{loc}^{1}(\left[
0,T\right)  ;L^{1}\left(  \Omega\right)  )$ and $u^{r-2}|\nabla u|^{2}\in
L_{loc}^{1}(\left[  0,T\right)  ;L^{1}\left(  \Omega\right)  ),$ and
\begin{equation}
\frac{1}{r}\int_{\Omega}u^{r}(.,t)dx+\int_{0}^{t}\int_{\Omega}u^{r-1}|\nabla
u|^{q}dxdt+(r-1)\int_{0}^{t}\int_{\Omega}u^{r-2}|\nabla u|^{2}dxdt=\frac{1}%
{r}\int_{\Omega}u_{0}^{r}dx, \label{plo}%
\end{equation}
As a consequence, $u^{q-1+r}\in L_{loc}^{1}((\left[  0,T\right)  ;W_{0}%
^{1,1}\left(  \Omega\right)  ).\medskip$

\noindent2) Let $u_{0}\in\mathcal{M}_{b}^{+}(\Omega)$ and $u$ be any
non-negative weak $\mathcal{M}$ solution of problem (\ref{def}). Then
(\ref{clo}) and (\ref{ple}) still hold as in case $u_{0}\in L^{1}(\Omega),$
where the norm $\Vert u_{0}\Vert_{L^{1}(\Omega)}$ is replaced by $%
{\displaystyle\int_{\Omega}}
du_{0}.$ In particular $u$ is a mild solution.
\end{lemma}

\begin{proof}
1) (i) Let $0<\epsilon<\tau<T.$ Since $u$ is a weak solution of $(D_{\Omega
,T})$, we can apply Lemma \ref{pare} with $f=-|\nabla u|^{q}$ in
$Q_{\Omega,\epsilon,\tau}.$ Thus $u$ is a mild solution of the problem in
$Q_{\Omega,\epsilon,\tau}:$ for any $t\in\left[  \epsilon,\tau\right]  ,$
\[
u(.,t)=e^{(t-\epsilon)\Delta}u(.,\epsilon)-\int_{\epsilon}^{t}e^{(t-\sigma
)\Delta}|\nabla u|^{q}d\sigma\qquad\text{in }L^{1}\left(  \Omega\right)  .
\]
therefore $u(.,t)\leqq e^{(t-\epsilon)\Delta}u(.,\epsilon).$ From our
assumptions $u\in C\left(  \left[  0,T\right)  ;L^{r}(\Omega\right)  ),$ we
deduce that $u(.,t)\leqq e^{t\Delta}u_{0}$ as $\epsilon\rightarrow0.$ Then
(\ref{clo}) follows from the properties of the semi-group $e^{t\Delta}%
$.$\medskip$

(ii) The function $u$ is bounded in $Q_{\Omega,s,\tau},$ thus from Lemma
\ref{dil}, for any $\rho>1,$
\begin{equation}
\frac{1}{\rho}\int_{\Omega}u^{\rho}(.,t)dx+\int_{\epsilon}^{t}\int_{\Omega
}u^{\rho-1}|\nabla u|^{q}dxdt+(\rho-1)\int_{\epsilon}^{t}\int_{\Omega}%
u^{\rho-2}|\nabla u|^{2}dxdt=\frac{1}{\rho}\int_{\Omega}u^{\rho}%
(.,\epsilon)dx. \label{plm}%
\end{equation}
As $\rho\rightarrow1,$ we deduce that $|\nabla u|^{q}\in L^{1}\left(
Q_{\Omega,\epsilon,\tau}\right)  $ from the Fatou Lemma, and
\[
\int_{\Omega}u(.,t)dx+\int_{\epsilon}^{t}\int_{\Omega}|\nabla u|^{q}%
dxdt\leqq\int_{\Omega}u(.,\epsilon)dx.
\]
As $\epsilon\rightarrow0$ we deduce that $|\nabla u|^{q}\in L^{1}\left(
Q_{\Omega,\tau}\right)  $ and (\ref{ple}) holds. If $r>1,$ we can take
$\rho=r$ in (\ref{plm}) and obtain (\ref{plo}) as $\epsilon\rightarrow0$. Then
$u^{q-1+r}\in L_{loc}^{1}((\left[  0,T\right)  ;W_{0}^{1,1}\left(
\Omega\right)  )$ as in the case of $\mathbb{R}^{N}.$

\noindent2) The same estimates hold because $\lim_{\epsilon\rightarrow0}\Vert
u(.,\epsilon)\Vert_{L^{1}(\Omega)}=%
{\displaystyle\int_{\Omega}}
du_{0}.$\medskip
\end{proof}

\begin{theorem}
\label{deco}Let $q>1$ and $u_{0}\in L^{r}(\Omega),$ $r\geqq1.$ 1) Let $u$ be
any non-negative weak $L^{r}$\textbf{ }solution of problem (\ref{def}). Then%
\begin{equation}
\Vert u(.,t)\Vert_{L^{\infty}(\mathbb{R}^{N})}\leqq\left\{
\begin{array}
[c]{ccc}%
Ct^{-\sigma_{r,q,N}}\Vert u_{0}\Vert_{L^{r}(\Omega)}^{\varpi_{r,q,N}}, &
C=C(N,q,r), & \text{if }q\neq N,\\
C_{\varepsilon}t^{-(1+\varepsilon)\sigma_{r,N,N}}\Vert u_{0}\Vert
_{L^{r}(\Omega)}^{(1+\varepsilon)\varpi_{r,q,N}}, & \forall\varepsilon>0,\quad
C_{\varepsilon}=C(N,q,r,\varepsilon), & \text{if }q=N,
\end{array}
\right.  \label{yuc}%
\end{equation}
where $\sigma_{r,q,N},\varpi_{r,q,N}$ are given at (\ref{sigr}).$\medskip$

\noindent2) Any non-negative weak solution $u$ of $(D_{\Omega,T})$ satisfies
the universal estimate, where $C=C(N,q,\left\vert \Omega\right\vert ),$%
\begin{equation}
\Vert u(.,t)\Vert_{L^{\infty}(\Omega)}\leqq Ct^{-\frac{1}{q-1}}. \label{wad}%
\end{equation}

\end{theorem}

\begin{proof}
1) First assume $q<N.$ For any $\alpha>0,$ setting $\rho=1+\alpha,$ and
$0<\epsilon\leqq s<t<T,$ setting $\beta=1+\alpha/q,$ we obtain, from
(\ref{plm}),
\[
\frac{1}{\alpha+1}\int_{\Omega}u^{\alpha+1}(.,t)dx+\frac{1}{\beta^{q}}\int%
_{s}^{t}\int_{\Omega}\left\vert \nabla(u^{\beta})\right\vert ^{q}%
dxdt\leqq\frac{1}{\alpha+1}\int_{\Omega}u^{\alpha+1}(.,s)dx.
\]
Then $u^{\beta}(.,t)\in W_{0}^{1,q}\left(  \Omega\right)  ,$ since $u(.,t)\in
L^{\infty}(Q_{\Omega,s,\tau})\cap W_{0}^{1,1}\left(  \Omega\right)  ).$ From
the Sobolev injection of $W_{0}^{1,q}\left(  \Omega\right)  $ into
$L^{Nq/(N-q)}\left(  \Omega\right)  $,
\[
\frac{1}{\alpha+1}\int_{\Omega}u^{\alpha+1}(.,t)dx+\frac{C(N,q)}{\beta^{q}%
}\int_{s}^{t}(\int_{\Omega}u^{\beta\frac{Nq}{N-q}}(.,\sigma)dx)^{\frac{N-q}%
{N}}dt\leqq\frac{1}{\alpha+1}\int_{\Omega}u^{\alpha+1}(.,s)dx.
\]
From Lemma \ref{prod} on $\left[  \epsilon,T\right)  $ with $m=q$ and
$\theta=N/(N-q)$, we obtain estimates for $\epsilon<t<T:$
\[
\Vert u(.,t)\Vert_{L^{\infty}(\Omega)}\leqq C(t-\epsilon)^{-\sigma_{r,q}%
,N}\Vert u(.,\epsilon)\Vert_{L^{r}(\Omega)}^{\varpi_{r,q,N}},\qquad\Vert
u(.,t)\Vert_{L^{\infty}(\Omega)}\leqq C(t-\epsilon)^{-\frac{1}{q-1}}.
\]
and we deduce (\ref{yuc}) and (\ref{wad}) as $\epsilon\rightarrow0$. In the
case $q=N$ the same conclusion follows from Lemma \ref{prod} with any
$\theta>1.$ If $q>N$ we proceed as in Theorem \ref{bound} by applying Lemma
\ref{gani}.\medskip

\noindent2) Let $u$ be any weak solution of $(D_{\Omega,T}).$ Since $u\in
C(\left[  \epsilon,T\right)  ;L^{1}(\Omega))$ for $\epsilon>0,$ we find, for
any $t\in\left[  \epsilon,T\right)  ,$
\[
\Vert u(.,t)\Vert_{L^{\infty}(\Omega)}\leqq C(t-\epsilon)^{-\frac{1}{q-1}}%
\]
with $C=C(N,q),$ and deduce (\ref{wad}) for any $t\in\left(  0,T\right)  $ as
$\epsilon\rightarrow0.$
\end{proof}

\begin{remark}
In particular we find again estimate (\ref{wad}) obtained in \cite{Por} in
case $q<2$, for solutions $u$ such that $u\in C((0,T);L^{2}\left(
\Omega\right)  )\cap L^{2}((0,T);W_{0}^{1,2}\left(  \Omega\right)  ),$ and
$(u-k)^{+}$ is admissible as a test function in the equation; those conditions
imply integrability properties of $u|\nabla u|^{q}.$ Our result is valid
without any of these conditions.
\end{remark}

\subsection{Existence and uniqueness results for $q\leqq2$}

From estimate (\ref{wad}), we deduce new convergence results when $q\leqq2$:

\begin{corollary}
\label{patt}Assume $1<q\leqq2.$ Then \medskip

\noindent(i) any weak solution $u$ of problem $(D_{\Omega,T})$ satisfies $u\in
C^{2,1}\left(  Q_{\Omega,T}\right)  \cap C^{1,0}\left(  \overline{\Omega
}\times\left(  0,T\right)  \right)  ;$\medskip

\noindent(ii) for any sequence of weak solutions $\left(  u_{n}\right)  $ of
$(D_{\Omega,T}),$ one can extract a subsequence converging in $C_{{}}%
^{2,1}(Q_{\Omega,T})\cap C^{1,0}\left(  \overline{\Omega}\times\left(
0,T\right)  \right)  $ to a weak solution $u$ of $(D_{\Omega,T})$.
\end{corollary}

\begin{proof}
(i) From \cite[Theorem 2.9]{BiDao1}, any weak solution $u$ of $(D_{\Omega,T})$
such that $u\in L_{loc}^{\infty}(\left(  0,T\right)  ;L^{\infty}(\Omega))$
satisfies $u\in C^{2,1}\left(  Q_{\Omega,T}\right)  \cap C^{1,0}\left(
\overline{\Omega}\times\left(  0,T\right)  \right)  .$ And we obtain precisely
$u\in L_{loc}^{\infty}(\left(  0,T\right)  ;L^{\infty}(\Omega)),$ at Theorem
\ref{deco},3.\medskip

(ii) Moreover $(u_{n})$ is uniformly bounded in $L_{loc}^{\infty}\left(
0,T);L^{\infty}\left(  \Omega\right)  \right)  $. From \cite{BiDao1}, there
exists $\upsilon\in\left(  0,1\right)  $ such that, for any $0<s<\tau<T,$
\begin{equation}
\left\Vert u_{n}\right\Vert _{C(\overline{\Omega}\times\left[  s,\tau\right]
)}+\left\Vert \nabla u_{n}\right\Vert _{C^{\upsilon,\upsilon/2}(\overline
{\Omega}\times\left[  s,\tau\right]  )}\leqq C\Phi(\left\Vert u_{n}\right\Vert
_{L^{\infty}(Q_{\Omega,s/2,\tau})}) \label{gou}%
\end{equation}
where $C=C((N,q,\Omega,s,\tau,\upsilon),$ and $\Phi$ is an increasing
function. The conclusion follows.
\end{proof}

\begin{theorem}
\label{unic} Suppose $1<q<(N+2)/(N+1).$ For any $u_{0}\in\mathcal{M}_{b}%
^{+}(\Omega),$ problem (\ref{def}) admits a unique weak $\mathcal{M}$ solution.
\end{theorem}

\begin{proof}
From \cite[Theorem 3.2]{BeDa}, \cite{Al}, for any (possibly signed) $u_{0}%
\in\mathcal{M}_{b}(\Omega),$ problem (\ref{def}) has a unique mild
$\mathcal{M}$ solution, and it is nonnegative when $u_{0}\in\mathcal{M}%
_{b}^{+}(\Omega)$. From Lemma \ref{sim}, any weak $\mathcal{M}$ solution is a
mild $\mathcal{M}$ solution, thus uniqueness holds in this class.$\medskip$
\end{proof}

Next assume that $u_{0}\in L^{r}(\Omega)$ and consider the subcritical case
(\ref{subi}). In \cite[Theorem 3.3]{BeDa}, it is proved that there exists a
weak $L^{r}$ solution such that $u\in L_{loc}^{q}(\left[  0,T\right)
;W_{0}^{1,qr}\left(  \Omega\right)  ),$ and \textit{it is unique in this
space.} The local existence and uniqueness in an interval $\left(
0,T_{1}\right)  $ is obtained by the Banach fixed point theorem in a ball of
radius $K_{1}$ of the space
\[
X_{K_{1}}(T_{1})=\left\{  u\in C(\left(  0,T_{1}\right]  ,W_{0}^{1,qr}\left(
\Omega\right)  ):\sup_{\left(  0,t_{1}\right]  }t^{\theta}(\left\Vert
u(.,t)\right\Vert _{L^{qr}\left(  \Omega\right)  }+t^{\frac{1}{2}}\left\Vert
\nabla u(.,t)\right\Vert _{L^{qr}\left(  \Omega\right)  })<\infty\right\}
\]
where $\theta=N/2rq^{\prime},$ under the condition
\begin{equation}
\left\Vert u_{0}\right\Vert _{L^{r}\left(  \Omega\right)  }+K_{1}^{q}%
T_{1}^{\gamma}\leqq CK_{1},\qquad\text{where }\gamma=1-q(\theta+1/2)\quad
\text{and }C=C(N,q,r,\Omega). \label{fac}%
\end{equation}
We prove the uniqueness \textit{with no condition of integrability}:

\begin{theorem}
\label{nouveau}Assume that $u_{0}\in L^{r}(\Omega)$ and $1<q<(N+2r)/(N+r).$
Then problem (\ref{def}) admits a unique weak $L^{r}$ solution.
\end{theorem}

\begin{proof}
Let $\epsilon>0.$ From Theorem \ref{deco}, $u$ is bounded on $(\epsilon,T)$
for any $\epsilon\in\left(  0,T\right)  $. Then $u\in C^{2,1}(Q_{\Omega
,T})\cap C^{1,0}(\overline{\Omega}\times\left(  0,T\right)  )$ because $q<2,$
from \cite[Theorem 2.10]{BiDao1}. From (\ref{gd}), there exists a function
$D\in C((0,\infty)$ such that for any $\epsilon>0$ and for $t\geqq\epsilon$
\[
\Vert\nabla u(.,t)\Vert_{L^{\infty}(\Omega)}\leqq D(t-\epsilon).
\]
Then $\left\vert \nabla u\right\vert $ is bounded in $Q_{\epsilon,T,\Omega}$
for any $\epsilon>0.$ Thus $u\in C((0,T),W_{0}^{1,qr}\left(  \Omega\right)
).$ The problem with initial data $u(.,\epsilon)$ at time $0$ has a unique
solution $v_{\epsilon}$ such that $v_{\epsilon}\in C((0,T-\epsilon
),W_{0}^{1,qr}\left(  \Omega\right)  ),$ then $v_{\epsilon}%
(.,t)=u(.,t+\epsilon).$ Let $K_{1}$ and $T_{1}$ such that (\ref{fac}) holds.
Since $\left\Vert u(.,\epsilon)\right\Vert _{L^{r}\left(  \Omega\right)
}\leqq$ $\left\Vert u_{0}\right\Vert _{L^{r}\left(  \Omega\right)  }$, we also
have $\left\Vert v_{\epsilon}(0)\right\Vert _{L^{r}\left(  \Omega\right)
}+K_{1}^{q}T_{1}^{\gamma}\leqq CK_{1},$ thus for any $t\in(0,T_{1})$%
\[
t^{\theta}(\left\Vert v_{\epsilon}(.,t)\right\Vert _{L^{qr}\left(
\Omega\right)  }+t^{\frac{1}{2}}\left\Vert \nabla v_{\epsilon}(.,t)\right\Vert
_{L^{qr}\left(  \Omega\right)  })\leqq K_{1}.
\]
Going to the limit as $\epsilon\rightarrow0$ from the Fatou Lemma, we obtain
\[
t^{\theta}(\left\Vert u(.,t)\right\Vert _{L^{qr}\left(  \Omega\right)
}+t^{\frac{1}{2}}\left\Vert \nabla u(.,t)\right\Vert _{L^{qr}\left(
\Omega\right)  })\leqq K_{1}.
\]
Uniqueness follows in $(0,T_{1})$, and by induction on $(0,T)$.\medskip
\end{proof}

Finally we give existence results for any $u_{0}\in L^{r}(\Omega),r\geq1$,
extending the results of \cite[Theorem 3.4]{BeDa} for $u_{0}\in L^{1}%
(\Omega),$ see also \cite{Po} for more general operators. We proceed as in
Proposition \ref{pou}.

\begin{proposition}
\label{plu} Let $1<q\leq2$. For any nonnegative $u_{0}\in L^{r}(\Omega
),r\geq1,$ there exists a weak $L^{r}$ solution of problem (\ref{def}). And it
is unique if $q=2.$
\end{proposition}

\begin{proof}
(i) Case $q<2.$ Let $u_{0,n}=\min(u_{0},n).$ Then for $\rho>N(q-1)/(2-q),$
from \cite[Theorem 3.3]{BeDa}, there exists a mild solution $u_{n}$ with
initial data $u_{0,n},$ and $u_{n}\in C(\left[  0,T\right)  ;L^{\rho}%
(\Omega))\cap L^{q}((0,T);W_{0}^{1,q\rho}(\Omega)\cap C^{2,1}\left(
Q_{\Omega,T}\right)  .$ Then $u_{n}(.,t)\leq e^{t\Delta}u_{0},$ and $\left(
u_{n}\right)  $ is nondecreasing and $\left\vert \nabla u_{n}\right\vert ^{q}$
is bounded in $L_{loc}^{1}\left(  \left[  0,T\right)  ;L^{1}(\Omega)\right)  $
from (\ref{ple}). From Corollary \ref{pass}, $\left(  u_{n}\right)  $
converges in $C_{loc}^{2,1}(Q_{\Omega,T})$ to a weak solution $u$ of
(\ref{un}) in $Q_{\Omega,T}$. As a consequence, $u(.,t)\leq e^{t\Delta}u_{0}$
and $\left\vert \nabla u\right\vert ^{q}\in$ $L_{loc}^{1}\left(  \left[
0,T\right)  ;L^{1}(\Omega)\right)  .$ From \cite[Proposition 2.11]{BiDao1},
$u(.,t)$ converges weakly* to some Radon measure $\mu_{0}$ on $\Omega$. And
$e^{t\Delta}u_{0}$ converges to $u_{0}$ in $L^{r}(\Omega),$ thus $\mu_{0}\in
L_{loc}^{1}(\Omega)$ and $0\leq\mu_{0}\leq u_{0}$. Since $u_{n}\leq u,$ there
holds $u_{0,n}\leq\mu_{0},$ hence $\mu_{0}=u_{0}\in L^{r}(\Omega).$ Also there
exists a function $g\in L^{r}(\Omega)$ such that $u(.,t)\leqq g$ for small
$t.$ Then $e^{t\Delta}u_{0}-u(.,t)$ converges weakly* to $0,$ and then in
$L_{loc}^{1}(\Omega).$ Hence $u(.,t)$ converges to $u_{0}$ in $L_{loc}%
^{1}(\Omega),$ then in $L^{r}(\Omega)$ from the dominated convergence theorem.
Thus $u\in C(\left[  0,T\right)  ;L^{r}\left(  \Omega\right)  ).\medskip$

(ii) Case $q=2.$ As in \cite[Theorem 4.2]{BASoWe}, using the classical
transformation $v=1-e^{-u},$ it can be shown that there exists a unique
solution $u$ such that $u\in C(\left[  0,T\right)  ;L^{r}\left(
\Omega\right)  )\cap C^{2,1}\left(  Q_{\Omega,T}\right)  \cap C^{1}\left(
\overline{\Omega}\times\left(  0,T\right)  \right)  $. Then it is a weak
$L^{r}$ solution. Reciprocally any weak $L^{r}$ solution $u$ satisfies the
conditions above, from Corollary \ref{patt} and \cite[Theorem 2.17]{BiDao1}.
\end{proof}

\section{Regularizing effects for quasilinear Dirichlet problems\label{quas}}

Here we extend some results of section \ref{omega} to a general quasilinear
problem, where $u$ may be a signed solution. In this section, we suppose
$\Omega$ is a smooth bounded domain in $\mathbb{R}^{N}.$ $\medskip$

Let $p>1$ and $\mathrm{A}$ be a Caratheodory function on $Q_{\Omega,\infty
}\times\mathbb{R}\times\mathbb{R}^{N}$ such that for any $(u,\eta
)\in\mathbb{R}\times\mathbb{R}^{N},$ and a.e. $(x,t)\in Q_{\Omega,\infty},$
\begin{equation}
\left\vert \mathrm{A}(x,t,u,\eta)\right\vert \leqq C(\left\vert \eta
\right\vert ^{p-1}+b(x,t)),\qquad C>0,\quad b\in L^{p^{\prime}}(Q_{\Omega
,\infty}), \label{hypa}%
\end{equation}
and $\mathrm{A}$ is nonnegative operator:
\begin{equation}
\mathrm{A}(x,t,u,\eta).\eta\geqq\nu\left\vert \eta\right\vert ^{p}\qquad
\nu\geqq0, \label{coe}%
\end{equation}
with no monotonicity assumption.\medskip

Let $q>1$ and $g$ be a Caratheodory function on $Q_{\Omega,\infty}%
\times\mathbb{R}^{+}\times\mathbb{R}^{N}$, such that%
\begin{equation}
g(x,t,u,\eta)u\geqq\gamma\left\vert u\right\vert ^{\lambda+1}\left\vert
\eta\right\vert ^{q},\text{ }\qquad\text{ }\lambda\geqq0,\quad\gamma\geqq0.
\label{mino}%
\end{equation}
We say that $\mathrm{A}$ is \textbf{coercive} if\textit{ }(\ref{coe}) holds
with $\nu>0,$ and $g$ is \textbf{coercive} if (\ref{mino}) holds with
$\gamma>0.\medskip$

We consider the solutions of the Dirichlet problem%
\begin{equation}
(P_{\Omega,T})\left\{
\begin{array}
[c]{l}%
u_{t}-\operatorname{div}(\mathrm{A}(x,t,u,\nabla u))+g(x,t,u,\nabla
u)=0,\quad\text{in}\hspace{0.05in}Q_{\Omega,T},\\
u=0,\quad\text{on}\hspace{0.05in}\partial\Omega\times(0,T),\\
u(x,0)=u_{0},
\end{array}
\right.  \label{gen}%
\end{equation}
where $u_{0}\in L^{r}\left(  \Omega\right)  ,$ $r\geqq1$ or only $u_{0}%
\in\mathcal{M}_{b}(\Omega).$

\subsection{Solutions of quasilinear heat equation with $L^{1}$ data}

First consider the problem in $Q_{\Omega,s,\tau}$
\begin{equation}
\left\{
\begin{array}
[c]{l}%
u_{t}-\operatorname{div}(\mathrm{A}(x,t,u,\nabla u))=f,\quad\text{in}%
\hspace{0.05in}Q_{\Omega,s,\tau},\\
u=0,\quad\text{on}\hspace{0.05in}\partial\Omega\times(s,\tau),\\
u(x,s)=u_{s}%
\end{array}
\right.  \label{cru}%
\end{equation}
Let us recall the notion of renormalized solution introduced in \cite{BlMu}
for this problem with $L^{1}$ data, where the truncations $T_{k}$ are defined
by (\ref{tro}):

\begin{definition}
\label{renor}Let $s,\tau\in\mathbb{R}$ with $s<\tau,$ and $f\in L^{1}%
(Q_{\Omega,s,\tau})$ and $u_{s}\in L^{1}(\Omega).$ A function $u\in L^{\infty
}((s,\tau);L^{1}(\Omega))$ is a renormalized solution in $Q_{\Omega,s,\tau}$
of (\ref{cru}) if $T_{k}(u)\in L^{p}((s,\tau);W_{0}^{1,p}(\Omega))$ for any
$k\geqq0,$ and for any $S\in W^{2,\infty}(\mathbb{R})$ such that $S^{\prime}$
has a compact support,
\begin{equation}
(S(u))_{t}-\operatorname{div}(\mathrm{A}(x,t,u,\nabla u)S^{\prime
}(u))+S^{\prime\prime}(u)(\mathrm{A}(x,t,u,\nabla u).\nabla u-S^{\prime
}(u)f=0\quad\text{in}\hspace{0.05in}\mathcal{D}^{\prime}(Q_{\Omega,s,\tau}),
\label{quo}%
\end{equation}
and $u(s)=u_{s},$ and
\begin{equation}
\lim_{n\rightarrow\infty}%
{\displaystyle\int}
\int_{Q_{\Omega,s,\tau}\cap\left\{  n\leqq u\leqq n+1\right\}  }|\nabla
u|^{p}dxdt=0, \label{ruo}%
\end{equation}

\end{definition}

\begin{remark}
\label{zzz} The initial condition takes sense from \cite{BlMu}, because $S(u)$
lies in the set
\begin{equation}
E=\left\{  \varphi\in L^{p}((0,T);W_{0}^{1,p}(\Omega)):\varphi_{t}\in
L^{p^{\prime}}((0,T);W^{-1,p^{\prime}}(\Omega))+L^{1}\left(  Q_{\Omega
,T}\right)  \right\}  \label{ens}%
\end{equation}
and $E\subset C(\left[  0,T\right]  ;L^{1}(\Omega)).$ Any function $\varphi\in
L^{p}((0,T);W_{0}^{1,p}(\Omega))\cap L^{\infty}\left(  Q_{\Omega,T}\right)  $
can be chosen as a test function in equation (\ref{quo}). Moreover, from
\cite[Lemma 7.1]{DrPr}, $v=S(u)$ satisfies for any $\psi\in C^{\infty}(\left[
s,\tau\right]  \times\bar{\Omega})$ the integration formula
\begin{equation}
\int_{s}^{\tau}<v_{t},M(v)\psi>=\int_{\Omega}\mathcal{M}(v(.,\tau))\psi
(.,\tau)dx-\int_{\Omega}\mathcal{M}(v(.,s))\psi(.,s)dx-\int_{s}^{\tau}%
\int_{\Omega}\psi_{t}\mathcal{M}(v)dxdt, \label{form}%
\end{equation}
for any function $M$ continuous and piecewise $C^{1}$ such that $M(0)=0$ and
$M^{\prime}$ has a compact support, where $\mathcal{M}(r)=\int_{0}^{r}%
M(\theta)d\theta.$
\end{remark}

A main point in the sequel is the choice of test functions: here we
approximate $\left\vert u\right\vert ^{\alpha-1}u$ for $\alpha>0$ by
truncation. In the following lemma, we solve some technical difficulties
arising because the truncates are not smooth enough to apply the integration
formula, and moreover we do not assume $\alpha\geqq1.$

\begin{lemma}
\label{admi}Let $s,\tau\in\mathbb{R}$ with $s<\tau,$ and $f\in L^{1}%
(Q_{\Omega,s,\tau})$. Let $u\in C(\left[  s,\tau\right]  ;L^{1}(\Omega))$ be
any nonnegative renormalized solution in $Q_{\Omega,s,\tau}$ of (\ref{cru}),
with $u_{s}=u(.,s)$. For any $\alpha>0$ and $k>0,$ we set
\[
\mathcal{T}_{k,\alpha}(r)=\int_{0}^{r}\left\vert T_{k}(\theta)\right\vert
^{\alpha-1}T_{k}(\theta)d\theta.
\]
Then $\left\vert T_{k}(u)\right\vert ^{\alpha-1}\mathrm{A}(x,t,u,\nabla
u).\nabla(T_{k}(u))\in L^{1}(Q_{\Omega,s,\tau})$ and
\begin{align}
&  \int_{\Omega}\mathcal{T}_{k,\alpha}(u)(.,\tau))dx+\alpha%
{\displaystyle\int}
\int_{Q_{\Omega,s,\tau}}\left\vert T_{k}(u)\right\vert ^{\alpha-1}%
\mathrm{A}(x,t,u,\nabla u).\nabla(T_{k}(u))dxdt\nonumber\\
&  =\int_{\Omega}\mathcal{T}_{k,\alpha}(u)(.,s))dx+\int_{s}^{\tau}\int%
_{\Omega}f\left\vert T_{k}(u)\right\vert ^{\alpha-1}T_{k}(u)dxdt. \label{rela}%
\end{align}

\end{lemma}

\begin{proof}
Let $\alpha>0,k>0$ be fixed, and for any $n\geqq2,$ and $\theta\in\mathbb{R},$
\textbf{ }%
\[
S_{n}(\theta)=\int_{0}^{\theta}(1-\left\vert T_{1}(s-T_{n}(s)\right\vert
)ds,\qquad n\geqq2.
\]
This function, introduced in \cite{BlMu}, is still a truncation, smoother than
$T_{n+1},$ such that $0\leqq S_{n}(\theta)\theta\leqq T_{n+1}(\theta)\theta,$
supp $S_{n}^{\prime}\subset\left[  -(n+1),n+1\right]  $, $S_{n}^{\prime\prime
}=\chi_{(-n,-n-1)\cup(n,n+1)},$ and $S_{n}(T_{k}(\theta))=T_{k}(\theta)$ for
any $n>k.$ Let $\delta\in\left(  0,\min(1,k\right)  ),$ and $n>k.$ We set
\[
T_{\delta,k,\alpha}(\theta)=((T_{k}(\left\vert \theta\right\vert
)+\delta))^{\alpha}-\delta^{\alpha})\mathrm{sign}\theta,\qquad\mathcal{T}%
_{\delta,k,\alpha}(r)=\int_{0}^{r}T_{\delta,k,\alpha}(\theta)d\theta.
\]
We can take in (\ref{quo}) $S=S_{n}$ and $\varphi=T_{\delta,k,\alpha
}(u)=T_{\delta,k,\alpha}(S_{n}(u))$ as a test function. We obtain%
\begin{align*}
\int_{s}^{t}  &  <(S_{n}(u))_{t},\varphi>+\int_{s}^{t}\int_{\Omega}%
S_{n}^{\prime}(u)\mathrm{A}(x,t,u,\nabla u).\nabla\varphi dxdt\\
&  =\int_{s}^{t}\int_{\Omega}S_{n}^{\prime}(u)f\varphi dxdt-\int_{s}^{t}%
\int_{\Omega}S_{n}^{\prime\prime}(u)(\mathrm{A}(x,t,u,\nabla u).\nabla
u)\varphi dxdt.
\end{align*}
then from (\ref{form}), we deduce
\begin{align*}
&  \alpha%
{\displaystyle\int}
\int_{Q_{\Omega,s,\tau}}(T_{k}(\left\vert u\right\vert )+\delta)^{\alpha
-1}\mathrm{A}(x,t,u,\nabla u).\nabla(T_{k}(u))dxdt\\
&  =\int_{\Omega}\mathcal{T}_{\delta,k,\alpha}(S_{n}(u)(.,s))dx-\int_{\Omega
}\mathcal{T}_{\delta,k,\alpha}(S_{n}(u)(.,\tau))dx\\
&  +\int_{s}^{\tau}\int_{\Omega}S_{n}^{\prime}(u)f\varphi dxdt-\int_{s}%
^{t}\int_{\Omega}S_{n}^{\prime\prime}(u)(\mathrm{A}(x,t,u,\nabla u).\nabla
u)\varphi dxdt
\end{align*}
First we make $\delta\rightarrow0.$ Notice that $\left\vert \mathcal{T}%
_{\delta,k,\alpha}(\theta)\right\vert \leqq(k+1)^{\alpha}\left\vert
\theta\right\vert $ for any $\theta\in\mathbb{R},$ and $S_{n}(u)\in C(\left[
0,T\right]  ;L^{1}(\Omega)),$ and $S_{n}^{\prime}$ is bounded. Thus we can go
to the limit in the right hand side$.$ In the left hand side, from the
positivity of $A,$ and the Fatou Lemma we deduce that
\[
T_{k}(\left\vert u\right\vert )^{\alpha-1}\mathrm{A}(x,u,\nabla u).\nabla
T_{k}(u)\in L^{1}(Q_{\Omega,s,\tau}).
\]
Moreover we can apply dominated convergence theorem. Indeed $\mathrm{A}%
(x,u,\nabla u).\nabla T_{k}(u)$ $\in L^{1}(Q_{\Omega,s,\tau})$ from
(\ref{hypa}), since $T_{k}(u)\in L^{p}((s,\tau);W_{0}^{1,p}(\Omega)),$ and%
\[
(T_{k}(\left\vert u\right\vert )+\delta)^{\alpha-1}\mathrm{A}(x,t,u,\nabla
u).\nabla(T_{k}(u))\leqq\max(T_{k}^{\alpha-1}(\left\vert u\right\vert
),(k+1)^{\alpha-1})\mathrm{A}(x,u,\nabla u).\nabla(T_{k}(u)).
\]
Hence the same relation holds with $\delta=0$, with $T_{0,k,\alpha}%
(r)=T_{k}^{\alpha-1}(\left\vert u\right\vert )T_{k}(u):$
\begin{align*}
&  \int_{\Omega}\mathcal{T}_{k,\alpha}(S_{n}(u)(.,\tau))dx-\int_{\Omega
}\mathcal{T}_{k,\alpha}(S_{n}(u)(.,s))dx+\alpha\int_{s}^{t}\int_{\Omega}%
T_{k}^{\alpha-1}(\left\vert u\right\vert )\mathrm{A}(x,u,\nabla u).\nabla
(T_{k}(u))dxdt\\
&  =\int_{s}^{\tau}\int_{\Omega}S_{n}^{\prime}(u)fT_{0,k,\alpha}%
(u)dxdt-\int_{s}^{t}\int_{\Omega}S_{n}^{\prime\prime}(u)(\mathrm{A}%
(x,t,u,\nabla u).\nabla u)T_{0,k,,\alpha}(u)dxdt.
\end{align*}
Then we make $n\rightarrow\infty.$ Since $u\in C(\left[  0,T\right]
;L^{1}(\Omega)),$ for any $t\in\left[  s,\tau\right]  ,$ we find
\[
\lim_{n\rightarrow\infty}\int_{\Omega}\mathcal{T}_{k,\alpha}(S_{n}%
(u)(.,t))dx=\int_{\Omega}\mathcal{T}_{k,\alpha}(u(.,t))dx;
\]
moreover
\[
\lim_{n\rightarrow\infty}\int_{s}^{t}\int_{\Omega}S_{n}^{\prime\prime
}(u)(\mathrm{A}(x,t,u,\nabla u).\nabla u)\;T_{0,k,\alpha}(u)dxdt=0,
\]
from (\ref{ruo}), (\ref{hypa}). Finally%
\[
\lim_{n\rightarrow\infty}\int_{s}^{\tau}\int_{\Omega}S_{n}^{\prime
}(u)fT_{0,k,,\alpha}(u)dxdt=\int_{s}^{\tau}\int_{\Omega}fT_{0,k,,\alpha
}(u)dxdt,
\]
since $S_{n}^{\prime}(u)$ $\rightarrow1$ a.e. and is uniformly bounded. Then
(\ref{rela}) follows.
\end{proof}

\subsection{Notion of solutions of problem $(P_{\Omega,T})$}

\begin{definition}
We say that $u$ is a \textbf{renormalized solution} of problem $(P_{\Omega
,T})$ if:\medskip

\noindent(i) $u\in C((0,T);L^{1}(\Omega)),$ $T_{k}(u)\in L_{loc}%
^{p}((0,T);W_{0}^{1,p}(\Omega))$ for any $k\geqq0,$ and $g(x,u,\nabla u)\in
L_{loc}^{1}((0,T);L^{1}(\Omega));$\medskip

\noindent(ii) for any $0<s<\tau<T,$ $u$ is a renormalized solution of problem%
\[
\left\{
\begin{array}
[c]{l}%
u_{t}-\operatorname{div}(\mathrm{A}(x,t,u,\nabla u))+g(x,t,u,\nabla
u)=0,\quad\text{in}\hspace{0.05in}Q_{\Omega,s,\tau},\\
u=0,\quad\text{on}\hspace{0.05in}\partial\Omega\times(0,T),
\end{array}
\right.
\]
with initial data $u(.,s);$\medskip

\noindent(iii) for $u_{0}\in L^{r}(\Omega),$ the extension of $u$ by $u_{0}$
at time $0$ belongs to $C(\left[  0,T\right)  ;L^{r}(\Omega));$ for $u_{0}%
\in\mathcal{M}_{b}(\Omega),$ there holds%
\begin{equation}
\lim_{t\rightarrow0}\int_{\Omega}u(.,t)\psi dx=\int_{\Omega}\psi du_{0}%
,\qquad\forall\psi\in C_{b}(\Omega). \label{wou}%
\end{equation}
\medskip
\end{definition}

\begin{remark}
Recall that $\nabla u$ is defined by $\nabla u=\nabla(T_{k}(u))$ on the set
$\left\vert u\right\vert \leq k.$ The assumption on $g$ means that, for any
$0<s<\tau<T,$
\[
\int_{Q_{\Omega,s,\tau}}\left\vert g(.,u,\nabla u)\right\vert dxdt=\sum
_{k=1}^{\infty}\int_{Q_{\Omega,s,\tau}\cap\left\{  k-1\leqq\left\vert
u\right\vert \leqq k\right\}  }\left\vert g(.,u,\nabla(T_{k}(u))\right\vert
dxdt<\infty.
\]

\end{remark}

We first prove decay properties of the solutions.

\begin{theorem}
\label{dec}Let $p,q>1,$ and $A$ and $g$ satisfying (\ref{hypa}) (\ref{coe})
and (\ref{mino}).\medskip

1) Let $u_{0}\in L^{r}(\Omega),r\geqq1$ and $u$ be any renormalized solution
of $(P_{\Omega,T}).$ Then for any $t\in\left[  0,T\right)  ,$
\begin{equation}
\int_{\Omega}\left\vert u\right\vert ^{r}(.,t)dx\leqq\int_{\Omega}\left\vert
u_{0}\right\vert ^{r}dx. \label{des}%
\end{equation}
Moreover if $r>1,$ or if $g$ is coercive, then $\gamma\left\vert u\right\vert
^{\lambda+r-1}|\nabla u|^{q}+\nu\left\vert u\right\vert ^{r-2}|\nabla
u|^{p}\in L_{loc}^{1}(\left[  0,T\right)  ;L^{1}\left(  \Omega\right)  ),$
and
\begin{equation}
\int_{\Omega}\left\vert u\right\vert ^{r}(.,t)dx+r\gamma\int_{0}^{t}%
\int_{\Omega}\left\vert u\right\vert ^{\lambda+r-1}|\nabla u|^{q}%
dxdt+r(r-1)\nu\int_{0}^{t}\int_{\Omega}\left\vert u\right\vert ^{r-2}|\nabla
u|^{p}dxdt\leqq\int_{\Omega}\left\vert u_{0}\right\vert ^{r}dx. \label{bac}%
\end{equation}

2) Let $u_{0}\in\mathcal{M}_{b}^{+}(\Omega)$ and $u$ be any nonnegative
renormalized solution of $(P_{\Omega,T})$ of problem (\ref{def}). Then the
same conclusions hold as in case $u_{0}\in L^{1}(\Omega),$ where the norm
$\Vert u_{0}\Vert_{L^{1}(\Omega)}$ is replaced by $%
{\displaystyle\int_{\Omega}}
du_{0}.$
\end{theorem}

\begin{proof}
\textbf{ } 1) Let $0<s<t<T.$ Then for any $\alpha>0,$ any $k>0,$ from Lemma
\ref{admi},
\begin{align*}
&  \text{ }\int_{\Omega}\mathcal{T}_{k,\alpha}(u)(.,\tau))dx+\alpha\int%
_{s}^{t}\int_{\Omega}\left\vert T_{k}(u)\right\vert ^{\alpha-1}\mathrm{A}%
(x,t,u,\nabla u).\nabla(T_{k}(u))dxdt\\
&  =\int_{\Omega}\mathcal{T}_{k,\alpha}(u)(.,s))dx-\int_{s}^{\tau}\int%
_{\Omega}\left\vert T_{k}(u)\right\vert ^{\alpha-1}T_{k}(u)g(.,u,\nabla
u)dxdt.
\end{align*}
And $\left\vert T_{k}(u)\right\vert ^{\alpha-1}T_{k}(u)g(.,u,\nabla
u)\geq\gamma\left\vert T_{k}(u)\right\vert ^{\alpha+\lambda}\left\vert \nabla
T_{k}(u)\right\vert ^{q}$ from (\ref{mino}). Therefore $\int_{\Omega
}\mathcal{T}_{k,\alpha}(u)(.,t))$ is decreasing for any $k,\alpha>0,$ and
\begin{align}
&  \int_{\Omega}\mathcal{T}_{k,\alpha}(u)(.,\tau))dx+\gamma\int_{s}^{t}%
\int_{\Omega}\left\vert T_{k}(u)\right\vert ^{\alpha+\lambda}\left\vert \nabla
T_{k}(u)\right\vert ^{q}dxdt+\alpha\nu\int_{s}^{t}\int_{\Omega}\left\vert
T_{k}(u)\right\vert ^{\alpha-1}\left\vert \nabla T_{k}(u)\right\vert
^{p}dxdt\nonumber\\
&  \leqq\int_{\Omega}\mathcal{T}_{k,\alpha}(u)(.,s))dx. \label{arc}%
\end{align}
$\bullet$ If $r>1,$ we can take $\alpha=r-1>0$ in (\ref{arc}) and get
\begin{align}
&  \int_{\Omega}\mathcal{T}_{k,r-1}(u)(.,t))dx+\gamma\int_{s}^{t}\int_{\Omega
}\left\vert T_{k}(u)\right\vert ^{r-1+\lambda}\left\vert \nabla T_{k}%
(u)\right\vert ^{q}dxdt+\alpha\nu\int_{s}^{t}\int_{\Omega}\left\vert
T_{k}(u)\right\vert ^{r-2}\left\vert \nabla T_{k}(u)\right\vert ^{p}%
dxdt\nonumber\\
&  \leqq\int_{\Omega}\mathcal{T}_{k,r-1}(u)(.,s))dx\leqq\frac{1}{r}%
\int_{\Omega}\left\vert u\right\vert ^{r}(.,s)dx. \label{koc}%
\end{align}
Since $u\in C(\left[  0,T\right)  ;L^{r}(\Omega))$ we can go to the limit as
$k\rightarrow\infty,$ and $s\rightarrow0;$ we obtain that $\gamma\left\vert
u\right\vert ^{r-1+\lambda}\left\vert \nabla u\right\vert ^{q}$ and $\alpha
\nu\left\vert u\right\vert ^{r-2}\left\vert \nabla u\right\vert ^{p}$ belong
to $L_{loc}^{1}\left(  \left[  0,T\right)  ;L^{1}(\Omega)\right)  ;$ and for
any $t\in\left(  0,T\right)  ,$
\[
\int_{\Omega}\left\vert u\right\vert ^{r}(.,t)dx+r\gamma\int_{0}^{t}%
\int_{\Omega}\left\vert u\right\vert ^{r-1+\lambda}\left\vert \nabla
u\right\vert ^{q}dxdt+r(r-1)\nu\int_{0}^{t}\int_{\Omega}\left\vert
u\right\vert ^{r-2}\left\vert \nabla u\right\vert ^{p}dxdt\leqq\int_{\Omega
}\left\vert u_{0}\right\vert ^{r}dx.
\]
$\bullet$ If $r=1,$ we take any $\alpha>0$ in (\ref{arc}); notice that
\begin{equation}
\frac{\left\vert T_{k}(\theta)\right\vert ^{\alpha+1}}{\alpha+1}%
\leqq\mathcal{T}_{k,\alpha}(\theta)\leqq k^{\alpha}\left\vert \theta
\right\vert , \label{per}%
\end{equation}
for any $\theta>0.$ Then
\[
\int_{\Omega}\left\vert T_{k}(u)\right\vert ^{\alpha+1}(.,t))dx+(\alpha
+1)\gamma\int_{s}^{t}\int_{\Omega}\left\vert T_{k}(u)\right\vert
^{\alpha+\lambda}\left\vert \nabla T_{k}(u)\right\vert ^{q}dxdt\leqq
(\alpha+1)k^{\alpha}\int_{\Omega}\left\vert u\right\vert (.,s)dx.
\]
Going to the limit as $\alpha\rightarrow0,$ we deduce
\begin{equation}
\int_{\Omega}\left\vert T_{k}(u)\right\vert (.,t))dx+\gamma\int_{s}^{t}%
\int_{\Omega}\left\vert T_{k}(u)\right\vert ^{\lambda}\left\vert \nabla
T_{k}(u)\right\vert ^{q}dxdt\leqq\int_{\Omega}\left\vert u\right\vert (.,s)dx;
\label{rou}%
\end{equation}
and then as $s\rightarrow0$ we find
\begin{equation}
\int_{\Omega}\left\vert T_{k}(u)\right\vert (.,t))dx+\gamma\int_{s}^{t}%
\int_{\Omega}\left\vert T_{k}(u)\right\vert ^{\lambda}\left\vert \nabla
T_{k}(u)\right\vert ^{q}dxdt\leqq\int_{\Omega}\left\vert u_{0}\right\vert dx,
\label{loi}%
\end{equation}
and finally as $k\rightarrow\infty,$ we obtain that $\int_{\Omega}\left\vert
u\right\vert (.,t)dx\leqq\int_{\Omega}\left\vert u_{0}\right\vert dx.$
Moreover if $\gamma>0,$ we find
\[
\int_{\Omega}\left\vert u\right\vert (.,t)dx+\gamma\int_{0}^{t}\int_{\Omega
}\left\vert u\right\vert ^{\lambda}|\nabla u|^{q}dxdt\leqq\int_{\Omega
}\left\vert u_{0}\right\vert dx,
\]
thus (\ref{bac}) still holds with $r=1.\medskip$

\noindent2) We still find (\ref{rou}). And $\lim_{s\rightarrow0}\int_{\Omega
}u(.,s)dx=%
{\displaystyle\int_{\Omega}}
du_{0}$ from (\ref{wou}), hence the conclusion.\medskip
\end{proof}

Next we deduce $L^{\infty}$ estimates, in particular a universal one.

\begin{theorem}
\label{Bound2}Let $p,q>1,$ and $A$ and $g$ satisfying (\ref{hypa}) (\ref{coe})
and (\ref{mino}). Let $u_{0}\in L^{r}(\Omega),r\geqq1,$ and $u$ be any
renormalized solution of $(P_{\Omega,T}).$\medskip

(i) If $g$ is coercive, then%
\begin{equation}
\Vert u(.,t)\Vert_{L^{\infty}(\Omega)}\leqq\left\{
\begin{array}
[c]{ccc}%
Ct^{-\sigma_{r,q,\lambda}}\Vert u_{0}\Vert_{L^{r}(\Omega)}^{\varpi
_{r,q,N,\lambda}}, & C=C(N,q,r,\lambda,\gamma), & \text{if }q\neq N,\\
C_{\varepsilon}t^{-(1+\varepsilon)\sigma_{r,n,\lambda}}\Vert u_{0}\Vert
_{L^{r}(\Omega)}^{\varpi_{r,q,N,\lambda}}, & C_{\varepsilon}=C(N,q,r,\lambda
,\gamma,\varepsilon), & \text{if }q=N,
\end{array}
\right.  \label{xac}%
\end{equation}
where
\[
\sigma_{r,q,N,\lambda}=\frac{1}{\frac{rq}{N}+\lambda+q-1}=\frac{N}{rq}%
\varpi_{r,q,N,\lambda}.
\]
Moreover
\begin{equation}
\Vert u(.,t)\Vert_{L^{\infty}(\Omega)}\leqq Ct^{-\frac{1}{q-1+\lambda}},\qquad
C=C(N,q,\lambda,\left\vert \Omega\right\vert ). \label{niv}%
\end{equation}

(ii) If $\mathrm{A}$ is coercive and $r>(2-p)N/p,$ in particular if
$p>2N/(N+1)$, then%
\begin{equation}
\Vert u(.,t)\Vert_{L^{\infty}(\Omega)}\leqq\left\{
\begin{array}
[c]{ccc}%
Ct^{-\sigma_{r,p,N,-1}}\Vert u_{0}\Vert_{L^{r}(\Omega)}^{\varpi_{r,p,N?-1}}, &
C=C(N,p,r,\nu,\Omega), & \text{if }p\neq N,\\
C_{\varepsilon}t^{-(1+\varepsilon)\sigma_{r,N,N,-1}}\Vert u_{0}\Vert
_{L^{r}(\Omega)}^{\varpi_{r,p,N,-1}}, & C_{\varepsilon}=C(N,p,r,\nu
,\Omega,\varepsilon), & \text{if }p=N,
\end{array}
\right.  \label{yoc}%
\end{equation}
where
\[
\sigma_{r,p,N,-1}=\frac{1}{\frac{rp}{N}+p-2}=\frac{N}{rp}\varpi_{r,p,N,-1}.
\]
Moreover if $p>2,$ then
\begin{equation}
\Vert u(.,t)\Vert_{L^{\infty}(\Omega)}\leqq Ct^{-\frac{1}{p-2}},\qquad
C=C(N,p,\left\vert \Omega\right\vert ). \label{pni}%
\end{equation}
(iii) The same conclusions hold if $u$ is nonnegative and $u_{0}\in
\mathcal{M}_{b}^{+}(\Omega),$ as in case $u_{0}\in L^{1}(\Omega),$ where the
norm $\Vert u_{0}\Vert_{L^{1}(\Omega)}$ is replaced by $%
{\displaystyle\int_{\Omega}}
du_{0}.$ In particular (\ref{pni}) holds for $p>2.$
\end{theorem}

\begin{proof}
(i) Let $0<s<t<T$. Since $g$ is coercive, from Theorem \ref{dec}, for any
$\alpha\geqq0$ such that\textbf{ }$\left\vert u\right\vert ^{\alpha+1}(.,s)\in
L^{1}(\Omega),$ there holds
\[
\int_{\Omega}\left\vert u\right\vert ^{\alpha+1}(.,t)dx+(\alpha+1)\gamma
\int_{s}^{t}\int_{\Omega}\left\vert u\right\vert ^{\lambda+\alpha}|\nabla
u|^{q}dxdt\leqq\int_{\Omega}\left\vert u\right\vert ^{\alpha+1}(.,s)dx,
\]
from (\ref{bac}); in particular
\[
\int_{\Omega}\left\vert T_{k}(u)\right\vert ^{\alpha+1}(.,t)dx+(\alpha
+1)\gamma\int_{s}^{t}\int_{\Omega}(T_{k}(u))^{\lambda+\alpha}\left\vert \nabla
T_{k}(u)\right\vert ^{q}dxdt\leqq\int_{\Omega}\left\vert u\right\vert
^{\alpha+1}(.,s)dx.
\]
And $\left\vert u\right\vert ^{\lambda+\alpha}|\nabla u|^{q}=|\nabla
(\left\vert u\right\vert ^{\beta-1}u)|^{q}$ with $\beta=1+(\alpha
+\lambda)/q\geqq1.$ Then $|\nabla((\left\vert u\right\vert ^{\beta
-1}u)(.,t))|$, and also $|\nabla((\left\vert T_{k}(u)\right\vert ^{\beta
-1}T_{k}(u))(.,t))|$ belong to $L^{q}(\Omega)$ for almost any $t\in\left(
0,T\right)  .$ Since $\left\vert T_{k}(u)\right\vert ^{\beta-1}T_{k}%
(u)(.,t)\in L^{\infty}(\Omega),$ it follows that $\left\vert T_{k}%
(u)\right\vert ^{\beta-1}T_{k}(u)(.,t)\in W^{1,q}\left(  \Omega\right)  $.
Moreover $T_{k}(u)(.,t)\in W_{0}^{1,p}\left(  \Omega\right)  ),$ hence
$\left\vert T_{k}(u)\right\vert ^{\beta-1}T_{k}(u)(.,t)\in W_{0}^{1,q}\left(
\Omega\right)  .$ If $q<N,$ we deduce
\[
\int_{\Omega}\left\vert T_{k}(u)\right\vert ^{\alpha+1}(.,t)dx+\gamma
\frac{(\alpha+1)C(N,q)}{\beta^{q}}\int_{s}^{t}(\int_{\Omega}\left\vert
T_{k}(u)\right\vert ^{\beta\frac{Nq}{N-q}}(.,\sigma)dx)^{\frac{N-q}{N}}%
d\sigma\leqq\int_{\Omega}\left\vert u\right\vert ^{\alpha+1}(.,s)dx.
\]
Going to the limit as $k\rightarrow\infty,$ we find
\[
\int_{\Omega}\left\vert u\right\vert ^{\alpha+1}(.,t)dx+\gamma\frac
{(\alpha+1)C(N,q)}{\beta^{q}}\int_{s}^{t}(\int_{\Omega}\left\vert u\right\vert
^{\beta\frac{Nq}{N-q}}(.,\sigma)dx)^{\frac{N-q}{N}}d\sigma\leqq\frac{1}%
{\alpha+1}\int_{\Omega}\left\vert u\right\vert ^{\alpha+1}(.,s)dx.
\]
Then we can apply Lemma \ref{prod} on $\left[  \epsilon,T\right)  $, with
$m=q$ and $\theta=N/(N-q)$; indeed (\ref{ele}) is satisfied, since
$\lambda\geqq0;$ we deduce the estimate for $\left[  \epsilon,T\right)  ,$
\[
\Vert u(.,t)\Vert_{L^{\infty}(\Omega)}\leqq C(t-\epsilon)^{-\sigma
_{r,q,N,\lambda}}\Vert u(.,\epsilon)\Vert_{L^{r}(\Omega)}^{\varpi
_{r,q,N,\lambda}},
\]
with $C=C(N,q,r,\lambda,\gamma,\Omega).$ Going to the limit as $\epsilon
\rightarrow0,$ we get (\ref{niv}), and (\ref{xac}) for $u_{0}\in L^{r}%
(\Omega),$ and the analogous when $u_{0}\in\mathcal{M}_{b}^{+}(\Omega).$ In
case $q\geqq N$ we proceed as in Theorem \ref{deco}.\medskip

(ii) Assume that $\mathrm{A}$ is coercive. Then for any $\alpha>0$,
\begin{equation}
\int_{\Omega}\mathcal{T}_{k,\alpha}(u)(.,t))dx+\alpha\nu\int_{s}^{t}%
\int_{\Omega}\left\vert T_{k}(u)\right\vert ^{\alpha-1}\left\vert \nabla
T_{k}(u)\right\vert ^{p}dxdt\leqq\int_{\Omega}\mathcal{T}_{k,\alpha
}(u)(.,s))dx,\nonumber
\end{equation}
from (\ref{arc}). First assume $p<N.$ From the Sobolev injection of
$W_{0}^{1,p}\left(  \Omega\right)  $ into $L^{Np/(N-p)}\left(  \Omega\right)
$, we deduce
\[
\frac{1}{\alpha+1}\int_{\Omega}\left\vert u\right\vert ^{\alpha+1}%
(.,t)dx+\alpha\nu\frac{C(N,p)}{k^{p}}\int_{s}^{t}(\int_{\Omega}^{k\frac
{Np}{N-p}}\left\vert u\right\vert (.,\sigma)dx)^{\frac{N-p}{N}}dt\leqq\frac
{1}{\alpha+1}\int_{\Omega}\left\vert u\right\vert ^{\alpha+1}(.,s)dx,
\]
with $k=1+(\alpha-1)/p.$\medskip\ 

First suppose $r>1;$ then we start from $\alpha_{0}=r-1>0$, and we can apply
Lemma \ref{prod} with $C_{0}=(r-1)\nu C(N,p),$ $m=p,$ $\theta=N/(N-p)$ and
$\lambda=-1;$ indeed (\ref{ele}) is satisfied, since $r>N(2-p)/p$.\medskip

Next suppose $r=1.$ Then $1>(2-p)N/p,$ thus $p-1+p/N>1.$ For any $\alpha>0,$
\[
\int_{\Omega}\left\vert T_{k}(u)\right\vert ^{\alpha+1}(.,t))dx+\alpha
(\alpha+1)\nu\int_{s}^{t}\int_{\Omega}\left\vert T_{k}(u)\right\vert
^{\alpha-1}\left\vert \nabla T_{k}(u)\right\vert ^{p}dxdt\leqq(\alpha
+1)k^{\alpha}\int_{\Omega}\left\vert u\right\vert (.,s)dx.
\]
Taking $\alpha=1,$ we get from (\ref{des}),
\[
\nu\int_{s}^{t}\int_{\Omega}\left\vert \nabla T_{k}(u)\right\vert
^{p}dxdt\leqq k\int_{\Omega}\left\vert u\right\vert (.,s)dx\leqq k\int%
_{\Omega}\left\vert u_{0}\right\vert dx.
\]
And from (\ref{des}), $u\in L^{\infty}\left(  (s,T);L^{1}\left(
\Omega\right)  \right)  .$ From standard estimates, there holds $u\in L^{\rho
}(Q_{\Omega,s,t})$ for any $\rho\in\left(  1,p-1+p/N\right)  ,$ see
\cite{BoGa}. Then $\left\vert u\right\vert ^{\rho}(.,t)$ $\in L^{1}\left(
\Omega\right)  $ for almost any $t\in\left(  0,T\right)  .$ Hence we can apply
Lemma \ref{prod} on $\left[  \epsilon,T\right)  $ for $\epsilon>0,$ with the
same parameters, after fixing such a $\rho=\rho_{p,N}$ such that $\rho
N(2-p)/p<1.$ We obtain that
\[
\Vert u(.,t)\Vert_{L^{\infty}(\Omega)}\leqq C(t-\epsilon)^{-\sigma_{1,p,-1}%
}\Vert u(.,\epsilon)\Vert_{L^{1}(\Omega)}^{\varpi_{1,p,-1}},
\]
where $C=C(N,p\rho_{p,N})=C(N,p);$ finally we go to the limit as
$\epsilon\rightarrow0$ because $u\in C(\left[  0,T\right]  ;L^{1}(\Omega)).$
Estimate (\ref{pni}) follows, since $-1+p-1>0.\medskip$

If $p=N,$ we proceed as above, applying Lemma \ref{prod} with $m=N,$
$\lambda=-1$ and $\theta>1$ arbitrary. Next assume $p>N.$ In case $r>1,$ there
holds, for any $t\in(0,T),$
\[
r(r-1)\frac{\nu}{\kappa^{p}}\int_{0}^{t}\int_{\mathbb{R}^{N}}|\nabla
(\left\vert u\right\vert ^{\kappa})|^{p}dxdt\leqq\int_{\mathbb{R}^{N}%
}\left\vert u_{0}\right\vert ^{r}dx,
\]
where $\kappa=1+(r-2)/p>0.$ From Lemma \ref{gani}, applied to $v=\left\vert
u\right\vert ^{\kappa},$ with $m=p,$ $1/k=1+r(p-N)/Np\kappa,$ we obtain
\[
\left\Vert u(.,t)\right\Vert _{L^{\infty}\left(  \mathbb{R}\right)  }%
^{\frac{\kappa p}{k}}\leqq C\left\Vert u(.,t)\right\Vert _{L^{r}\left(
\mathbb{R}\right)  }^{\frac{\kappa((1-k)p}{k}}\int_{\mathbb{R}}|\nabla
(\left\vert u\right\vert ^{\kappa})|^{p}dxdt;
\]
and by integration, with a new constant $C=C(N,p,r,\nu),$
\[
t\left\Vert u(.,t)\right\Vert _{L^{\infty}\left(  \mathbb{R}\right)  }%
^{\frac{\kappa p}{k}}\leqq C\left\Vert u_{0}\right\Vert _{L^{r}\left(
\mathbb{R}^{N}\right)  }^{r+\frac{\kappa((1-k)p}{k}},
\]
which is precisely (\ref{yoc}). In case $r=1,$ we choose $\rho=p\in\left(
1,p-1+p/N\right)  ,$ and obtain from above, for any $0<\epsilon<s<t<T,$
\[
\Vert u(.,t)\Vert_{L^{\infty}(\mathbb{R})}\leqq C(t-s)^{-\sigma_{1,p,p,-1}%
}\Vert u(.,s)\Vert_{L^{\rho}(\mathbb{R})}^{\varpi_{1,p,N,-1}}\leqq
C(t-\epsilon)^{-\sigma_{1,p,p,-1}}\Vert u(.,s)\Vert_{L^{\infty}(\mathbb{R}%
)}^{\frac{\varpi_{1,p,N,-1}}{p^{\prime}}}\Vert u_{0}\Vert_{L^{1}(\mathbb{R}%
)}^{\frac{\varpi_{1,p,N,-1}}{p}},
\]
where $C=C(N,p,\nu).$ From Lemma \ref{elem}, we deduce precisely
\[
\Vert u(.,t)\Vert_{L^{\infty}(\mathbb{R})}\leqq C(t-\epsilon)^{-\sigma
_{1,p,p-1}}\Vert u_{0}\Vert_{L^{1}(\mathbb{R})}^{\varpi_{1,p,N,-1}},
\]
and we conclude as $\epsilon\rightarrow0.$

(iii) We obtain the estimates on $\left(  \epsilon,T\right)  $ as above and go
to the limit as $\epsilon\rightarrow0$.
\end{proof}

\begin{remark}
Our results apply in particular to the problem
\[
\left\{
\begin{array}
[c]{l}%
u_{t}-\operatorname{div}(\mathrm{A}(x,t,u,\nabla u))=0,\quad\text{in}%
\hspace{0.05in}Q_{\Omega,T},\\
u=0,\quad\text{on}\hspace{0.05in}\partial\Omega\times(0,T),\\
u(x,0)=u_{0}%
\end{array}
\right.
\]
Thus we find again and improve the estimates of \cite[Theorem 5.3]{Por}, with
less regularity on the solutions: those estimates were proved for solutions
$u\in C(\left[  0,T\right)  ;L^{r}(\Omega))$ such that $u\in L^{p}(\left(
0,T\right)  ;W_{0}^{1,p}(\Omega))\cap C(\left[  0,T\right)  ;L^{2}(\Omega))$.
The notion of renormalized solutions, equivalent to the notion of entropy
solutions of \cite{Pr} (see \cite{DrPr}), is weaker. Moreover our results in
case $p>N$ are optimal.\bigskip
\end{remark}

\begin{remark}
The extension of results of section \ref{Rn} to the case of equation of type
(\ref{fru}) in the case $\Omega=\mathbb{R}^{N}$ will be treated a further article.
\end{remark}

\section{Appendix}

\begin{proof}
[Proof of Lemma \ref{lem}](i) Let $u$ be a mild $\mathcal{M}$ solution. Then
clearly (\ref{ha}) holds. Moreover for any $\psi\in C_{0}\left(
\mathbb{R}^{N}\right)  ,$ from the assumption on the gradient,
\[
<e^{t\Delta}u_{0},\psi>=<u_{0},e^{t\Delta}\psi>=\int_{\mathbb{R}^{n}%
}e^{t\Delta}\psi du_{0}=\int_{\mathbb{R}^{N}}(u(.,t)+\int_{0}^{t}%
e^{(t-s)\Delta}|\nabla u(.,s)|^{q}ds)\psi dx
\]
The relation extends to any $\varphi\in C_{b}\left(  \mathbb{R}^{N}\right)  :$
we can assume that $\varphi\geqq0;$ from the Beppo-Levi theorem,
\begin{align*}
\int_{\mathbb{R}^{N}}e^{t\Delta}\varphi du_{0}  &  =\int_{\mathbb{R}^{N}%
}u(.,t)\varphi dx+\int_{\mathbb{R}^{N}}(\int_{0}^{t}e^{(t-s)\Delta}|\nabla
u(.,s)|^{q}ds)\varphi dx\\
&  =\int_{\mathbb{R}^{N}}u(.,t)\varphi dx+\int_{0}^{t}\int_{\mathbb{R}^{N}%
}|\nabla u|^{q}\varphi dxds,
\end{align*}
since the measure is bounded. From the integrability of the gradient and the
dominated convergence theorem in $L^{1}(\mathbb{R}^{N},du_{0}),$ we deduce
\[
\lim_{t\rightarrow0}\int_{0}^{t}\int_{\mathbb{R}^{N}}|\nabla u|^{q}\varphi
dxds=0,\qquad\lim_{t\rightarrow0}\int_{\mathbb{R}^{N}}e^{t\Delta}\varphi
d\mu_{0}=\int_{\mathbb{R}^{N}}\varphi d\mu_{0},
\]
since $\left\Vert e^{t\Delta}\varphi\right\Vert _{L^{\infty}\left(
\mathbb{R}^{N}\right)  }\leq\left\Vert \varphi\right\Vert _{L^{\infty}\left(
\mathbb{R}^{N}\right)  }$ and $e^{t\Delta}\varphi$ converges to $\varphi$
everywhere as $t\rightarrow0;$ thus (\ref{hb}) holds.\medskip

(ii) Let $u$ be a weak semi-group solution. Then obviously $u\in
C_{b}((0,T);L^{1}\left(  \mathbb{R}^{N}\right)  ).$ As $\epsilon\rightarrow0,$
we have
\[
\lim_{\epsilon\rightarrow0}\int_{\epsilon}^{t}e^{(t-s)\Delta}|\nabla
u(.,s)|^{q}ds=\int_{0}^{t}e^{(t-s)\Delta}|\nabla u(.,s)|^{q}ds\qquad\text{in
}L^{1}(\mathbb{R}^{N}).
\]
Then
\[
\lim_{\epsilon\rightarrow0}e^{(t-\epsilon)\Delta}u(.,\epsilon)=u(.,t)+\int%
_{0}^{t}e^{(t-s)\Delta}|\nabla u(.,s)|^{q}ds\qquad\text{in }L^{1}%
(\mathbb{R}^{N}).
\]
Moreover (\ref{hb}) entails that that $u(.,\epsilon)\rightarrow u_{0}$ in
$\mathcal{S}^{\prime}(\mathbb{R}^{N})$ and
\begin{equation}
\lim_{\epsilon\rightarrow0}e^{(t-\epsilon)\Delta}u(.,\epsilon)=e^{t\Delta
}u_{0}\qquad\text{in }\mathcal{S}^{\prime}(\mathbb{R}^{N}); \label{chic}%
\end{equation}
indeed for any $\varphi\in\mathcal{S}(\mathbb{R}^{N}),$
\begin{align*}
\left\vert <e^{(t-\epsilon)\Delta}u(.,\epsilon)-e^{t\Delta}u_{0}%
,\varphi>\right\vert  &  \leq\left\vert <e^{t\Delta}(u(.,\epsilon
)-u_{0}(.),\varphi>\right\vert +\left\vert \int_{\mathbb{R}^{N}}%
(u(x,\epsilon)((e^{(t-\epsilon)\Delta}-e^{t\Delta})\varphi)(x)dx\right\vert \\
&  \leq\left\vert <e^{t\Delta}(u(.,\epsilon)-u_{0}(.),\varphi>\right\vert \\
&  +\left\Vert u(.,\epsilon)\right\Vert _{L^{1}(\mathbb{R}^{N})}\left\Vert
(e^{(t-\epsilon)\Delta}-e^{t\Delta})\varphi\right\Vert _{L^{\infty}%
(\mathbb{R}^{N})})
\end{align*}
and $e^{t\Delta}is$ continuous on $\mathcal{S}(\mathbb{R}^{N})$. Hence, for
any $\varphi\in\mathcal{S}(\mathbb{R}^{N}),$ we get
\[
<e^{t\Delta}u_{0},\varphi>=\int_{\mathbb{R}^{n}}u(.,t)\varphi dx+\int%
_{\mathbb{R}^{n}}(\int_{0}^{t}e^{(t-s)\Delta}|\nabla u(.,s)|^{q}ds)\varphi dx
\]
which extends to any $\varphi\in C_{0}(\mathbb{R}^{N})$ by density. Thus
(\ref{hc}) follows. \medskip
\end{proof}

\begin{acknowledgement}
We thank Professor F. Weissler for helpfull discussions during the preparation
of this article.
\end{acknowledgement}

\end{document}